%% file: DMS.tex
\documentclass[11pt]{article}
\usepackage[english]{babel}
\usepackage[latin1]{inputenc}
\usepackage{vmargin}       % Régler la taille de la feuille.
\usepackage{ifthen}         % Faire des tests if/then/else.
\usepackage{xspace}         % Ajuster l'espace après des mots.
\usepackage{array}
\usepackage{graphicx}

\usepackage{epsfig}
\usepackage{graphics}
\usepackage{color}

\usepackage{indentfirst}        % Ce package indente les débuts de paragraphe
\usepackage{enumerate}          % pour les moches enumerate
\usepackage{url}

\usepackage{amsmath}    % Les symboles les plus fréquents.
\usepackage{amssymb}    % Des symboles.
\usepackage{amsfonts}   % Des fontes, eg pour \mathbb.
\usepackage{mathrsfs}   % Des lettres majuscules cursives (\mathscr).
\usepackage{amsthm}
\usepackage{amstext}
\usepackage{verbatim}   % Pour les codes sources en informatique.
\usepackage{amsbsy}
\usepackage{stmaryrd}   % Plus de symboles (*comme les segment d'entiers)

\newtheorem{theo}{Theorem}
\newtheorem{theorem}[theo]{Theorem}
\newtheorem{lemma}[theo]{Lemma}
\newtheorem{proposition}[theo]{Proposition}
\newtheorem{corollary}[theo]{Corollary}

\theoremstyle{definition}
\newtheorem{remark}[theo]{Remark}
\newtheorem{assumption}[theo]{Assumption}

\newtheorem{example}[theo]{Example}
\newtheorem{definition}[theo]{Definition}

\renewcommand{\qed}{\ensuremath{\hfill \quad \rule{1.2ex}{1.2ex}}
\vspace{\baselineskip}}

\renewenvironment{proof}
{\vspace{0.5ex} \noindent  {\bfseries Proof.~}}
{\hfill \qed\vspace{1ex}}

\newcommand{\theoref}[1]{Theorem~\ref{#1}}
\newcommand{\thref}[1]{Theorem~\ref{#1}}
\newcommand{\lemref}[1]{Lemma~\ref{#1}}
\newcommand{\corref}[1]{Corollary~\ref{#1}}
\newcommand{\propref}[1]{Proposition~\ref{#1}}

\newcommand{\defref}[1]{Definition~\ref{#1}}

\newcommand{\secref}[1]{Section~\ref{#1}}
\newcommand{\ssecref}[1]{Subsection~\ref{#1}}

\setpapersize{custom}{20.5cm}{28cm}
\setmarginsrb{20mm}{10mm}{20mm}{10mm}{10mm}{6mm}{0mm}{10mm}

\newcommand{\RR}{\mathbf{R}} % warning bf
\newcommand{\NN}{\mathbb{N}}

\newcommand{\Mm}{\mathcal{M}} \newcommand{\Dd}{\mathcal{D}}
\newcommand{\Nn}{\mathcal{N}}

\newcommand{\al}{\alpha}

\newcommand{\la}{\lambda}

\newcommand{\de}{\delta}
\newcommand{\si}{\sigma}

\newcommand{\ph}{\varphi}

\newcommand{\De}{\Delta}

\numberwithin{equation}{section}
\numberwithin{theo}{section}

\newcommand{\qqandqq}{\qquad\text{and}\qquad}

\newcommand{\qandq}{\quad\text{and}\quad}

\newcommand{\kmin}{{\kappa_{\ast}}}

\newcommand{\bx}{\bar x}
\newcommand{\bX}{\bar X}

\newcommand{\On}{{\bf O}^n}
\newcommand{\Sn}{{\bf S}^n}

\newcommand{\Perm}{{\bf \Sigma}^n}
\newcommand{\trans}[1]{\ensuremath{{#1}^{\!\top}}}

\newcommand{\Diag}{\ensuremath{\mathrm{Diag\xspace}}}
\newcommand{\diag}{\ensuremath{\mathrm{diag\xspace}}}

\newcommand{\norm}[1]{\ensuremath{\Arrowvert #1 \Arrowvert}}

\newcommand{\ra}{\rightarrow}

\DeclareMathOperator*{\id}{id}
\DeclareMathOperator*{\reduc}{red}

\newcommand{\fonction}[5]{%
      \ensuremath{#1\colon
      \left\{\hskip -1.5 mm
      \begin{array}{c@{\ }c@{\ }l}
      \medskip #2 & \longrightarrow & #3 \\
      #4 & \longmapsto & #5 \\
      \end{array}
      \right .
      }}

\newcommand{\tr}{\mathrm{tr\,}}

\newcommand{\R}{\mathbf{R}}
\newcommand{\N}{\mathbb{N}}

\title{{\LARGE Locally symmetric submanifolds lift to spectral manifolds}}

\medskip

\author{{\large {\sc Aris DANIILIDIS, J\'er\^ome MALICK, Hristo SENDOV}}}

\begin{document}
\allowdisplaybreaks

\maketitle

\medskip
\noindent \textbf{Abstract.} In this work we prove that every
locally symmetric smooth submanifold $\mathcal{M}$ of $\RR^n$ gives
rise to a naturally defined smooth submanifold of the space of
$n\times n$ symmetric matrices, called spectral manifold, consisting
of all matrices whose ordered vector of eigenvalues belongs to
$\mathcal{M}$. We also present an explicit formula for the dimension
of the spectral manifold in terms of the dimension and the intrinsic
properties of $\mathcal{M}$.

\bigskip



\bigskip
\tableofcontents

%********************************************************************
%********************************************************************
\section{Introduction}
%********************************************************************
%********************************************************************

Denoting by $\Sn$ the Euclidean space of $n \times n$ symmetric
matrices with inner product $\langle X, Y\rangle = \tr(XY)$, we
consider the \textit{spectral mapping} $\la$, that is, a function
from the space $\Sn$ to $\RR^n$, which associates to $X\in \Sn$ the
vector $\la(X)$ of its eigenvalues. More precisely, for a matrix
$X\in \Sn$, the vector
$\lambda(X)=(\lambda_{1}(X),\ldots,\lambda_{n}(X))$ consists of the
eigenvalues of $X$ counted with multiplicities and ordered in a
non-increasing way:
$$
\lambda_{1}(X) \ge \lambda_2(X) \ge \cdots \ge \lambda_n(X).
$$
The object of study in this paper are \textit{spectral sets}, that
is, subsets of $\Sn$ stable under orthogonal similarity
transformations: a subset ${\bf M}\subset \Sn$ is a spectral set if
for all $X\in {\bf M}$ and $U\in \On$ we have $\trans{U}XU\in {\bf
M}$, where $\On$ is the set of $n\times n$ orthogonal matrices. In
other words, if a matrix $X$ lies in a spectral set ${\bf M}\subset
\Sn$, then so does its orbit under the natural action of the group
of $\On$
\[
  \On.X = \{\trans{U}XU : \ U \in \On\}.
\]

The spectral sets are entirely defined by their eigenvalues,
and can be equivalently defined as inverse images of subsets of
$\RR^n$ by the spectral mapping $\la$, that is,
$$
\la^{-1}(M):=\{X \in \Sn : \la(X) \in M\},\quad\text{for
some }\,M\subset\RR^n.
$$
For example, if $M$ is the Euclidean unit ball $B(0,1)$ of $\RR^n$,
then $\la^{-1}(M)$ is the Euclidean unit ball of $\Sn$ as well.
A spectral set can be written as union of orbits:
\begin{equation}\label{jm-adds}
  \la^{-1}(M)= \bigcup_{x\in M} \On.\Diag(x),
\end{equation}
where $\Diag(x)$ denotes the diagonal matrix with the vector $x\in
\RR^n$ on the main diagonal.

\smallskip

In this context, a general question arises:
What properties on $M$ remain true on the
corresponding spectral set $\la^{-1}(M)$?

\smallskip

In the sequel we often refer to this as the {\it transfer
principle}. The spectral mapping $\la$ has nice geometrical
properties, but it may behave very badly as far as, for example,
differentiability is concerned. This imposes intrinsic difficulties
for the formulation of a generic transfer principle. Invariance
properties of $M$ under permutations often correct such bad behavior
and allow us to deduce transfer properties between the sets $M$ and
$\la^{-1}(M)$. A set $M \subset \RR^n$ is $\textit{symmetric}$ if
$\sigma M = M$ for all permutations $\sigma$ on $n$ elements, where
the permutation $\sigma$ permutes the coordinates of vectors in
$\RR^n$ in the natural way. Thus, if the set $M \subset \RR^n$ is
symmetric, then properties such as closedness and convexity are
transferred between $M$ and $\la^{-1}(M)$. Namely, $M$ is closed
(respectively, convex \cite{Lewis:1996}, prox-regular \cite{DLMS08})
if and only if $\la^{-1}(M)$ is closed (respectively, convex,
prox-regular). The next result is another interesting example of
such a transfer.

\begin{proposition}[Transferring algebraicity]
\label{Trans-alg}
Let $\Mm \subset \RR^n$ be a symmetric algebraic variety. Then,
$\la^{-1}(\Mm)$ is an algebraic variety of $\Sn$.
\end{proposition}

\begin{proof}
  Let $p$ be any polynomial equation of $\Mm$, that is $p(x)=0$ if and
  only if $x\in\Mm$. Define the symmetric polynomial
  $q(x):=\sum_{\sigma} p^2(\si x)$.  Notice that $q$ is again a
  polynomial equation of $\Mm$ and $q(\lambda(X))$ is an equation of
  $\la^{-1}(\Mm)$. We just have to prove that $q \circ \lambda$ is a
  polynomial in the entries of $X$.  It is known that $q$ can be written as a
  polynomial of the elementary symmetric polynomials
  $p_1,p_2,\ldots,p_n$. Each $p_j(\lambda(X))$, up to a sign, is  a
  coefficient of the characteristic polynomial of $X$, thus it is a polynomial in
  $X$. Thus we can complete the proof.
 \end{proof}

Concurrently, similar transfer properties hold for \textit{spectral
functions}, that is, functions $F\colon\Sn\ra\RR^n$ which are
constant on the orbits $\On.X$ or equivalently, functions $F$ that
can be written as $F = f\circ \la$ with $f\colon\RR^n\ra\RR$ being
symmetric, that is invariant under any permutation of entries of
$x$.  Since $f$ is symmetric, closedness and convexity are
transferred between $f$ and $F$ (see \cite{Lewis:1996} for details).
More surprisingly, some differentiability properties are also
transferred (see \cite{Lewis:1994b}, \cite{LewisSendov:2000a} and
\cite{sendov2007}). As established recently in \cite{DLMS08}, the
same happens for an important property of variational analysis, the
so-called prox-regularity (we refer to
\cite{poliquin-rockafellar-1996} for the definition).

\smallskip

In this work, we study the transfer of differentiable structure of a
submanifold $\Mm$ of $\RR^n$ to the corresponding spectral set. This
gives rise to an orbit-closed set $\lambda^{-1}(\Mm)$ of $\Sn$,
which, in case it is a manifold, will be called {\it spectral
manifold}. Such spectral manifolds often appear in engineering
sciences, often as constraints in feasibility problems (for example,
in the design of tight frames \cite{tropp-2005} in image processing
or in the design of low-rank controller \cite{helmke-2005} in
control). Given a manifold $\Mm$, the answer, however, to the
question of whether or not the spectral set $\lambda^{-1}(\Mm)$ is a
manifold of $\Sn$ is not direct: indeed, a careful glance at
\eqref{jm-adds} reveals that $\On.\Diag(x)$ has a natural (quotient)
manifold structure (we detail this in \secref{ssec-loc-sym}), but
the question is how the different strata combine as $x$ moves along
$\Mm$.

\smallskip

For functions, transferring local properties as differentiability
requires some symmetry, albeit not with respect to all permutations:
it turns out that many properties still hold under {\it local
symmetry}, that is, invariance under permutations that preserve
balls centered at the point of interest. We define precisely these
permutations in \secref{ssec:si}, and we state in
\thref{thm-26Nov2007-1} that the differentiability of spectral
functions is valid under this local invariance.
%Establishing that
%local properties on functions yield the transfer of
%differentiability to spectral functions is a secondary contribution
%of this paper.

\smallskip

The main goal here is to prove that local smoothness of $\Mm$ is
transferred to the spectral set $\la^{-1}(\Mm)$, whenever $\Mm$ is
locally symmetric. More precisely, our aim here~is

\begin{itemize}
\item to prove that every connected $C^k$ locally symmetric manifold
  $\mathcal{M}$ of $\mathbb{R}^{n}$ is {\it lifted} to a connected
  $C^k$ manifold $\lambda^{-1}(\mathcal{M})$ of ${\bf S}^{n}$, for $k
  \in \{2,\infty, \omega\}$;
\item to derive a formula for the dimension of
  $\lambda^{-1}(\mathcal{M})$ in terms of the dimension of
  $\mathcal{M}$ and some characteristic properties of $\Mm$.
\end{itemize}

This is eventually accomplished with \theoref{main-thm-4Dec2007-3}.
To get this result, we use extensively differential properties of
spectral functions and geometric properties of locally symmetric
manifolds. Roughly speaking, given a manifold $\Mm$ which is locally
symmetric around $\bar x$, the idea of the proof~is:

\begin{enumerate}
\item to exhibit a simple locally symmetric affine manifold $\Dd$, see
  \eqref{Dcal}, which will be used as a domain for a locally symmetric
  local equation for the manifold $\Mm$ around $\bar x$
  (Theorem~\ref{thm-40});
\item to show that $\la^{-1}(\Dd)$ is a smooth manifold
  (\thref{main-res-simple-form-2}) and use it as a domain for a local
  equation of $\la^{-1}(\Mm)$ (see definition in
  \eqref{aris-D11}), in order to establish that the latter is a
  manifold (Theorem~\ref{main-thm-4Dec2007-3}).
\end{enumerate}

The paper is organized as follows. We start with grinding our tools:
in \secref{sec:permut} we recall basic properties of permutations
and define a stratification of $\RR^n$ naturally associated to
them which will be used to study properties of locally
symmetric manifolds in \secref{sec:manifold}. Then, in
\secref{sec:spectral} we establish the transfer of the
differentiable structure from locally symmetric subsets of $\RR^n$ to
spectral sets of $\Sn$.

%********************************************************************
%********************************************************************
\section{Preliminaries on permutations}\label{sec:permut}
%********************************************************************
%********************************************************************

This section gathers several basic results about permutations that
are used extensively later. In particular, after defining order
relations on the group of permutations in \ssecref{ssec:si} and the
associated stratification of $\RR^n$ in \ssecref{ssec:strat}, we
introduce the subgroup of permutations that preserve balls centered
at a given point.

%********************************************************************
\subsection{Permutations and partitions}\label{ssec:si}
%********************************************************************

Denote by $\Sigma^{n}$ the group of permutations over
$\mathbb{N}_{n}:=\{1,\ldots,n\}$. This group has a natural action on
$\RR^n$ defined for $x=(x_{1},\ldots,x_{n})$ by
\begin{equation}
\label{eqn-sigma}
\sigma x:=(x_{\sigma^{-1}(1)},\ldots,x_{\sigma^{-1}(n)}).
\end{equation}

Given a permutation $\si\in\Sigma^n$, we define its support
$\mathrm{supp}(\sigma)\subset\mathrm{N}_n$ as the set of indices
$i\in\mathbb{N}_n$ that do not remain fixed under $\sigma$. Further,
we denote by $\RR^n_{\ge}$ the closed convex cone of all vectors
$x\in \RR^n$ with $x_1 \ge x_2 \ge \cdots \ge x_n$.

Before we proceed, let us recall some basic facts on permutations. A
cycle of length $k\in\N_n$ is a permutation
$\sigma\in\Sigma^{n}$ such that for $k$ distinct
elements $i_{1} ,\ldots,i_{k}$ in $\mathbb{N}_{n}$ we have
$\mathrm{supp}(\sigma)=\{i_{1},\ldots,i_{k}\},$ and
$\sigma(i_{j})=i_{j+1\;(\operatorname{mod}\;k)}$; we represent
$\sigma$ by $(i_{1},\ldots,i_{k})$. Every
permutation has a cyclic decomposition:
that is, every permutation $\sigma\in\Sigma^{n}$ can be represented (in
a unique way up to reordering) as a composition of disjoint cycles
\[
\sigma=\sigma_{1}\circ\cdots\circ\sigma_{m},
\qquad \textrm{where the $\sigma_{i}$'s are cycles.}
\]
It is easy to see that if the cycle decomposition
of $\sigma \in \Sigma^n$ is
$$
(a_1,a_2, \ldots,a_{k_1})(b_1,b_2, \ldots,b_{k_2})\cdots
$$
then for any $\tau \in \Sigma^n$ the cycle decomposition of $\tau
\sigma \tau^{-1}$ is
\begin{equation}\label{eqn-17Nov2007-1}
(\tau(a_1),\tau(a_2),\ldots,\tau(a_{k_1}))(\tau(b_1),\tau(b_2),\ldots,\tau(b_{k_2}))\cdots
\end{equation}

Thus, the support $\mathrm{supp}(\sigma)$ of the permutation $\sigma$ is
the (disjoint) union of the supports $I_{i}
=\mathrm{supp}(\sigma_{i})$ of the cycles $\sigma_{i}$ of length at least two (the non-trivial cycles) in its
cycle decomposition. The partition
$$
\{I_{1},\ldots,I_{m},\mathbb{N}_{n}\setminus
\mathrm{supp}(\sigma)\}
$$
of $\mathbb{N}_{n}$ is thus naturally
associated to the permutation $\sigma$. Splitting further the set
$\mathbb{N}_{n}\setminus \mathrm{supp}(\sigma)$ into the singleton
sets $\{j\}$ we obtain a refined partition of $\mathbb{N}_{n}$
\begin{equation}\label{Partition}
P(\sigma):=\{I_{1},\ldots,I_{\kappa+m}\},
\end{equation}
where $\kappa$ is the cardinality of the complement of the support
of $\si$ in $\N_n$, and $m$ is the number of non-trivial cycles in
the cyclic decomposition of $\sigma$. For example, for
$\sigma=(123)(4)(5)\in\Sigma^5$ we have $\kappa=2$, $m=1$ and the
partition of $\{\{1,2,3\},\{4\},\{5\}\}$ of $\NN_5$. Thus, we obtain
a correspondence from the set of permutations $\Sigma^{n}$ onto the
set of partitions of $\mathbb{N}_{n}$.

\begin{definition}
\label{definition_order}
\begin{description}
\item[An order on the partitions:] Given two partitions $P$ and $P'$
  of $\mathbb{N}_{n}$ we say that $P'$ is a {\it refinement} of $P$,
  written $P \subseteq P'$, if every set in $P$ is a (disjoint) union
  of sets from $P'$. If $P'$ is a refinement of $P$ but $P$ is not a
  refinement of $P'$ then we say that the refinement is {\it strict}
  and we write $P \subset P'$.  Observe this partial order is a {\it
    lattice}.
\item[An order on the permutations:] The permutation $\sigma^{\prime}$
  is said to be {\it larger than or equivalent to $\sigma$}, written
  $\sigma\precsim\sigma^{\prime}$, if $P(\sigma) \subseteq
  P(\sigma^{\prime})$. The permutation $\sigma^{\prime}$ is said to be
  {\it strictly larger than $\sigma$}, written
  $\sigma\prec\sigma^{\prime}$, if $P(\sigma) \subset
  P(\sigma^{\prime})$.
\item[Equivalence in \boldmath$\Sigma^{n}$:] The permutations
  $\sigma,\sigma^{\prime}\in\Sigma^{n}$ are said to be
  \emph{equivalent}, written $\sigma\sim\sigma^{\prime}$, if they
  define the same partitions, that is if
  $P(\sigma)=P(\sigma^{\prime})$.

\item[Block-Size type of a permutation:]\label{def-SBST} Two
  permutations $\sigma$, $\si'$ in $\Sigma^n$ are said to be of the
  same \textit{block-size type}, whenever the set of cardinalities,
  counting repetitions, of the sets in the partitions $P(\si)$ and
  $P(\si')$, see \eqref{Partition}, are in a one-to-one
  correspondence.  Notice that if $\sigma$ and $\si'$ are of the same
  block-size type, then they are either equivalent or non-comparable.

\end{description}
\end{definition}

We give illustrations (by means of simple examples) of the
above notions, which are going to be used extensively in the paper.

\begin{example}[Permutations vs Partitions]
  The following simple examples illustrate the notions defined in
  Definition~\ref{definition_order}.
  \begin{itemize}
  \item[(i)] The set of permutations of $\Sigma^3$ that are larger than
    or equivalent to
    $\sigma:=(1,2,3)$ is
    \[
    \{(1,2,3),(1,3,2),(1,2),(1,3),(2,3), \mbox{\rm id}_3\}.
    \]
  \item[(ii)] The following three permutations of $\Sigma^4$
    have the same block-size type:
    \[
    \si = (123)(4), \quad \si' = (132)(4), \quad \si'' =(124)(3)\,.
    \]
    Note that the first two permutations are equivalent and not
    comparable to the third one.
  \item[(iii)] The minimal elements of $\Sigma^{n}$ under the partial order relation $\precsim$
    are exactly the $n$-cycles, corresponding to the partition
    $\{\mathbb{N}_{n}\}$.
  \item[(iv)] The (unique) maximum element of $\Sigma^n$ under $\precsim$ is the identity permutation $\mathrm{id}_{n}$,
    corresponding to the discrete partition
    $\{\{i\}:i\in\mathbb{N}_{n}\}$. \qed
  \end{itemize}
\end{example}

Consider two permutations $\sigma,\sigma'\in\Sigma^n$
such that $\sigma'\precsim \sigma$; according to the above,
each cycle of $\sigma'$ is either a permutation of the
elements of a cycle in $\sigma$ (giving rise to the same set in
the corresponding partitions $P(\sigma)$ and $P(\sigma')$) or it is
formed by merging (and permuting) elements from several cycles of
$\sigma$. If no cycle of $\sigma'$ is of the latter type, then
$\sigma$ and $\sigma'$ define the same partition (thus they are
equivalent), while on the contrary, $\sigma'\prec\sigma$.
Later, in \ssecref{ssec:structure},
we will introduce a subtle refinement of the order relation
$\prec$, which will be of crucial importance in our development.

We also introduce another partition of $\NN_n$ depending on the point
$x\in \RR^n$ denoted $P(x)$ and defined by the indexes of the equal
coordinates of $x$. More precisely, for $i,j \in \N_n$ we have:
\begin{equation}\label{aris1}
i,j\mbox{ are in the same subset of } P(x) ~~\Longleftrightarrow~~
x_{i}=x_{j}.
\end{equation}
This partition will appear frequently in the sequel, when we
study subsets of $\RR^n$ that are symmetric around $x$.
For $\bar x\in \mathbb{R}^{n}$ and $\bar \sigma\in\Sigma^{n}$, we define two invariant sets
\[
\mathrm{Fix\,}(\bar \sigma):=\{x\in\mathbb{R}^{n}:\bar \sigma x=x\}
\qqandqq \mathrm{Fix\,}(\bar x):=\{\sigma\in\Sigma^{n}:\sigma \bar
x=\bar x\}.
\]
Then, in view of \eqref{aris1} we have
\begin{equation}\label{partition_r}
  \bar \sigma\in\mathrm{Fix\,}(\bar x)
  ~~\Longleftrightarrow~~ \bar x\in \mathrm{Fix\,}(\bar \sigma)
  ~~\Longleftrightarrow~~ P(\bar x) \subseteq P(\bar \sigma).
\end{equation}

%********************************************************************
\subsection{Stratification induced by the permutation group}\label{ssec:strat}
%********************************************************************

In this section, we introduce a stratification of $\RR^n$ associated
with the set of permutations $\Sigma^n$. In view of
\eqref{partition_r}, associated to a permutation $\si$ is the subset
$\Delta(\sigma)$ of $\RR^n$ defined by
\begin{equation}\label{def:delta}
\Delta(\sigma) := \{x\in \RR^n: \  P(\sigma)=P(x)\}.
\end{equation}
For $\sigma \in \Sigma^n$ and $P(\sigma)=\{I_{1},\ldots,I_{m}\}$,
we have the representation
\[
\Delta(\sigma)= \{x\in\RR^{n} : x_i = x_j \iff \exists\,
k\in\N_m\mbox{ with } i,j \in I_k\}.
\]
Obviously $\Delta(\sigma)$ is an affine manifold, not connected in general.
Note also that its orthogonal and bi-orthogonal spaces
have the following expressions, respectively,
\begin{equation}\label{eqn-arisd3}
\Delta(\sigma)^{\perp} = \Big\{x\in\mathbb{R}^{n} :
\sum\limits_{j\in
  I_{i}}x_{j} =0,\text{ for } i \in \N_m \Big\},
\end{equation}
\begin{equation}\label{eqn-arisd4}
\Delta(\sigma)^{\perp\perp}=
\{x\in\mathbb{R}^{n} :
x_i = x_j \mbox{ for any } i,j \in I_k,\ k \in \N_m
\}.
\end{equation}
Note that $\Delta(\sigma)^{\perp\perp} = \overline{\Delta(\sigma)}$,
where the latter set is the closure of  $\Delta(\sigma)$.
Thus, $\Delta(\sigma)^{\perp}$ is a vector space of dimension $n-m$
while $\Delta(\sigma)^{\perp\perp}$ is a vector space of
dimension~$m$.  For example, $\Delta({\rm id}_n)^{\bot}=\{0\}$ and
$\Delta({\rm id}_n)^{\bot\bot}=\mathbb{R}^n$. We show now, among other things, that
$\{\Delta(\sigma) : \sigma \in \Sigma^n\}$ is a {\it stratification}
of $\RR^n$, that is, a collection of disjoint smooth submanifolds of
$\RR^n$ with union $\RR^n$ that fit together in a regular way.
In this case, the submanifolds in the
stratification are affine.

\begin{figure}[!htb]
\begin{center}
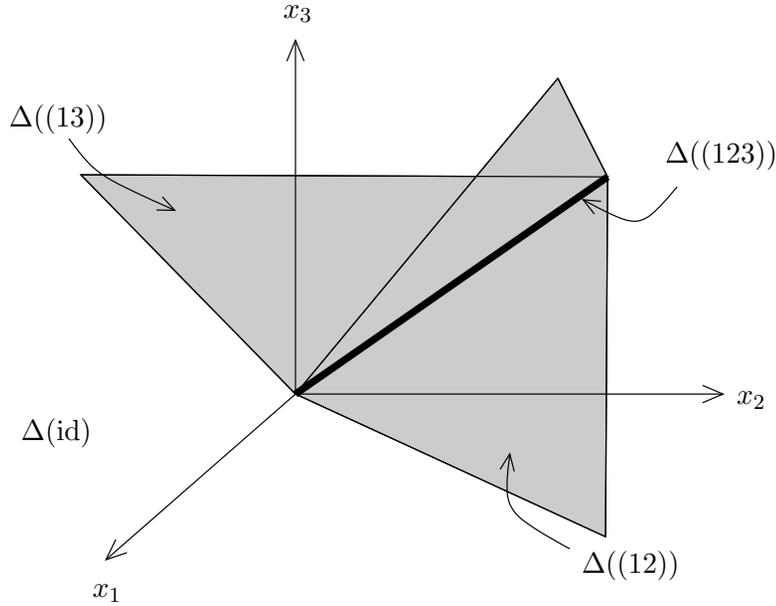\\*[-3pt]
\caption{The affine stratification of $\RR^3$\label{strat}}
\end{center}
\end{figure}

\vspace*{-11pt}

\begin{proposition}[Properties of $\Delta(\si)$]
\label{prop:delta}
(i) Let $x\in\mathbb{R}^{n}$ and let $P$ be any partition of
$\mathbb{N}_{n}$. Then, $P(x) \subseteq P$ if and only if there is a
sequence $x_{n} \rightarrow x$ in $\R^n$ satisfying $P(x_{n})=P$ for
all $n\in\mathbb{N}$. \smallskip

 (ii) Let $\sigma,\sigma^{\prime}\in\Sigma^{n}$. Then,
\begin{equation}
\label{Ph1}
  \sigma\precsim \sigma^{\prime}\Longleftrightarrow\Delta(\sigma)\subset
  \Delta(\sigma^{\prime})^{\bot\bot},
\end{equation}
\begin{equation}
\label{Ph2}
  \sigma\sim\sigma^{\prime}\Longleftrightarrow\Delta(\sigma)\cap\Delta
  (\sigma^{\prime})\neq\emptyset\Longleftrightarrow\Delta(\sigma)=\Delta
  (\sigma^{\prime}).
\end{equation}

(iii) For any $\sigma \in \Sigma^n$ we have
\begin{equation}
  \label{Ph3}
  \Delta(\sigma)^{\bot\bot} = \bigcup_{\sigma^{\prime} \precsim \sigma} \Delta(\sigma^{\prime}).
\end{equation}

\pagebreak

(iv) Given $\sigma, \sigma' \in \Sigma^n$ let $\sigma \wedge
\sigma'$ be any infimum of $\si$ and $\si'$ $($notice this is unique modulo $\sim)$. Then
\begin{equation}
  \label{Ph4}
  \Delta(\sigma)^{\bot\bot} \cap \Delta(\sigma')^{\bot\bot} = \Delta(\sigma \wedge \sigma')^{\bot\bot}.
\end{equation}

(v) For any $\tau, \sigma \in \Sigma^n$ we have
\begin{equation*}\label{lem-17Nov2007}
  \tau \Delta(\sigma) = \Delta(\tau \sigma \tau^{-1}).
\end{equation*}
\end{proposition}

\begin{proof}
Assertion $(i)$ is straightforward. Assertion $(ii)$ follows from $(i)$,
\eqref{partition_r}, \eqref{def:delta} and \eqref{eqn-arisd4}.
Assertion~$(iii)$ is a direct consequence of $(i)$, $(ii)$ and
\eqref{eqn-arisd4}.

To show assertion $(iv)$, let first $x\in
\Delta(\sigma)^{\bot\bot} \cap \Delta(\sigma')^{\bot\bot}$. Then, in
view of $(iii)$, there exist $\tau_1\precsim\si$ and
$\tau_2\precsim\si'$ such that
$x\in\Delta(\tau_1)\cap\Delta(\tau_2)$. Thus, by \eqref{Ph2},
$\tau_1\sim\tau_2$ and by \eqref{Ph1} they are both smaller than or
equivalent to $\sigma \wedge\sigma'$. Thus, $x\in\Delta(\sigma
\wedge\sigma')^{\bot\bot}$ showing that $\Delta(\sigma)^{\bot\bot}\cap
\Delta(\sigma')^{\bot\bot}\subset\Delta(\si\wedge\si')^{\bot\bot}$.
Let now $x\in\Delta(\sigma\wedge\si')^{\bot\bot}$. Then, for some
$\tau\precsim\si\wedge\si'$ we have $x\in\Delta(\tau)$. Since
$\tau\precsim\si$ and $\tau\precsim\si'$ the inverse inclusion
follows from $(iii)$.

We finally prove $(v)$. We
have that $x \in \tau \Delta(\sigma)$ if and only if $\tau^{-1}x\in
\Delta(\sigma)$. This latter happens if and only if for all $i,j
\in\N_n$ one has $(\tau^{-1}x)_i=(\tau^{-1}x)_j$ precisely
when $i,j$ belong to the same cycle of $\sigma$. By
(\ref{eqn-sigma}), this is equivalent to $x_{\tau(i)}=x_{\tau(j)}$
precisely when $i,j$ are in the same cycle of $\sigma$ for all $i,j
\in \N_n$. In view of \eqref{eqn-17Nov2007-1}, $i,j$ are in the same
cycle of $\sigma$ if and only if $\tau(i)$, $\tau(j)$ are in the
same cycle of $\tau \sigma \tau^{-1}$. This completes the proof.
\end{proof}

\begin{corollary}[Stratification]\label{Corollary_stratification}
The collection $\{\Delta(\sigma):\sigma\in\Sigma^{n}\}$
is an affine stratification of $\mathbb{R}^{n}$.
\end{corollary}

\begin{proof}
  Clearly, each $\Delta(\sigma)$ is an affine submanifold of $\RR^n$.
  By (\ref{Ph2}), for any $\sigma,\sigma^{\prime}\in\Sigma^{n}$, the
  sets $\Delta(\sigma)$ and $\Delta(\sigma^{\prime})$ are either
  disjoint or they coincide. Thus, the elements in the set
  $\{\Delta(\sigma):\sigma\in\Sigma^{n}\}$ are disjoint. By
  construction, the union of all $\Delta(\sigma)$'s equals $\RR^{n}$. The frontier condition
  of the stratification follows from \eqref{eqn-arisd4} and \eqref{Ph3}.
\end{proof}

We introduce an important set for our next development. Consider the
set of permutations that are larger than, or equivalent to a given
permutation $\sigma \in \Sigma^n$
\[
S^{\succsim}(\sigma):=\{\sigma^{\prime}\in\Sigma^{n}:\sigma^{\prime}
\succsim\sigma\}.
\]
Notice that $S^{\succsim}(\sigma)$ is a subgroup of $\Sigma^n$, and
that
\begin{equation}\label{cardS}
|S^{\succsim}(\sigma)| = (|I_{1}|)!\cdots (|I_{m}|)!,
\end{equation}
if $P(\sigma)=\{I_{1},\ldots,I_{m}\}$.
Observe then that $\si\sim
\si'$ if and only if $S^{\succsim}(\sigma) =
S^{\succsim}(\sigma')$. So we also introduce the corresponding set for
a point $x\in \RR^n$
\begin{equation}\label{Sx}
  S^{\succsim}(x) :=   S^{\succsim}(\si) \qquad \textrm{for any $\si$ such that $x\in
  \Delta(\si)$}\,,
\end{equation}
which is nothing else than the set $\mathrm{Fix\,}(x)$. The
forthcoming result shows that the above permutations are the only
ones preserving balls centered at $\bar x$.

\begin{lemma}[Local invariance and ball preservation]
  \label{key}
  For any $\bar x\in \RR^n$, we have the dichotomy:
  \begin{enumerate}[(i)]
  \item \label{key_pt1}  $\sigma\in S^{\succsim}(\bar x) \ \Longleftrightarrow \
    \forall \delta > 0: \ \sigma B(\bar{x},\delta) =B(\bar{x},\delta)$;
  \item \label{key_pt2} $\sigma \not\in S^{\succsim}(\bar x) \
    \Longleftrightarrow \ \exists \delta > 0: \  \sigma
    B(\bar{x},\delta)   \cap B(\bar{x},\delta )=\emptyset$.
  \end{enumerate}
\end{lemma}

\begin{proof}
  Observe that $\si\in S^{\succsim}(\bar x)$ if and only if
  $P(\bar x) \subseteq P(\sigma)$ if and only if $\norm{\bar x- \si\bar x}=0$.
  So implication $\Leftarrow$ of $(i)$ follows by taking
  $\de\ra 0$. The implication $\Rightarrow$ of $(i)$ comes from the symmetry of
  the norm which yields for any $x\in \RR^n$
  \[
  \norm{x-\bar x} = \norm{\si x - \si \bar x} = \norm{\si x - \bar x}.
  \]
  
  \pagebreak
  
  \noindent
  To prove $(ii)$, we can just consider $\de=\norm{\bar x - \si\bar
    x}/3$ and note that $\de>0$ whenever $\si\notin
  S^{\succsim}(\bar x)$. Utilizing
  $$
  \norm{\bar x - \sigma x} \ge |\norm{\bar x- \si \bar x } - \norm{\si \bar x - \si x}  |
  =  \norm{\bar x- \si \bar x } - \norm{\bar x - x} \geq 2\de
  $$
  concludes the proof.
\end{proof}

In words, if the partition associated to $\si$ refines the partition of
$\bar x$, then $\si$ preserves all the balls centered at $\bar x$;
and this property characterizes those permutations. The next corollary goes
a bit further by saying that the preservation of only one ball, with a
sufficiently small radius, also characterizes $S^{\succsim}(\bar x)$.

\begin{corollary}[Invariance of one ball]
   For every $\bar x\in \RR^n$ there exists $r>0$ such that:
   \[
   \sigma\in S^{\succsim}(\bar x)  \Longleftrightarrow
   \sigma B(\bar{x},r) =B(\bar{x},r)
   \qqandqq
   \sigma \not\in S^{\succsim}(\bar x)
   \Longleftrightarrow \sigma
   B(\bar{x},r)   \cap B(\bar{x},r)=\emptyset.
   \]
\end{corollary}

\begin{proof}
  For any $\si\notin S^{\succsim}(\bar x)$, Lemma~\ref{key}$(ii)$ gives
  a radius, that we denote here by $\de_{\si}>0$, such that $\sigma
  B(\bar{x},\delta_{\si}) \cap B(\bar{x},\delta_{\si})=\emptyset$.
  Note also that for all $\de\leq\de_{\si}$, there still holds $\sigma
  B(\bar{x},\delta) \cap B(\bar{x},\delta )=\emptyset$.  Set now
  \[
  r = \min\left\{ \de_{\si} : \ \si\notin  S^{\succsim}(\bar x)\right\} > 0.
  \]
  Thus $\sigma B(\bar{x},r) \cap B(\bar{x},r)=\emptyset$ for all
  $\si\notin S^{\succsim}(\bar x)$. This yields that if
  a permutation preserves the ball $B(\bar{x},r)$, then it lies in
  $S^{\succsim}(\bar x)$. The converse comes from Lemma~\ref{key}.
\end{proof}

We finish this section by expressing the orthogonal
projection of a point onto a given stratum using permutations. Letting
$P(\sigma)=\{I_{1},\ldots,I_{m}\}$, it is easy to see that
\begin{equation}\label{aris2}
y=\mathrm{Proj\,}_{\Delta(\sigma)^{\bot\bot}}(x)
\quad\Longleftrightarrow \quad y_{\ell}=\frac{1}{|I_{i}|}\sum\limits_{j\in
I_{i}}x_{j} \quad\text{for all \,}\, \ell \in
I_{i}\,\text{ with } i\in\N_m.
\end{equation}
Note also that if the numbers
$$
\frac{1}{{|I_i|}}\sum_{j \in I_i}\, x_j \qquad \mbox{for }
i\in\N_m
$$
are distinct, then this equality also provides
the projection of $x$ onto the (non-closed) set $\Delta(\sigma)$.
We can state the following result.

\begin{lemma}[Projection onto $\Delta(\si)^{\bot\bot}$]\label{Lemma_project_strata}
  For any $\sigma\in\Sigma^{n}$ and $x\in\mathbb{R}^{n}$ we have
  \begin{equation}\label{aris7}
    \mathrm{Proj\,}_{\Delta(\sigma)^{\bot\bot}}(x)=
    \frac{1}{|S^{\succsim}(\sigma)|}
    \sum\limits_{\sigma^{\prime}\succsim\sigma} \sigma^{\prime}x.
  \end{equation}
\end{lemma}

\begin{proof}
For every $j,\ell \!\in\! I_i$, the coordinate $x_j$ is repeated
$|S^{\succsim}(\sigma)|/|I_i|$ times in the sum
$\big(\sum_{\sigma^{\prime}\succsim\sigma}
\sigma^{\prime}x\big)_{\ell}$. Thus, \eqref{aris2} together with
\eqref{cardS} yields the result.
\end{proof}

%% In particular, the
%% projection operator is represented by the following $n \times
%% n$ symmetric matrix $Q$
%% \begin{equation}
%% \label{aris-Q}
%% Q^{lj} = \left\{
%% \begin{array}{ll}
%%   1/|I_i| & \mbox{\quad if } l,j  \in I_i, \\
%%   0 & \mbox{\quad if } l \in I_k, \, j \in I_i \mbox{ and } k \not= i,
%% \end{array}
%% \right.
%% \end{equation}
%% and formula \eqref{aris7} becomes
%% \begin{align}
%% \label{26Nov2008-hss}
%% \mbox{\rm Proj}_{\Delta(\sigma)^{\perp\perp}}\,(x) = Q \,x\,.
%% \end{align}

%********************************************************************
%********************************************************************
\section{Locally symmetric manifolds}\label{sec:manifold}
%********************************************************************
%********************************************************************

In this section we introduce and study the notion of locally
symmetric manifolds; we will then prove in Section~4 that these
submanifolds of $\RR^n$ are lifted up, via the mapping
$\lambda^{-1}$, to spectral submanifolds of $\Sn$.

After defining the notion of a locally symmetric manifold in
Subsection~3.1, we illustrate some intrinsic difficulties that
prevent a direct proof of the aforementioned result. In
\ssecref{ssec-tang} we study properties of the tangent and the
normal space of such manifolds. In Subsection~\ref{ssec:structure},
we specify the location of the manifold with respect to the
stratification, which leads in Subsection~\ref{ssec-simin} to the
definition of a {\it characteristic} permutation naturally
associated with a locally symmetric manifold. We explain in
\ssecref{sssec-341} that this induces a {\it canonical
decomposition} of $\RR^n$ yielding a reduction of the active normal
space in \ssecref{sssec-342}. Finally, in \ssecref{ssec-tangparam}
we obtain a very useful description of such manifolds by means of a
{\em reduced} locally symmetric local equation. This last step will
be crucial for the proof of our main result in
Section~\ref{sec:spectral}.

%********************************************************************
\subsection{Locally symmetric functions and manifolds}\label{ssec-loc-sym}
%********************************************************************

Let us start by refining the notion of symmetric function
employed in previous works (see \cite{LewisSendov:2000a},
\cite{DLMS08} for example).

\begin{definition}[Locally symmetric function]\label{Definition_locsymfun}
  A function $f\colon \RR^n\rightarrow\RR$ is called {\it locally symmetric}
  around a point $\bar x\in \RR^n$ if for any $x$ close to $\bar x$
  $$
  f(\si x)=f(x) \quad \text{for all } \si\in S^{\succsim}(\bar x)\,.
  $$
  Naturally, a vector-valued function $g\colon\RR^n\rightarrow\RR^p$ is called {\it locally symmetric}
  around $\bar x$ if each component function $g_i\colon \RR^n\ra \RR$
  is locally symmetric $(i=1,\dots,p$).
\end{definition}

In view of \lemref{key} and its corollary, locally symmetric
functions are those which are symmetric on an open ball centered at
$\bar x$, under all permutations of entries of $x$ that preserve
this ball. It turns out that the above property is the invariance
property needed on $f$ for transferring its differentiability
properties to the spectral function $f \circ \la$, as stating in the
next theorem. Recall that for any vector $x$ in $\RR^n$, Diag$\,x$
denotes the diagonal matrix with the vector $x$ on the main
diagonal, and diag$\,: \Sn \rightarrow \RR^n$ denotes its adjoint
operator, defined by diag$\,(X):=(x_{11},\ldots,x_{nn})$ for any matrix
$X=(x_{ij})_{ij} \in \Sn$.

\begin{theorem}[Derivatives of spectral functions]
  \label{thm-26Nov2007-1}
  Consider a function $f\colon \R^n \rightarrow \R$ and define the function
  $F\colon \Sn \rightarrow \R$ by
  $$
  F(X)= (f \circ \lambda)(X)
  $$
  in a neighborhood of $\bar X$. If $f$ is locally symmetric at $\bar x$, then
  \begin{itemize}
  \item[(i)] the function $F$ is $C^1$ at $\bar X$ if and only if $f$ is
    $C^1$ at $\lambda(\bar X)$;
  \item[(ii)] the function $F$ is $C^{2}$ at $\bar X$ if and only if $f$
    is $C^{2}$ at $\lambda(\bar X)$;
  \item[(iii)] the function $F$ is $C^{\infty}$ (resp. $C^{\omega}$) at
    $\bar X$ if and only if $f$ is $C^{\infty}$ (resp. $C^{\omega}$) at
    $\lambda(\bar X)$, where $C^{\omega}$ stands for the class of
    real analytic functions.
  \end{itemize}

  In all above cases we have
  $$
  \nabla F(\bar X) = \trans{\bar{U}}(\Diag \, \nabla f(\lambda(\bar X))) \bar{U}
  $$
  where $\bar U$ is any orthogonal matrix such that
  $X=\trans{\bar{U}}(\Diag\, \lambda(\bar X))\bar U$. Equivalently, for any
  direction $H \in \Sn$ we have
  \begin{equation}\label{H1}
    \nabla F(\bar X)[H] =  \nabla f(\lambda(\bar X)))[\diag\,(\bar{U}H \trans{\bar{U}})].
  \end{equation}
\end{theorem}

\begin{proof}
  The proof of the results is virtually identical with the
  proofs in the case when $f$ is a symmetric function with respect to
  all permutations. For a proof of $(i)$ and the
  expression of the gradient, see \cite{Lewis:1994b}. For $(ii)$, see
  \cite{LewisSendov:2000a} (or \cite[Section~7]{sendov2007}), and for
  $(iii)$ \cite{Dadok:1982} and \cite{TsingFanVerriest:1994}.
\end{proof}

\pagebreak

The differentiability of spectral functions will be used intensively when
defining local equations of spectral manifolds.
Before giving the definition of spectral manifolds and locally symmetric manifolds,
let us first recall the definition of submanifolds.

\begin{definition}[Submanifold of $\R^n$]
  A nonempty set $\mathcal{M} \subset \R^n$ is a $C^k$ {\it submanifold} of
  dimension~$d$ (with $d\in \{0,\ldots,n\}$ and $k\in \NN\cup \{\omega\}$)
  if for every $\bar x \in \mathcal{M}$, there is a neighborhood $U \subset \R^n$
  of $\bar x$ and $C^k$ function $\varphi \colon U \rightarrow \R^{n-d}$ with Jacobian
  matrix $J\varphi(\bar x)$ of full rank, and such that for all $x \in U$ we
  have $x \in \mathcal{M} \Leftrightarrow \varphi(x)=0$.  The map $\varphi$
  is called {\it local equation} of $\mathcal{M}$ around $\bar x$.
\end{definition}

\begin{remark}[Open subset]\label{claude}
  Every (nonempty) open subset of $\RR^n$ is trivially a
  \mbox{$C^k$-submanifold} of $\RR^n$ (for any $k$) of dimension $d=n$.
\end{remark}

\begin{definition}[Locally symmetric sets]\label{defn-spectrman}
  Let $S$ be a subset of $\RR^n$ such that
  \begin{equation}
    S\,\cap\,\RR^n_{\,\ge}\,\neq\, \emptyset\,.
  \end{equation}
  The set $S$ is called {\it strongly locally symmetric} if
  $$
  \sigma S = S \qquad \mbox{ for all } \bar x
  \in S \mbox{ and all }  \sigma\in S^{\succsim}(\bar x).
  $$ The set $S$ is called {\it locally symmetric} if for every $x \in
  S$ there is a $\delta > 0$ such that $S \cap B(x, \delta)$ is
  strongly locally symmetric set. In other words, for every $x \in S$
  there is a $\delta > 0$ such that
  $$
  \sigma (S \cap B(x, \delta)) = S \cap B(x, \delta) \qquad \mbox{ for all }
  \bar x \in S \cap B(x, \delta) \mbox{ and all } \sigma \in S^{\succsim}(\bar x).
  $$
  In this case, observe that $S\cap B(x, \rho)$ for $\rho\leq
  \delta$ is a strongly locally symmetric set as well (as an easy consequence
  of \lemref{key}).
\end{definition}

\begin{example}[Trivial examples]
  Obviously the whole space $\RR^n$ is (strongly locally) symmetric.
  It is also easily seen from the definition that any stratum $\Delta(\si)$ is a strongly locally
  symmetric affine manifold. If $\bar x \in \Delta(\sigma)$ and the ball $B(\bar x, \delta)$
  is small enough so that it intersects only strata $\Delta(\sigma')$ with $\sigma' \succsim
  \sigma$, then $B(\bar x, \delta)$ is strongly locally symmetric. \qed
\end{example}

\begin{definition}[Locally symmetric manifold]
  A subset $\mathcal{M}$ of $\RR^n$ is said to be a {\it (strongly) locally symmetric manifold}
  if it is both a connected submanifold of $\RR^n$ without boundary and a (strongly) locally symmetric set.
\end{definition}

Our objective is to show that locally symmetric smooth submanifolds
of $\RR^n$ are lifted to (spectral) smooth submanifolds of $\Sn$.
Since the entries of the eigenvalue vector $\la(X)$ are
non-increasing (by definition of $\la$), in the above
definition we only consider the case where $\mathcal{M}$ intersects $\RR^n_{\,\ge}$.
Anyhow, this technical assumption is not restrictive
since we can always reorder the orthogonal basis of $\RR^n$ to get
this property fulfilled. Thus, our aim is to show that
$\lambda^{-1}(\Mm \cap\RR^n_{\ge})$ is a manifold, which will be
eventually accomplished by Theorem~\ref{main-thm-4Dec2007-3} in
Section~\ref{sec:spectral}.

\smallskip

Before we proceed, we sketch two simple approaches that could be
adopted, as a first try, in order to prove this result, and we
illustrate the difficulties that appear.

\smallskip

The first example starts with the expression \eqref{jm-adds} of the
manifold $\la^{-1}(\Mm)$. Introduce the {\it stabilizer} of a matrix
$X\in \Sn$ under the action of the orthogonal group $\On$
$$
\On_{X} :=\{U \in \On : \ \trans{U} X U = X\}.
$$
Observe that for $x\in \RR^n_{\geq}$, we have an exact description
of the stabilizer $\On_{\Diag(x)}$ of the matrix $\Diag(x)$. Indeed,
considering the partition $P(x) = \{I_1,\ldots,I_{\kappa+m}\}$ we
have that $U\in \On_{\Diag(x)}$ is a block-diagonal matrix, made of
matrices $U_i\in {\bf O}^{|I_i|}$. Conversely, every such
block-diagonal matrix belongs clearly to $\On_{\Diag(x)}$. In other
words, we have the identification
\[
\On_{\Diag(x)} \simeq {\bf O}^{|I_1|} \times \cdots \times {\bf
O}^{|I_{\kappa+m}|}.
\]

\pagebreak

\noindent
Since ${\bf O}^p$ is a manifold of dimension
$p(p-1)/2$, we deduce that $\On_{\Diag(x)}$ is a manifold of dimension
\[
\dim \On_{\Diag(x)} = \sum^{\kappa+m}_{i=1}
\frac{|I_i|(|I_i|-1)}{2}.
\]
It is a standard result that the orbit $\On\!.{\Diag(x)}$ is diffeomorphic to
the quotient manifold $\On\!/\On_{\!\Diag(x)\!}$. Thus, $\On.{\Diag(x)}$ is a submanifold of $\Sn$ of
dimension
\begin{align*}
\dim \On.{\Diag(x)} &=  \dim \On - \dim \On_{\Diag(x)} \\
  &= \frac{n(n-1)}{2} - \sum_{i=1}^{\kappa + m} \frac{|I_{i}|(|I_{i}|-1)}{2} \\
  &= \frac{n^{2} - \sum_{i=1}^{\kappa + m} |I_{i}|^{2}}{2} \\
  &= \sum_{1 \le i < j \le \kappa + m} |I_{i}||I_{j}|,
\end{align*}
where we used twice the fact that $n=\sum^{\kappa+m}_{i=1}|I_i|$. What
we need to show is that the (disjoint) union of these manifolds
$$
\lambda^{-1}(\mathcal{M}) =\bigcup_{x \in
  \mathcal{M}} \On.\Diag(x)
$$
is a manifold as well. We are not aware of a straightforward answer
to this question. Our answer,  developed in \secref{sec:spectral},
uses crucial properties of locally symmetric manifolds
derived in this section. We also exhibit explicit local equations
of the spectral manifold $\lambda^{-1}(\mathcal{M})$.

\smallskip

Let us finish this overview by explaining how a second
straightforward approach involving local equations of manifolds
would fail. To this end, assume that the manifold $\Mm$ of dimension
$d\in \{0,1,\dots, n\}$ is described by a smooth equation
$\ph\colon \RR^n\rightarrow\RR^{n-d}$ around the point
$\bx\in\Mm\cap\RR^n_{\,\ge}$.  This gives a function $\ph\circ \la$
whose zeros characterize $\la^{-1}(\Mm)$ around
$\bX\in\la^{-1}(\Mm)$, that is, for all $X\in\Sn$ around $\bX$
\begin{equation}\label{eq:equation}
X\in \la^{-1}(\Mm) \iff \la(X)\in \Mm \iff \ph(\la(X))=0.
\end{equation}
However we cannot guarantee that the function
$\Phi:=\varphi\circ\lambda$ is a smooth function unless
$\ph$ is locally symmetric (since in this case
\thref{thm-26Nov2007-1} applies). But in general, local equations
$\ph\colon\RR^n\ra\RR$ of a locally symmetric submanifold of $\RR^n$
might fail to be locally symmetric, as shown by the next easy
example.

\begin{example}[A symmetric manifold without symmetric
equations]\label{ex:jm}
  Let us consider the following symmetric (affine) submanifold
  of $\RR^2$ of dimension one:
  $$
  \Mm \ = \ \{(x,y)\in \RR^2: x=y\} \ = \ \Delta((12)).
  $$
  The associated spectral set 
  \[
  \la^{-1}(\Mm) = \{A\in \Sn: \la_1(A)=\la_2(A)\}=\{\al I_n: \ \al\in \RR\}
  \]
  is a submanifold of $\Sn$ around $I_n = \la^{-1}(1,1)$. It is
  interesting to observe that though $\la^{-1}(\Mm)$ is a (spectral)
  $1$-dimensional submanifold of $\Sn$, this submanifold cannot be
  described by local equation that is a composition of $\la$ with
  $\ph\colon\RR^2\ra\RR$ a symmetric local
  equation of $\Mm$ around $(1,1)$. Indeed, let
  us assume on the contrary that such a local equation of
  $\Mm$ exists, that is, there exists a smooth symmetric
  function $\ph\colon \RR^2\rightarrow\RR$ with surjective derivative
  $\nabla \ph(1,1)$ which satisfies
  \[
  \ph(x,y) = 0 \iff x=y\,.
  \]
  Consider now the two smooth paths $c_1\colon t\mapsto(t,t)$ and
  $c_2\colon t\mapsto(t,2-t)$. Since $\ph\circ c_1(t) = 0$ we infer
  \begin{equation}\label{eq-c1}
    \nabla\ph(1,1)(1,1)=0.
  \end{equation}
  On the other hand, since $c_2'(1)=(1,-1)$ is normal to $\Mm$ at
  $(1,1)$, and since $\ph$ is symmetric, we deduce that
  the smooth function $t \mapsto (\ph\circ c_2)(t)$ has a local extremum at $t=1$.
  Thus,
  \begin{equation}\label{eq-c2}
     0=(\ph\circ c_2)'(1)=\nabla\ph(1,1)(1,-1).
  \end{equation}
  Therefore, \eqref{eq-c1} and \eqref{eq-c2} imply that
  $\nabla\ph(1,1)=(0,0)$ which is a contradiction. This proves that
  there is no symmetric local equation $\varphi\colon \RR^2\ra \RR$
  of the symmetric manifold $\Mm$ around $(1,1)$. \qed
\end{example}

We close this section by observing that the property of local
symmetry introduced in Definition~\ref{defn-spectrman} is
necessary and in a sense minimal. In any case, it cannot easily be
relaxed as reveals the following examples.

\begin{example}[A manifold without symmetry]
  Let us consider the set
  \[
  \mathcal{N}=\{(t,0): t\in (-1,1)\}\,\subset\,\RR^2.
  \]
    We have an explicit expression of $\la^{-1}(\Nn)$
  \[
  \la^{-1}(\Nn) = \left\{
  \left[\begin{array}{cc}
      t\cos^2\theta  & t(\sin2\theta)/2\\
      t(\sin2\theta)/2 & t\sin^2\theta  \\
    \end{array}\right], \
  \left[\begin{array}{cc}
      -t\sin^2\theta  & t(\sin2\theta)/2\\
      t(\sin2\theta)/2 & -t\cos^2\theta  \\
    \end{array}\right],
  \ t\geq 0
  \right\}.
  \]
      \begin{figure}[h]%[!htb]
    \begin{center}
    \includegraphics[scale=0.75]{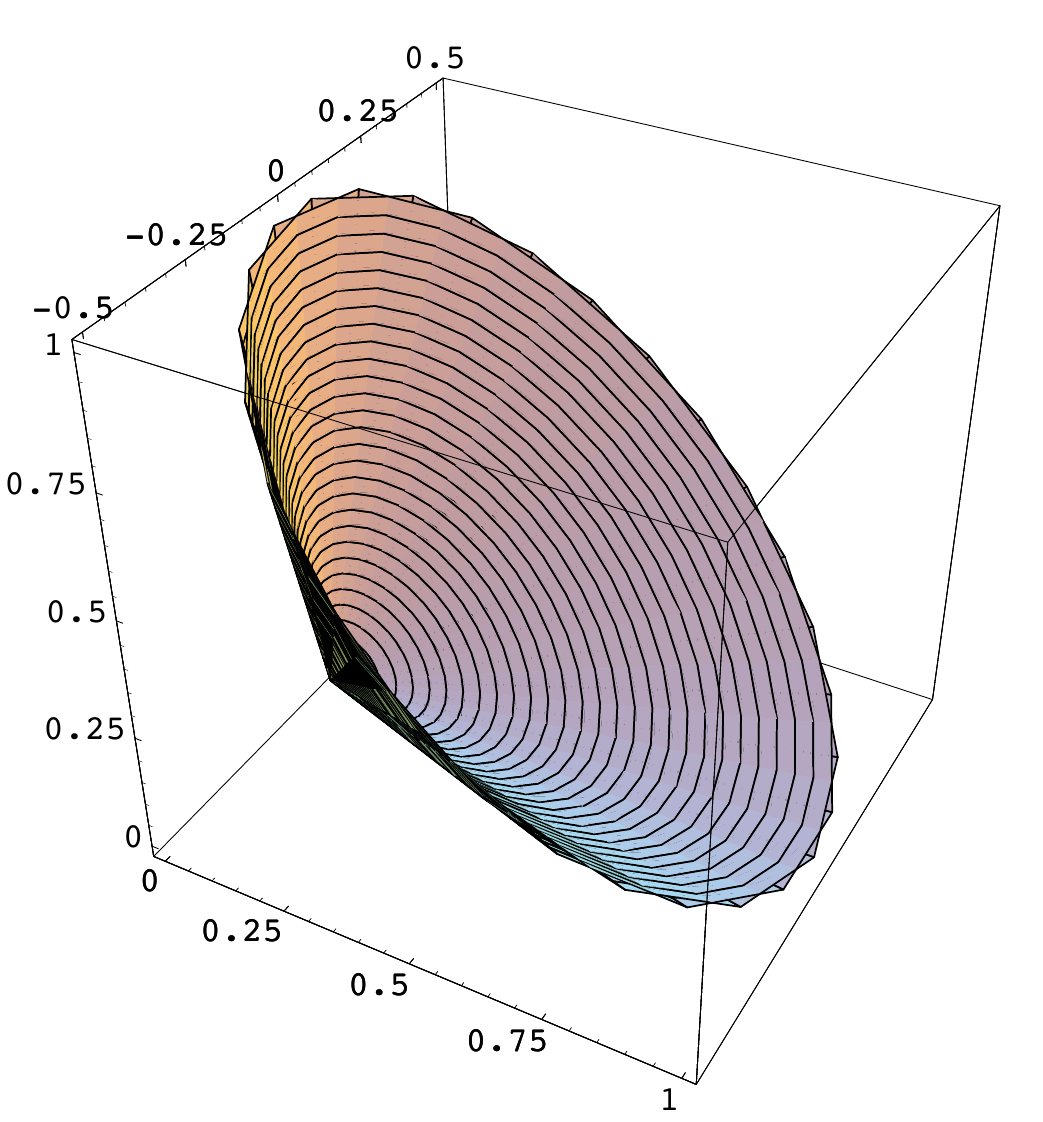}
    \caption{A spectral set of ${\bf S}_2$ represented in $\RR^3$\label{fig-ex3D}}
   \end{center}
  \end{figure}

\noindent
  It can be proved that this lifted set is not a submanifold of $S^2$
  since it has a sharp point at the zero matrix, as suggested
  by its picture in $\RR^3\simeq S^2$ (see Figure~\ref{fig-ex3D}).
%      \begin{figure}[h]%[!htb]
%    \begin{center}
%    \includegraphics[scale=0.75]{ex3D.pdf}
%    \caption{A spectral set of ${\bf S}_2$ represented in $\RR^3$\label{fig-ex3D}}
%   \end{center}
%  \end{figure}
   \qed
\end{example}

\begin{example}[A manifold without enough symmetry]
  Let us consider the set
  $$
  \mathcal{N}=\{(t,0,-t): t\in
  (-\epsilon,\epsilon)\}\,\subset\,\RR^3
  $$
  and let $\bar x
  =(0,0,0)\in\Nn$, $\si=(1,2,3)$. Then, $\Delta(\si) =
  \{(\alpha,\alpha,\alpha) : \alpha \in \R\}$ and $\Nn$ is a smooth
  submanifold of $\R^3$ that is
  symmetric with respect to the affine set $\Delta(\sigma)$, but it
  is not locally symmetric. It can be easily proved that the set
  $\lambda^{-1}(\Mm)$ is not a submanifold of ${\bf S}_3$ around the zero
  matrix. \qed
\end{example}

%********************************************************************
\subsection{Structure of tangent and normal space}\label{ssec-tang}
%********************************************************************

From now on
\begin{center}
{\em  $\mathcal{M}$ is a locally symmetric $C^2$-submanifold of $\R^n$ of dimension $d$,}
\end{center}
unless otherwise explicitly stated. We also denote by
$T_{\mathcal{M}}(\bar{x})$ and $N_{\mathcal{M}}(\bar{x})$ its
tangent and normal space at $\bar{x}\in\Mm$, respectively. In this subsection, we
derive several natural properties for these two spaces, stemming from the
symmetry of $\Mm$. The next lemma ensures
that the tangent and normal spaces at $\bar{x}\in\Mm$ inherit the
local symmetry of $\Mm$.

\begin{lemma}[Local symmetry of $T_{\Mm}(\bx)$, $N_{\Mm}(\bx)$] \label{Tinv}
If $\bar x \in \mathcal{M}$ then
\begin{enumerate}[(i)]
\item \label{Tinv-i}
  $\sigma T_{\mathcal{M}}(\bar{x})= T_{\mathcal{M}}(\bar{x}) \,\,\,
  \text{for all } \sigma \in S^{\succsim}(\bar x);$
\item \label{Tinv-ii}
  $\sigma N_{\mathcal{M}}(\bar{x})=
  N_{\mathcal{M}}(\bar{x}) \,\,\, \text{for all } \sigma
  \in S^{\succsim}(\bar x)\,.$
\end{enumerate}
\end{lemma}

\begin{proof}
  Assertion~$(\ref{Tinv-i})$ follows directly from the definitions since the elements of
   $T_{\Mm}(\bx)$ can be viewed as the differentials at $\bx$ of smooth paths on $\Mm$.
  Assertion~$(\ref{Tinv-ii})$ stems from the fact that
  $S^{\succsim}(\sigma)$ is a group, as follows: for any $w\in
  T_{\Mm}(\bx)$, $v \in N_{\mathcal{M}}(\bar{x})$, and $\sigma \in
  S^{\succsim}(\sigma)$ we have $\si^{-1}w\in T_{\Mm}(\bx)$ and
  $\langle \sigma v, w \rangle = \langle v, \sigma^{-1}w \rangle =
  0$, showing that $\si\,v\in [T_{\Mm}(\bx)]^{\bot}=N_{\Mm}(\bx)$.
\end{proof}

Given a set
$S\subset\mathbb{R}^{n}$, denote by $\mathrm{dist}_{S}(x):=\inf_{s\in S} \|x-s\|$ the
distance of $x\in\mathbb{R}^{n}$ to $S$.

\begin{proposition}[Local invariance of the distance]\label{Proposition_distance}
  If $\bar {x}\in\mathcal{M}$, then
  \[
  \mathrm{dist}_{(\bar{x}+T_{\mathcal{M}}(\bar{x}))}(x)=
  \mathrm{dist}_{(\bar{x}+T_{\mathcal{M}}(\bar{x}))}(\sigma x)
  \quad \text{for any \,}\,x\in\R^n
  \,\text{ and }\,\sigma \in S^{\succsim}(\bar x)\,.
  \]
\end{proposition}

\begin{proof}
  Assume that for some $x\in\R^n$ and $\sigma\in
  S^{\succsim}(\bar x)$ we have
  \[
  \mathrm{dist}_{(\bar{x}+T_{\mathcal{M}}(\bar{x}))}(x)<\mathrm{dist}_{(\bar
    {x}+T_{\mathcal{M}}(\bar{x}))}(\sigma x).
  \]
  Then, there exists $z\in T_{\mathcal{M}}(\bar{x})$ satisfying
  $||x-(\bar{x}+z)||<\mathrm{dist}_{(\bar{x}+T_{\mathcal{M}}(\bar{x}))}(\sigma x)$,
  which yields (recalling $\si\bx=\bx$ and the fact that the norm is symmetric)
  $$
  ||x-(\bar{x}+ z)||=||\sigma
  x-(\bar{x}+\sigma z)||<\mathrm{dist}_{(\bar{x}+T_{\mathcal{M}}(\bar
    {x}))}(\sigma x)
  $$
  contradicting the fact that $\sigma z\in T_{\mathcal{M}}(\bar{x})$.
  The reverse inequality can be established similarly.
\end{proof}

Let $\bar \pi_T\colon  \R^n \rightarrow \bar{x}+T_{\mathcal{M}}
(\bar{x})$ be the projection onto the affine space
$\bar{x}+T_{\mathcal{M}} (\bar{x})$, that is,
\begin{equation} \label{pi_T}
\bar{\pi}_T(x)=\mathrm{Proj\,}_{(\bar{x}+T_{\mathcal{M}}(\bar{x}))}(x),
\end{equation}
and similarly, let
\begin{equation}\label{pi_N}
\bar{\pi}_N(x)=\mathrm{Proj\,}_{(\bar{x}+N_{\mathcal{M}}(\bar{x}))}(x)
\end{equation}
denote the projection onto the affine space $\bx+N_{\Mm}(\bx)$.
We also introduce $\pi_{T}(\cdot)$ and $\pi_{N}(\cdot)$,
the projections onto the tangent and normal spaces
$T_{\mathcal{M}}(\bar x)$ and $N_{\mathcal{M}}(\bar x)$ respectively.
Notice the following relationships:
\begin{equation}\label{pi_TandN}
\bar{\pi}_T(x) + \bar{\pi}_N(x) = x + \bx \quad \mbox{and} \quad
\bar \pi_{T}(x) = \pi_{T}(x)+ \pi_{N}(\bx).
\end{equation}

\begin{corollary}[Invariance of projections]\label{corollary_pi}
  Let $\bar{x}\in\mathcal{M}$. Then, for all $x \in
  \R^n$ and all $\sigma\in S^{\succsim}(\bar x)$
  \begin{itemize}
  \item[(i)]  $\sigma \bar{\pi}_T(x)=\bar{\pi}_T(\sigma x),$
  \item[(ii)] $\sigma \bar{\pi}_N(x)=\bar{\pi}_N(\sigma x).$
  \end{itemize}
\end{corollary}

\begin{proof}
  Let $\bar{\pi}_T(x)=\bar{x}+u$ for some $u \in
  T_{\mathcal{M}}(\bar{x})$ and let $\sigma \in S^{\succsim}(\bar
  x)$. Since $\si\bx=\bx$, by Proposition~\ref{Proposition_distance},
  and the symmetry of the norm we obtain
  \[
  \mathrm{dist}_{(\bar x + T_{\mathcal{M}}(\bar{x}))}(x)
  = ||x-(\bar{x}+u)|| = ||\sigma x-(\bar{x}+\sigma u)||
  = \mathrm{dist}_{(\bar x + T_{\mathcal{M}}(\bar{x}))}(\sigma x).
  \]
  Since $\sigma u \in T_{\mathcal{M}}(\bar{x})$, we conclude
  $\bar{\pi}_T(\sigma x)=\bar x + \sigma u$ and assertion $(i)$ follows.

  Let us now prove
  the second assertion. Applying \eqref{pi_TandN} for the point
  $\si x\in\RR^n$, using $(i)$ and the fact that $\si \bx=\bx$ we deduce
  \begin{align*}
    \sigma x + \bx &= \bar{\pi}_T(\sigma x) + \bar{\pi}_N(\sigma x)
    = \sigma \bar{\pi}_T(x) + \bar{\pi}_N(\sigma x)\,.
  \end{align*}
  Applying $\sigma^{-1}$ to this equation, recalling that
  $\si^{-1}\bx=\bx$ and equating with \eqref{pi_TandN} we
  get~$(ii)$.
\end{proof}

The following result relates the tangent space to the
stratification.

\begin{proposition}[Tangential projection vs stratification]\label{Proposition_same_stratum}
  Let $\bar{x} \in\mathcal{M} \cap \Delta(\sigma)$. Then, there exists
  $\delta>0$ such that for any $x\in\mathcal{M}\cap B(\bar{x},\delta)$
  there exists $\sigma'\in S^{\succsim}(\sigma)$ such that
  \[
  x,\bar{\pi}_T(x) \in\Delta(\sigma').
  \]
\end{proposition}

\begin{proof}
  Choose $\delta > 0$ so that the ball $B(\bar{x},\delta)$ intersects
  only those strata $\Delta(\sigma')$ for which $\sigma' \in
  S^{\succsim}(\sigma)$ (see Lemma~\ref{key}$(ii)$). Shrinking
  $\delta>0$ further, if necessary, we may assume that the projection
  $\bar{\pi}_T$ is a one-to-one map between $\mathcal{M}\cap
  B(\bar{x},\delta)$ and its range. For any $x\in\mathcal{M}\cap
  B(\bar{x},\delta)$ let $u\in T_{\Mm}(\bar{x})\cap B(0,\,\delta)$ be
  the unique element of $T_{\mathcal{M}} (\bar{x})$ satisfying
  $\bar{\pi}_T(x)=\bar{x}+u$, or in other words such that
  \begin{equation}
    \mathrm{dist}_{\bar x + T_{\mathcal{M}}(\bar{x})}(x)
    =||x-(\bar{x}-u)||
    =\min_{z\in T_{\mathcal{M}}(\bar{x})}\,||(x-\bar{x})-z||\,.\label{fo}
  \end{equation}
  Then, for some $\sigma_{1},\sigma_{2}\in S^{\succsim}(\sigma)$ we
  have $\bar {x}+u\in\Delta(\sigma_{1})$ and $x\in\Delta(\sigma_{2})$.
  In view of Lemma~\ref{Tinv} and Lemma~\ref{key} we deduce
  \[
  \bx+\si_2 u=\sigma_{2}(\bar{x}+u)\in\left(
  \bar{x}+T_{\mathcal{M}}(\bar{x})\right) \cap B(\bar{x},\delta).
  \]
  We are going to show now that $\sigma_{1}\sim\sigma_{2}$.
  To this end, note first that
  \[
  ||x-(\bar{x}+\sigma_{2}u)||
  =||\sigma_{2}x-(\sigma_2\bar{x}+\sigma_{2}u)||
  =||(x-\bar{x})-u||.
  \]
  It follows from (\ref{fo}) that
  $\bar{\pi}_T(x)=\bar{x}+\sigma_{2}u,$ thus $\sigma _{2}u=u$, which
  yields $\si_2 (\bx+u)=\bx+u$, $\sigma_{1}\precsim\sigma_{2}$, by
  \eqref{partition_r}.  If we assume that $\sigma_{1}\prec\sigma_{2}$
  then $\sigma_1x \not= x$ (or else by \eqref{partition_r} $P(\si_1)\supseteq
  P(x)=P(\si_2)$ and $\si_1\succsim\si_2$, a contradiction).
  We have $\sigma_1 x \in \mathcal{M}\cap B(\bar{x},\delta)$, but
  $\si_1x \neq x$ yields $\bar{\pi}_T(x) \not= \bar{\pi}_T(\sigma_1x)$.
  Thus, there exists $v \in
  T_{\mathcal{M}}(\bar{x})$ with
  $$
  \|\sigma_1 x - (\bar x + v)\| < \|\sigma_1 x - (\bar x + u)\|  = \|x -(\bar x + u)\|,
  $$
  which contradicts Proposition~\ref{Proposition_distance}. Thus,
  $\si_1\sim\si_2$ and $x,\,\bx+u\,\in\,\Delta(\si_1)=\Delta(\si_2)$.
\end{proof}

We end this subsection by the following important property that
locates the tangent and normal spaces of $\Mm$ with respect to the
active stratum $\Delta(\si)$.

\begin{proposition}[Decomposition of $T_{\mathcal{M}}(\bar{x})$, $N_{\Mm}(\bx)$]\label{props-spectr-man}
  For any $\bar{x}\in\mathcal{M} \cap \Delta(\sigma)$ we have
  \[
  \mathrm{Proj\,}_{\Delta(\sigma)^{\bot\bot}}(T_{\mathcal{M}}(\bar{x}))=
  T_{\mathcal{M}}(\bar{x})\cap \Delta(\sigma)^{\bot\bot}
  \]
  which yields
  \begin{equation}\label{tang_decomp}
    T_{\mathcal{M}}(\bar{x})=(T_{\mathcal{M}}(\bar{x})\cap\Delta(\sigma)^{\bot\bot}
    )\oplus (T_{\mathcal{M}}(\bar{x}) \cap \Delta(\sigma)^{\bot}).
  \end{equation}
  Similarly,
  \begin{equation}\label{norm_decomp}
    N_{\mathcal{M}}(\bar{x})=(N_{\mathcal{M}}(\bar{x})\cap\Delta(\sigma)^{\bot\bot})
    \oplus (N_{\mathcal{M}}(\bar{x}) \cap \Delta(\sigma)^{\bot}).
  \end{equation}
\end{proposition}

\begin{proof}
  Lemma~\ref{Lemma_project_strata} and Lemma~\ref{Tinv} show that for any
  $u\in T_{\mathcal{M}}(\bar{x})$ we have
  \[
  \mathrm{Proj\,}_{\Delta(\sigma)^{\bot\bot}}(u)=\frac{1}{|S^{\succsim}(\sigma)|}
  \sum\limits_{\sigma^{\prime}\succsim\sigma} \sigma^{\prime}u \in
  T_{\mathcal{M}}(\bar{x}),
  \]
  which yields
  \[
  \mathrm{Proj\,}_{\Delta(\sigma)^{\bot\bot}}(T_{\mathcal{M}}(\bar{x}))
  \subseteq T_{\mathcal{M}}(\bar{x})\cap \Delta(\sigma)^{\bot\bot}.
  \]
  The opposite inclusion and decomposition (\ref{tang_decomp}) are
  straightforward.

  \smallskip

  Let us now prove the decomposition of $N_{\mathcal{M}}(\bar{x})$.
  For any $u \in T_{\mathcal{M}}(\bar{x})$, by (\ref{tang_decomp})
  there are (unique) vectors $u_{\bot} \in
  T_{\mathcal{M}}(\bar{x})\cap\Delta(\sigma)^{\bot}$ and $u_{\bot\bot}
  \in T_{\mathcal{M}}(\bar{x})\cap\Delta(\sigma)^{\bot\bot}$ such that
  $u=u_{\bot}+u_{\bot\bot}$. Since $\mathbb{R}^n =
  \Delta(\sigma)^{\bot} \oplus \Delta(\sigma)^{\bot\bot}$, we can
  decompose any $v \in N_{\mathcal{M}}(\bar{x})$ correspondingly as
  $v=v_{\bot} + v_{\bot\bot}$. Since $u_{\bot\bot},u_{\bot}\in
  T_{\Mm}(\bx)=N_{\Mm}(\bx)^{\bot}$ we have $\langle u_{\bot},v
  \rangle=0$ and $\langle u_{\bot\bot},v \rangle=0$. Using the fact
  that $\Delta(\sigma)^{\bot}$ and $\Delta(\sigma)^{\bot\bot}$ are
  orthogonal we get $\langle u_{\bot\bot},\,v_{\bot} \rangle=0$
  (respectively, $\langle u_{\bot},\,v_{\bot\bot} \rangle=0$) implying
  that $\langle u_{\bot\bot},v_{\bot\bot} \rangle=0$ (respectively,
  $\langle u_{\bot},v_{\bot} \rangle=0$), and finally $\langle
  u,\,v_{\bot}\rangle =0$ (respectively, $\langle
  u,\,v_{\bot\bot}\rangle=0$). Since $u \in T_{\mathcal{M}}(\bar{x})$
  has been chosen arbitrarily, we conclude $v_{\bot} \in
  N_{\mathcal{M}}(\bar{x}) \cap \Delta(\sigma)^{\bot}$ and
  $v_{\bot\bot} \in N_{\mathcal{M}}(\bar{x}) \cap
  \Delta(\sigma)^{\bot\bot}$. In other words,
  $N_{\mathcal{M}}(\bar{x})$ is equal to the (direct) sum of
  $N_{\mathcal{M}}(\bar{x}) \cap \Delta(\sigma)^{\bot}$ and
  $N_{\mathcal{M}}(\bar{x}) \cap \Delta(\sigma)^{\bot\bot}$.
\end{proof}

The following corollary is a simple consequence of the
fact that $T_{\mathcal{M}}(\bar{x}) \oplus N_{\mathcal{M}}(\bar{x})= \RR^n$.

\begin{corollary}[Decomposition of $\Delta(\si)^{\bot}$,
$\Delta(\si)^{\bot\bot}$]\label{Cor-23a} For any
$\bar{x}\in\mathcal{M} \cap \Delta(\sigma)$ we have
\begin{align*}
\Delta(\sigma)^{\bot}&=(\Delta(\sigma)^{\bot}\cap
  T_{\mathcal{M}}(\bar{x}))\oplus (\Delta(\sigma)^{\bot} \cap
  N_{\mathcal{M}}(\bar{x}))\\
\Delta(\sigma)^{\bot\bot}&=(\Delta(\sigma)^{\bot\bot}\cap
  T_{\mathcal{M}}(\bar{x}) )\oplus (\Delta(\sigma)^{\bot\bot} \cap
  N_{\mathcal{M}}(\bar{x})).
\end{align*}
\end{corollary}

The subspaces $\Delta(\sigma)^{\bot\bot} \cap
N_{\mathcal{M}}(\bar{x})$ and $T_{\mathcal{M}}(\bar{x}) \cap \Delta(\sigma)^{\bot}$ in the previous statements play an important role in \secref{sec:spectral} when
constructing adapted local equations.

%********************************************************************
\subsection{Location of a locally symmetric manifold}\label{ssec:structure}
%********************************************************************

Definition~\ref{defn-spectrman} yields
important structural properties on $\Mm$. These properties are
hereby quantified with the results of this section.

\smallskip

We need the following standard technical lemma about isometries
between two Riemannian manifolds. This lemma will be used in the
sequel as a link from local to global properties.  Given a Riemannian
manifold $M$ we recall that an open neighborhood $V$ of a point $p \in
M$ is called {\em normal} if every point of $V$ can be connected to
$p$ through a unique geodesic lying entirely in $V$. It is well-known
(see Theorem 3.7 in \cite[Chapter~3]{Carmo1993} for example) that
every point of a Riemannian manifold $\Mm$ (that is, $\Mm$ is at least
$C^2$) has a normal neighborhood. A more general version of
the following lemma can be found in \cite[Chapter~VI]{KobayashiNomizu1963},
we include its proof for completeness.

\begin{lemma}[Determination of isometries]\label{lem-isoms}
  Let $M$, $N$ be two connected Riemannian manifolds. Let $f_i\colon  M
  \rightarrow N$, $i\in\{1,2\}$ be two isometries and let $p \in M$ be
  such that
  $$f_1(p)=f_2(p) \qquad \mbox{and}\qquad df_1(v)=df_2(v) \quad
  \mbox{for every } \, v \in T_M(p)\,.$$ Then, $f_1=f_2$.
\end{lemma}

\begin{proof}
  Every isometry mapping between two Riemannian manifolds
  sends a geodesic into a geodesic. For any $p \in M$ and $v \in
  T_M(p)$, we denote by $\gamma_{v,p}$ (respectively by
  $\tilde\gamma_{\bar v, \bar p}$) the unique geodesic passing through
  $p\in M$ with velocity $v\in T_M(p)$ (respectively, through $\bar p
  \in N$ with velocity $\bar v \in T_N(\bar p)$). Using uniqueness of
  the geodesics, it is easy to see that for all $t$
  \begin{equation}\label{aris4}
    f_1( \gamma_{v,p}(t)) = \tilde\gamma_{df_1(v),f_1(p)}(t)
    =\tilde\gamma_{df_2(v),f_2(p)}(t) = f_2(\gamma_{v,p}(t)).
  \end{equation}
  Let $V$ be a normal neighborhood of $p$, let $q\in V$ and
  $[0,1]\backepsilon t \mapsto\gamma_{v,p}(t) \in M$ be the geodesic
  connecting $p$ to $q$ and having initial velocity $v \in T_M(p)$.
  Applying \eqref{aris4} for $t=1$ we obtain $f_1(q)=f_2(q)$. Since
  $q$ was arbitrarily chosen, we get $f_1=f_2$ on $V$. (Thus, since
  $V$ is open, we also deduce $df_1(v)=df_2(v)$ for every $v \in
  T_M(q)$.)

  \smallskip

  Let now $q$ be any point in $M$. Since connected manifolds are also
  path connected we can join $p$ to $q$ with a continuous path
  $t\in [0,1]  \mapsto \delta(t) \in M$. Consider the set
  \begin{equation}\label{aris5}
    \{t \in [0,1] : f_1(\delta(t))=f_2(\delta(t))\,\, \mbox{\rm and } \, df_1(v)=df_2(v)
    \,\mbox{\rm for every } v \in T_{M}(\delta(t))\}.
  \end{equation}
  Since $f_i\colon M\rightarrow N$ and $df_i: TM \rightarrow TN$
  ($i\in\{1,2\}$) are continuous maps, the above set is closed.
  further, since $f_1=f_2$ in a neighborhood of $p$ it follows that
  the supremum in \eqref{aris5}, denoted $t_0$, is strictly positive.
  If $t_0 \not=1$ then repeating the argument for the point
  $p_1=\delta(t_0)$, we obtain a contradiction. Thus, $t_0=1$ and
  $f_1(q)=f_2(q)$.
\end{proof}

The above lemma will now be used to obtain the following result
which locates the locally symmetric manifold $\Mm$ with respect to
the stratification.

\begin{corollary}[Reduction of the ambient space to $\Delta(\si)^{\bot\bot}$]\label{mc2}
  Let $\mathcal{M}$ be a locally symmetric manifold. If for some $\bar
  x \in \mathcal{M}$, $\sigma \in \Sigma^n$, and $\delta > 0$ we have
  $\mathcal{M} \cap B(\bar x, \delta) \subseteq \Delta(\sigma)$, then
  $\mathcal{M} \subseteq \Delta(\sigma)^{\bot\bot}$.
\end{corollary}

\begin{proof}
  Suppose first that $\Mm$ is strongly locally symmetric.
  Let $f_{1}\colon \mathcal{M}\rightarrow \mathcal{M}$ be the identity
  isometry on $\Mm$ and let $f_{2}\colon \mathcal{M}\rightarrow \mathcal{M}$
  be the isometry determined by the permutation $\sigma ,$ that is,
  $f_{2}(x)=\sigma x$ for all $x\in \mathcal{M}$. The assumption
  $\mathcal{M}\cap B(\bar{x},\delta )\subset \Delta (\sigma )$ yields
  that the isometries $f_{1}$ and $f_{2}$
  coincide around $\bar x$. Thus, by
  \lemref{lem-isoms} (with $ M=N=\mathcal{M}$) we conclude that $f_1$ and $f_2$ coincide on $\Mm$. This shows that $\Mm \subset \Delta(\sigma)^{\bot\bot}$.

  In the case when $\Mm$ is locally symmetric, assume, towards a contradiction, that there
  exists $\bar{ \bar{x}} \in \mathcal{M}\setminus \Delta (\sigma )^{\bot \bot }$
  Consider a continuous path
  $t\in [0,1]  \mapsto p(t) \in \Mm$ with $p(0)=\bar x$ and $p(1) = \bar{ \bar{x}}$. Find
  $0=t_0 < t_1 < \cdots < t_s = 1$ and $\{\delta_i > 0 : i=0,\ldots,s\}$ such that
  $\Mm_i:=\Mm \cap B(p(t_i), \delta_i)$ is strongly locally symmetric, the union of all $\Mm_i$
  covers the path $p(t)$, $\Mm_{i-1} \cap \Mm_{i} \not= \emptyset$, and $\Mm_0 \subset \Delta(\sigma)$. Let $s'$ be the first index such that $\Mm_{s'}
  \not\subset \Delta(\sigma)^{\bot\bot}$, clearly $s' >0$. Let $x' \in \Mm_{s'-1} \cap \Mm_{s'} \cap \Delta(\sigma)^{\bot\bot}$ and note that $x' \in \Delta(\sigma')$ for some $\sigma' \precsim \sigma$. By the strong local symmetry of $\Mm_{s'-1}$ and $\Mm_{s'}$,
 they are both invariant under the permutation $\sigma$. Since $\sigma$ coincides with the identity on $\Mm_{s'-1}$ and since $\Mm_{s'-1} \cap \Mm_{s'}$ is an open subset of $\Mm_{s'}$, we see by \lemref{lem-isoms} that $\sigma$ coincides with the identity on $\Mm_{s'}$. This contradicts the fact that  $\Mm_{s'}   \not\subset \Delta(\sigma)^{\bot\bot}$.
\end{proof}

In order to strengthen Corollary~\ref{mc2} we need to introduce a new notion.

\begin{definition}[{\it Much smaller} permutation]\label{def-smaller}
  For two permutations $\sigma,\sigma'\in\Sigma^n$.
  \begin{itemize}
  \item The permutation $\sigma'$ is called {\it much smaller} than
      $\sigma$, denoted $\sigma'\prec \hspace{-0.1cm} \prec \sigma$,
    whenever $\sigma' \prec \sigma$ and a set in $P(\sigma')$ is
    formed by merging at least two sets from $P(\sigma)$, of which
    at least one contains at least two elements.
  \item Whenever $\sigma'\prec\sigma$ but $\sigma'$ is not much
    smaller than $\sigma$ we shall write $\sigma'
    \prec \hspace{-0.1cm}\sim \sigma$. In other words, if $\sigma'\prec\sigma$ but $\sigma'$ is not much smaller than $\sigma$, then every set in $P(\sigma')$ that is not in
    $P(\sigma)$ is formed by merging one-element sets from $P(\sigma)$.
  \end{itemize}
\end{definition}

\begin{example}[{\it Smaller} vs {\it much smaller} permutations]\label{ex:ad}
  The following examples illustrate the notions of
  Definition~\ref{def-smaller}.  We point out that part
  (\ref{ex:ad-part-vii}) will be used frequently.
  \begin{enumerate}[(i)]
  \item \label{ex:ad-part-i}$(123)(45)(6)(7) \prec \hspace{-0.15cm} \prec (1)(23)(45)(6)(7)$.
  \item \label{ex:ad-part-ii} Consider $\si = (167)(23)(45)$ and $\si' =
    (1)(23)(45)(6)(7)$. In this case, $\si \prec \si' $ but $\si$ is
    not much smaller than $\si'$ because only cycles of length one are
    merged to form the cycles in $\si$. Thus,
    $\si\prec\hspace{-0,1cm}\sim\si'$.
  \item \label{ex:ad-part-iii} If $\sigma'' \precsim \sigma'$ and $\sigma'
    \prec \hspace{-0.1cm} \prec \sigma$ then $\sigma''
    \prec \hspace{-0.1cm} \prec \sigma$.
  \item \label{ex:ad-part-iv} It is possible to have $\sigma' \prec \hspace{-0.1cm}\sim
    \sigma$ and $\sigma'' \prec \hspace{-0.1cm}\sim\sigma$ but
    $\sigma'' \prec \hspace{-0.1cm} \prec \sigma'$, as shown by $\si =
    (1)(2)(3)(45)$, $\si' = (1)(23)(45)$, and $\si^{\prime\prime}=
    (123)(45)$.
  \item \label{ex:ad-part-v} If $\sigma' \prec \sigma$ and $\sigma$ fixes at most one
    element from $\N_n$, then $\sigma' \prec \hspace{-0.1cm} \prec
    \sigma$.
  \item \label{ex:ad-part-vi} If $\sigma \in \Sigma^n \setminus \mbox{\rm id}_n$ then
    $\sigma\prec \hspace{-0.1cm}\sim {\rm id}_n$.
  \item \label{ex:ad-part-vii} If $\sigma' \precsim \sigma$ and if $\sigma'$
    is not much smaller than $\sigma$, then either $\sigma' \sim
    \sigma$ or $\sigma' \prec \hspace{-0.1cm}\sim \sigma$.
   \item \label{ex:ad-part-viii} If $\sigma'' \prec \hspace{-0.1cm}\sim
    \sigma'$ and $\sigma' \prec \hspace{-0.1cm}\sim\sigma$, then $\sigma'' \prec \hspace{-0.1cm}\sim\sigma$. That is, the relationship `not much smaller' is transitive.  \qed
\end{enumerate}
\end{example}

We now describe a strengthening of Corollary~\ref{mc2}. It
lowers the number of strata that can intersect $\mathcal{M}$, hence
better specifies the location of the manifold $\mathcal{M}$.

\begin{corollary}[Inactive strata]
  \label{cor-20Nov2007-1} Let $\mathcal{M}$ be a locally symmetric
  manifold. If for some $\bar{x}\in \mathcal{M}$, $\sigma \in \Sigma
  ^{n}$ and $\delta >0$ we have $\mathcal{M}\cap B(\bar{x},\delta
  )\subseteq \Delta
  (\sigma )$ then%
  \begin{equation*}
    \mathcal{M}\subseteq \Delta (\sigma )^{\bot \bot }\setminus
    \bigcup_{\sigma ^{\prime }\prec \hspace{-0.1cm}\prec \sigma }\Delta
    (\sigma ^{\prime }).
  \end{equation*}
\end{corollary}

\begin{proof}
  By Corollary~\ref{mc2}, we already have $\mathcal{M}\subseteq
  \Delta (\sigma )^{\bot \bot }$. Assume, towards a contradiction,
  that $\mathcal{M}\cap \Delta (\sigma' )\neq \emptyset$ for some $\sigma' \prec
  \hspace{-0.1cm}\prec \sigma$. This implies in particular that $\sigma$ is not the
  identity permutation, see Example~\ref{ex:ad}\,(\ref{ex:ad-part-vi}).
  Consider a continuous path connecting $\bar x$ with a point in $\mathcal{M}\cap \Delta (\sigma' )\neq \emptyset$. Let $z$ be the first point on that path such that
  $z \in \Delta(\tau) \mbox{ for some } \tau \prec \hspace{-0.1cm}\prec \sigma$.
 (Such a first point exists since whenever $\tau \prec \sigma$, the points in $\Delta(\tau)$ are boundary points of $\Delta(\sigma)$.) Let $\delta >0$ be such that $\Mm \cap B(z, \delta)$ is strongly locally symmetric. Let $\bar z \in \Mm \cap B(z, \delta)$ be a point on the path before $z$. That means $\bar z$ is in a stratum $\Delta(\bar \sigma)$ with $\bar \sigma \prec \hspace{-0.1cm}\sim \sigma$ or $\bar \sigma \sim \sigma$.
 To summarize:
 $$
 z \in \Mm \cap \Delta(\tau),  \mbox{ where } \tau \prec \hspace{-0.1cm}\prec \sigma
 \mbox{ and } \bar z \in \Mm \cap \Delta(\bar \sigma) \cap B(z, \delta) \not= \emptyset, \mbox{ where } \bar \sigma \prec \hspace{-0.1cm}\sim \sigma  \mbox{ or } \bar \sigma \sim \sigma.
 $$

  By Definition~\ref{def-smaller} and the fact $\tau \prec \hspace{-0.1cm}\prec \sigma$,
  we have that for some $2\leq \ell <k\leq n$, and some
  subset $\{a_{1},\ldots ,a_{k}\}$ of
  $\mathbb{N}_{n},$ the cycle $(a_{1}\ldots a_{\ell })$ belongs to
  the cycle decomposition of $\sigma$ while the set $\{a_{1},\ldots
  ,a_{\ell },a_{\ell +1},\ldots, a_{k}\}$ belongs to the partition $P(\tau)$.
  Now, since $\bar \sigma \prec \hspace{-0.1cm}\sim \sigma$ or $\bar \sigma \sim \sigma$,
   the cycle~$(a_{1}\ldots a_{\ell })$ belongs to the cycle decomposition of $\bar \sigma$ as
   well. In order to simplify notation, without loss of generality, we
  assume that $a_{i}=i$ for $i\in \{1,\ldots ,k\}$.

  Since $\bar{z}=(\bar z_{1},\ldots ,\bar z_{n})\in \Mm \cap \Delta (\bar \sigma )
  \cap B(z, \delta)$ we have $\bar z_{1}=\cdots =\bar z_{\ell }=\alpha $ and $\bar
  z_{i}\neq \alpha $ for $i\in \{\ell +1,\ldots ,n\}$. By the fact that $\Mm \cap B(z, \delta)$ is
  strongly locally symmetric, we deduce that
  \begin{equation}\label{Hristo-fix}
    y:=\sigma_\circ \bar z\in \mathcal{M}\subset \Delta (\sigma)^{\bot
      \bot}\, \mbox{ for every } \si_\circ\succ\tau.
  \end{equation}
  We consider separately three cases. In each one we
  define appropriately a permutation $\sigma_\circ\succ\tau$ in order
  to obtain a contradiction with \eqref{Hristo-fix}.\medskip

  \textit{Case} 1.\hspace{0.10cm} Assume $\ell>2$ and let
  $\sigma_\circ \in\Sigma ^{n}$ be constructed by exchanging the places
  of the elements $a_{\ell }$ and $a_{k}$ in the cycle decomposition
  of $\sigma$. Obviously, $\sigma_\circ\succ\tau$. Then, $y=\sigma_\circ
  \bar z=(y_{1},\ldots ,y_{n})=(\bar z_{\sigma_\circ^{-1}(1)},\ldots
  ,\bar z_{\sigma_\circ^{-1}(n)})$ and notice that we have
  $y_{1}= \bar z_{\si_\circ^{-1}(1)} = \bar z_{k}\neq \alpha$, while
  $y_{2}=\bar z_{\sigma_\circ^{-1}(2)}=\bar z_{1}=\alpha$. In view of
  \eqref{eqn-arisd4} we deduce that $y\notin \Delta
  (\sigma)^{\bot\bot}$, a contradiction.

  \textit{Case} 2.\hspace{0.10cm} Let $\ell=2$ and suppose that $a_3\equiv 3$
  belongs to a cycle of length one in the cycle decomposition of $\sigma$ (recall that we
  have assumed $a_{i}=i$, for all $i\in\{1,\dots,k\}$). In other words,
  $\sigma \,=\, (1\, 2)(3)\sigma'$, where $\sigma'$ is a permutation of $\{4,\ldots,n\}$.
  Then, defining $\sigma_\circ \,:=\, (1\, 3)(2)\sigma'$ we get $y_1=\bar
  z_3\neq\alpha$ and $y_2=\bar z_2=\alpha$, thus again $y\notin \Delta
  (\sigma)^{\bot\bot}$. \smallskip

  \textit{Case} 3.\hspace{0.10cm} Let
  $\ell=2$ and suppose that $a_3\equiv 3$ belongs to a cycle of length at least
  two in the cycle decomposition of $\sigma$. Then, $\sigma \,=\, (1\,2)\,(3\,p\,\ldots)\cdots(\ldots\, q)\,\sigma'$,
  where $\si'$ is a permutation of $\{k+1,\dots,n\}$, and where the union of the elements in the cycles $(1\, 2)\,(3\, p\,\ldots)\cdots(\ldots\, q)$ is precisely $\{1,2,\ldots,k\}$. We define $\sigma_\circ\,=\, (1\, 2\, 3)\,(p\,\ldots)\cdots(\ldots\, q)\,\sigma' \succ\tau$ and obtain $y_1=\bar z_3\neq\alpha$ and $y_2=\bar z_1=\alpha$, thus again $y\notin \Delta (\sigma)^{\bot\bot}$.\smallskip

  The proof is complete.
\end{proof}

%********************************************************************
\subsection[The characteristic permutation $\si_{\ast}$ of $\Mm$]{The characteristic permutation \boldmath$\si_{\ast}$ of \boldmath$\Mm$}\label{ssec-simin}
%********************************************************************

In order to better understand the structure of the
lo lly symmetric manifold $\mathcal{M}$, we exhibit a
permutation (more precisely, a set of equivalent permutations) that is
characteristic of $\Mm$. To this end, we introduce the following
sets of active permutations. (These two sets will be used only in
this and the next subsections.) Define
\[
  \Delta (\mathcal{M}):=\{\sigma \in \Sigma ^{n}:\mathcal{M}\cap
  \Delta (\sigma )\not=\emptyset \},
\]
and
\[
  \Sigma _{\mathcal{M}}:=\{\sigma \in \Sigma ^{n}:\exists (\bar{x}\in
  \mathcal{M},\delta >0)\,\,\mbox{\rm such that}\,\,\mathcal{M}\cap
  B(\bar{x},\delta )\subseteq \Delta (\sigma )\}\,.
\]
 We note that if $\sigma \in
\Delta(\mathcal{M})$ then $\sigma' \in \Delta(\mathcal{M})$ whenever
$\sigma \sim \sigma'$, and similarly for $\Sigma _{\mathcal{M}}$.
The following result is straightforward.

\begin{lemma}[Maximality of $\Sigma _{\Mm}$ in $\Delta(\Mm)$]
  \label{lemma_maximal} The elements of $\Sigma_{\mathcal{M}}$ are
  equivalent to each other and maximal in $\Delta (\mathcal{M})$.
\end{lemma}

\begin{proof}
  It follows readily that $\Delta (\mathcal{M})\neq \emptyset $ and
  $\Sigma _{\mathcal{M}}\subset \Delta (\mathcal{M})$.  Let $\tau \in
  \Delta (\mathcal{M})$ and $\sigma \in \Sigma _{\mathcal{M}}$. By
  Corollary~\ref{mc2} we deduce that $\mathcal{M}\subset \Delta
  (\sigma )^{\bot \bot }$ and by Proposition~\ref{prop:delta}$(iii)$
  that $\tau \precsim \sigma$. This proves maximality of $\sigma $ in
  $\Delta (\mathcal{M})$. The
  equivalence of the elements of $\Sigma_{\Mm}$ is obvious.
\end{proof}

The next lemma is, in a sense, a converse of Corollary~\ref{mc2}. It shows in
particular that $\Sigma _{\mathcal{M}}\neq \emptyset$.

\begin{lemma}[Optimal reduction of the ambient space]
  \label{lem-27Nov2007-a} For a locally symmetric manifold
  $\mathcal{M}$, there exists a permutation $\sigma_{\ast }\in \Sigma
  ^{n}$, such that
  \begin{equation}
    \Sigma_{\mathcal{M}}=\{\sigma \in \Sigma^{n}:\,\sigma \sim
    \sigma_{\ast }\}\,. \label{sigma_min}
  \end{equation}
  In particular, if $\mathcal{M}\subseteq \Delta (\bar{\sigma})^{\bot
    \bot}$ for some $\bar{\sigma}\in \Sigma^{n}$ then $\sigma_{\ast
  }\precsim \bar{\sigma}$.
\end{lemma}

\begin{proof}
  Assertion \eqref{sigma_min} follows directly from
  Lemma~\ref{lemma_maximal} provided one proves that
  $\Sigma_{\mathcal{M}}\neq \emptyset$. To do so, we assume that
  $\mathcal{M}\subseteq \Delta (\bar{\sigma})^{\bot \bot}$ for some
  $\bar{\sigma}\in \Sigma ^{n}$ (this is always true for
  $\bar{\sigma}=\mathrm{id}_{n}$) and we prove both that
  $\Sigma_{\Mm}\neq\emptyset$ as well as the second part of the
  assertion. Notice that $\sigma \precsim \bar{\sigma}$ for
  all $\sigma \in \Delta (\mathcal{M})$. Let us denote by $\sigma
  ^{\circ }:=\bigvee \Delta (\mathcal{M})$ any supremum of the
  nonempty set $\Delta (\mathcal{M})$ (that is, any permutation
  $\sigma ^{\circ }$ whose partition is the supremum of the partitions
  $P(\sigma )$ for all $\sigma \in \Delta (\mathcal{M})$). If $\sigma
  ^{\circ }\in \Delta (\mathcal{M}),$ then $\sigma ^{\circ }\in \Sigma
  _{\mathcal{M}}$, $\sigma ^{\circ }=\sigma _{\ast }$ and we are done. If $\sigma
  ^{\circ }\notin \Delta (\mathcal{M})$, then choose any permutation
  $\sigma _{\circ}\in \Delta (\mathcal{M})$ such that
  \begin{equation}
    \{\sigma \in \Delta (\mathcal{M}):\sigma^{\circ }\succ \sigma \succ
    \sigma _{\circ }\}=\emptyset.   \label{empty-set}
  \end{equation}
  Such a permutation $\sigma _{\circ }$ exists since
  $\Delta(\mathcal{M})$ is a finite partially ordered set. By the
  definition~of~$\sigma _{\circ }$ there exists $\bar{x}\in
  \mathcal{M}\cap \Delta (\sigma _{\circ })$, and by
  \lemref{key}$(\ref{key_pt2})$ we can find $\delta >0$ such that
  $B(\bar{x},\delta )$ intersects only strata $\Delta (\sigma )$
  corresponding to permutations $\sigma \succsim \sigma _{\circ }$. If
  there exists $x\in \mathcal{M}\cap B(\bar{x},\delta )$ such that
  $x\in \Delta (\sigma)$ for some permutation $\sigma \succ
  \sigma_{\circ}$, then $\sigma \in \Delta (\mathcal{M})$ and by
  (\ref{empty-set}) $\sigma \sim \sigma ^{\circ }$ contradicting the
  assumption that $\sigma ^{\circ }\notin \Delta (\mathcal{M}).$ Thus,
  $\mathcal{M}\cap B(\bar{x},\delta )\subseteq \Delta (\sigma_{\circ
  })$ and $\sigma_{\circ }=\sigma_{\ast }\in \Sigma_{\mathcal{M}}$.
\end{proof}

\begin{corollary}[Density of $\mathcal{M} \cap \Delta(\sigma_*)$ in $\mathcal{M}$]
  \label{density-hss}
  For every $\bar x \in \mathcal{M}$, every $\delta > 0$ and $\sigma_*
  \in \Sigma_{\mathcal{M}}$, we have
  $$
  \mathcal{M}  \cap \Delta(\sigma_*)\cap B(\bar x, \delta) \not= \emptyset.
  $$
\end{corollary}

\begin{proof}
  Suppose $\bar x \in \mathcal{M} \cap \Delta(\sigma)$ and fix $\delta
  > 0$ small enough so that $B(\bar x, \delta)$ intersects only strata
  $\Delta(\sigma')$ for $\sigma' \succsim \sigma$. Then, by
  Lemma~\ref{key}, we have that the manifold
  $\mathcal{M}':=\mathcal{M} \cap B(\bar x, \delta)$ is locally
  symmetric.  By Lemma~\ref{lem-27Nov2007-a}, we obtain that
  $\Sigma_{\mathcal{M}'} \not= \emptyset$. Since
  $\Sigma_{\mathcal{M}'} \subset \Sigma_{\mathcal{M}}$, and
  all permutations in $\Sigma_{\mathcal{M}}$
  are equivalent, we have $\Sigma_{\Mm'} = \Sigma_{\Mm}$.
  Thus, $\Mm'\cap B(\bar y, \rho)\subset \De(\si_*)$
  for $\bar y\in \Mm'\subset \Mm$ and some $\rho>0$, whence the result follows.
\end{proof}

Clearly, if ${\rm id}_n \in \Sigma_{\mathcal{M}}$, then
$\Sigma_{\mathcal{M}}=\{{\rm id}_n\}$. In particular, we have the
following easy result.

\begin{corollary}\label{Ithoughtthiswasover}
  For a locally symmetric manifold $\mathcal{M} \subset \R^n$, we have
  $$
  \sigma_* = {\rm id}_n\,\,\, \Longleftrightarrow \,\,\,
  \mathcal{M} \cap \Delta({\rm id}_n) \not= \emptyset.
  $$
\end{corollary}

\begin{proof}
  The necessity is obvious, while the sufficiency follows from
  Lemma~\ref{lemma_maximal}, since ${\rm id}_n \in \Delta (\mathcal{M})$
  is the unique maximal element of $\Sigma^n$.
\end{proof}

Thus, the permutation $\sigma_{\ast}$ is naturally associated with the locally
symmetric manifold $\Mm$ via the property
\begin{equation}
  \label{eqn-27Nov2007-d}
  \exists (\bar x \in \mathcal{M}, \delta >0) \,\, \mbox{\rm such that} \,\,
  \mathcal{M} \cap B(\bar x, \delta) \subseteq
  \Delta(\sigma_{\ast}).
\end{equation}
Notice that $\si_{\ast}$ is unique modulo $\sim$, and will be called
{\it characteristic} permutation of $\Mm$.  Even though the
definition of the characteristic permutation $\sigma_{\ast}$ is
local, it has global properties stemming from
Corollary~\ref{cor-20Nov2007-1}, that is,
\begin{equation} \label{eqn-28Nov2007-1}
  \mathcal{M}\  \subseteq\
  \Delta(\sigma_{\ast})^{\bot\bot} \setminus \!\bigcup_{\sigma \prec \hspace{-0.1cm} \prec \sigma_{\ast}} \Delta(\sigma)\
  = \bigcup_{{\scriptsize   \begin{array}{c}\sigma \sim \si_{\ast}\\[-0.5ex] \si\prec  \hspace{-0.1cm} \sim \sigma_{\ast}\end{array}}} \hspace*{-1ex}\Delta(\sigma)
  \ \subseteq\  \Delta(\sigma_{\ast})^{\bot\bot}\,,
\end{equation}
and $\sigma_*$ is the minimal permutation for which
(\ref{eqn-28Nov2007-1}) holds.  The above formula determines precisely
which strata can intersect $\mathcal{M}$. Indeed, if $ \sigma \in
\Delta(\mathcal{M})$ then necessarily either $\sigma \sim
\sigma_{\ast}$ or $\sigma \prec\hspace{-0.1cm}\sim
\sigma_{\ast}$. Notice also that when $\sigma \prec\hspace{-0.1cm}\sim
\sigma_{\ast}$, every set in $P(\sigma)$, which is not in
$P(\sigma_{\ast})$, is obtained by merging sets of length one from
$P(\sigma_{\ast})$. Another consequence is the following relation:
\begin{equation}
  \label{32}
  T_{\Mm}(\bx)\subset\Delta(\si_{\ast})^{\bot\bot}\qquad\text{for all }
  \bx\in\Mm\,.
\end{equation}

\begin{remark}
\label{rem-8Apr2009}
Observe that for any fixed permutation $\sigma_* \in
\Sigma^n$, the set
$$
\bigcup_{{\scriptsize \begin{array}{c}\sigma \sim \si_{\ast}\\[-0.5ex] \si\prec  \hspace{-0.1cm} \sim \sigma_{\ast}\end{array}}} \hspace*{-1ex}\Delta(\sigma)
$$
is a locally symmetric manifold with characteristic permutation
$\sigma_*$. On the other hand, \eqref{eqn-28Nov2007-1} shows that
the affine space $\Delta(\sigma)^{\bot\bot}$ is a locally symmetric
manifold if (and only if) $\sigma \in \Sigma^n$ is equal to
$\mbox{id}_n$ or is a cycle of length $n$.  \qed
\end{remark}

We conclude with another fact about the characteristic permutation,
that stems from the assumption $\mathcal{M} \cap \R^n_{\ge} \not=
\emptyset$ (see \defref{defn-spectrman}). Though
(\ref{eqn-28Nov2007-1}) describes well the strata that can intersect
the manifold $\mathcal{M}$ (which is going to be sufficient for most
of our needs) we still need to say more about a slightly finer issue
- a necessary condition for a stratum to intersect $\mathcal{M} \cap
\R^n_{\ge}$.

\begin{lemma}\label{consec}
  Suppose that $\bar x \in \mathcal{M} \cap \R^n_{\ge} \cap \Delta(\sigma)$.
  Then, every set $I_i$ of the partition
  $$
P(\sigma) = \{I_1,\ldots,I_{\kappa+m}\}
  $$
  contains consecutive integers from $\N_n$.
\end{lemma}

\begin{proof}
  The lemma is trivially true, for sets $I_i$ with cardinality one.
  So, suppose on the contrary, that for some $\ell \in
  \{1,\ldots,\kappa+m\}$, the set $I_\ell$ contains at least two elements
  but does not contain consecutive numbers from $\N_n$. That
  is, there are three indexes $i,j,k \in \N_n$ with $i < j < k$ such
  that $i,k \in I_\ell$ but $j \not\in I_\ell$. Then, the fact $\bar x
  \in \Delta(\sigma)$ implies that $\bar x_i = \bar x_k$, while the
  fact that $\bar x \in \R^n_{\ge}$ implies that $\bar x_i \ge \bar
  x_j \ge \bar x_k$. We obtain $\bar x_i = \bar x_j = \bar x_k$, which
  contradicts the assumption $j \not\in I_\ell$.
\end{proof}

Lemma~\ref{consec} has consequences for the characteristic permutation
$\sigma_*$ of $\Mm$.

\begin{theorem}[Characteristic partition $P(\si_*)$]\label{consec-sets}
  Every set in the partition $P(\sigma_*)$
  contains consecutive integers from $\N_n$.
\end{theorem}

\begin{proof}
  Let $\si_{\ast}\in\Sigma_{\Mm}$ be the characteristic permutation of
  $\Mm$.  Since $\mathcal{M} \cap \R^n_{\ge} \not= \emptyset$ by
  Definition~\ref{defn-spectrman}, there is a stratum
  $\Delta(\sigma)$ intersecting $\mathcal{M} \cap \R^n_{\ge}$.
  Formula~(\ref{eqn-28Nov2007-1}) implies that $\sigma$ is not much
  smaller than $\sigma_*$, {\it i.e.} we have $\sigma \sim
  \sigma_{\ast}$ or $\sigma \prec\hspace{-0.1cm}\sim \sigma_{\ast}$.
  If a set $I^*_i \in P(\sigma_*)$ has more than one element, then it
  must be an element of the partition $P(\sigma)$ as well, by the fact
  that $\sigma$ is not much smaller than $\sigma_*$. Thus, $I^*_i$ contains
  consecutive elements from $\N_n$, by Lemma~\ref{consec}.
\end{proof}

For example, according to Theorem~\ref{consec-sets}, the permutation
$(1)(274)(35)(6) \in \Sigma^7$ cannot be the characteristic
permutation of any locally symmetric manifold $\mathcal{M}$ in
$\R^7$ (that intersects $\R^7_{\ge}$).

Let us illustrate the limitations imposed by the previous result.
Suppose that $n=12$ and the partition $P(\si_*)$ of $\N_{12}$ corresponding to
$\sigma_* \in \Sigma^{12}$ is
$$
P(\sigma_*) = \{ \{1\}, \{2\}, \{3, 4, 5\}, \{6\}, \{7\}, \{8\}, \{9\}, \{10, 11, 12\}\}.
$$
Pick a permutation $\sigma \in \Sigma^{12}$ with partition
$$
P(\sigma) = \{\{1\}, \{2\}, \{3,4,5\}, \{6,8,9\}, \{7\}, \{10, 11,12\}\}.
$$
In comparison with Formula~(\ref{eqn-28Nov2007-1}), $\sigma$ is not
much smaller than $\sigma_*$ but the stratum $\Delta(\sigma)$ does not
intersect $\mathcal{M} \cap \R^n_{\ge}$. Thus, the set of strata that
may intersect with $\mathcal{M} \cap \R^n_{\ge}$ is further reduced.

%***************************************************************************
\subsection[Canonical decomposition induced by $\sigma_{\ast}$]{Canonical decomposition induced by \boldmath$\sigma_{\ast}$}\label{sssec-341}
%***************************************************************************

We explain in this subsection that the characteristic permutation
$\sigma_{\ast}$ of $\Mm$ induces a decomposition of the space
$\RR^n$ that will be used later to control the lift into the matrix
space $\Sn$. We consider the partition $P(\sigma_*)$ of $\N_n$
associated with $\si_*$, and we define
\begin{equation}\label{aris-m}
  m_{\ast}:=
  \mbox{ number of sets in $P(\sigma_{\ast})$ that have more than one element},
\end{equation}
and
\begin{equation}\label{aris-k}
  \kmin:=\mbox{ number of sets in $P(\sigma_{\ast})$ with exactly one element.}
\end{equation}
In other words, $\kmin$ is the number of elements of $\N_n$ that are
fixed by the permutation $\sigma_{\ast}$, or equivalently,
$\kmin:=|\mathbb{N}_{n}\setminus\mathrm{supp}(\sigma_{\ast})|$.
Hence, we have
\begin{align}\label{part-of-sigmamin}
  P(\sigma_{\ast})
  :=\{I_1^{\ast},\ldots,I_\kmin^{\ast}, I_{\kmin+1}^{\ast},\ldots,I_{\kmin+m_{\ast}}^{\ast}\},
\end{align}
where $\{I_1^{\ast},\ldots,I_\kmin^{\ast}\}$ are the blocks of size one.
The following example treats the particular case where $\si_{\ast}$
has at most one cycle of length one.

\begin{example}[Case: $\kmin=0$ or $1$]\label{ex-k01}
  The assumption $\kmin\in\{0,1\}$ means that the permutation~$\sigma_{\ast}$ fixes at most one element, or in other words, for
  every $x\in\Mm$ at most one coordinate of the vector
  $x=(x_1,\dots,x_n)$ is not repeated. In this case, by
  Example~\ref{ex:ad}(\ref{ex:ad-part-v}), every $\sigma$ that is smaller than
  $\sigma_*$ is much smaller than $\sigma_*$ and therefore
  \eqref{eqn-28Nov2007-1} together with
  Proposition~\ref{prop:delta}$(iii)$ yields $\Mm \subset
  \Delta(\sigma_{\ast})$.~\qed
\end{example}

The partition of the characteristic permutation $\sigma_{\ast}$ of
$\Mm$ yields a {\it canonical split} of $\RR^n$ associated to $\Mm$,
as a direct sum of two parts, the spaces $\RR^{\kappa_{\ast}}$ and
$\RR^{n-\kmin}$, as follows: any vector $x\in\RR^n$ is represented as
\begin{equation}
\label{eqn-arisd1}
x = x^F \otimes x^M
\end{equation}
where
\begin{itemize}
\item $x^F \in \R^{\kmin}$ is the subvector of $x \in \R^n$ obtained
  by collecting from $x$ the coordinates that have indices in
  $\mathbb{N}_{n}\setminus\mathrm{supp}(\sigma_{\ast})$ and preserving
  their relative order;
\item $x^M \in \R^{n-\kmin}$ is the subvector of $x \in \R^n$ obtained
  by collecting from $x$ the remaining $n-\kmin$ coordinates,
  preserving their order again.
\end{itemize}
It is readily seen that the canonical split is linear and also a
reversible operation. Reversibility means that given any two
vectors $x^F \in \R^{\kappa_*}$ and $x^M \in \R^{n-\kappa_*}$, there
is a unique vector $x^F \otimes x^M \in \R^n$, such that
$$
(x^F \otimes x^M)^F = x^F \quad \mbox{ and } \quad (x^F \otimes x^M)^M = x^M.
$$
This operation is called {\it canonical product}.

\begin{example}\label{example-order}
  If $\sigma_* = (1)(23)(4)(567)(8) \in \Sigma^8$ and $x \in \R^8$
  then, $x^F = (x_1,x_4,x_8)$ and $x^M = (x_2,x_3,x_5,x_6,x_7)$.
  Conversely, if
  $$
  x^F=(a_1,a_2,a_3) \qquad \mbox{and} \qquad x^M=(b_1,b_2,b_3,b_4,b_5)
  $$
  then
  $$
  x^F \otimes x^M = (a_1,b_1,b_2,a_2, b_3,b_4,b_5, a_3).
  $$ In addition, if $x \in \R^8_{\ge}$ then $x^F \in \R^3_{\ge}$ and
  $x^M \in \R^5_{\ge}$, but the converse is not true: if $x^F \in
  \R^3_{\ge}$ and $x^M \in \R^5_{\ge}$ then in general, $x^F \otimes
  x^M$ is not in $\R^8_{\ge}$.\qed
\end{example}

\medskip

Furthermore, if $\si\in\Sigma^n$ is any permutation whose cycles do not
contain elements simultaneously from $\mathrm{supp}(\sigma_*)$ and
$\mathbb{N}_n\setminus\mathrm{supp}(\sigma_*)$, then it can be decomposed as
\begin{equation}\label{eqn-arisd2}
  \sigma = \sigma^F \circ \sigma^M,
\end{equation}
where
\begin{itemize}
\item $\sigma^F\in\Sigma^\kmin$ is obtained by those cycles of
  $\sigma$ that contain only elements from
  $\mathbb{N}_{n}\setminus\mathrm{supp}(\sigma_{\ast})$,
\item $\sigma^M\in\Sigma^{n-\kmin}$ is obtained from the remaining
  cycles of $\sigma$ (those that do not contain any element of
  $\mathbb{N}_{n}\setminus\mathrm{supp}(\sigma_{\ast})$).
\end{itemize}

Observe that $\si$ is the infimum of $\si^F$ and $\si^M$ ($\si
=\si ^{F}\wedge \si ^{M}$). We refer to (\ref{eqn-arisd2}) as
the {\it $(F,M)$-decomposition} of the permutation~$\si$.  For
example, applying this decomposition to $\si_{\ast}$ yields
\begin{equation}\label{eqn-arisd5}
\si_{\ast}^F=\mathrm{id}_{\kmin},
\end{equation}
where $\mathrm{id}_{\kmin}$ is the identity permutation on the set
$\mathbb{N}_{n}\setminus\mathrm{supp}(\sigma_{\ast})$.  Note that in
the particular case~$\kmin=n$, we have $\sigma_{\ast}=\id_n$, all
coefficients of $x\in\Delta(\si_{\ast})$ are different, and $x=x^F$.

The following proposition is a straightforward consequence of
\eqref{eqn-arisd5} and Example~\ref{ex:ad}(\ref{ex:ad-part-v}).

\begin{proposition}[$(F,M)$-decomposition for $\sigma \prec\hspace{-0.1cm}\sim \sigma_{\ast}$]\label{prop-arisd1}
  The following equivalences hold:
  \[
  \sigma \sim \sigma_{\ast}
  \ \iff \ \sigma^F =\mathrm{id}_\kmin
  \qandq \sigma^M \sim {\sigma_{\ast}}^M
  \]
  and
  \[
  \sigma \prec \hspace{-0.1cm}\sim \sigma_{\ast}
  \ \iff \ \sigma^F \prec \mathrm{id}_\kmin
    \qandq \sigma^M \sim {\sigma_{\ast}}^M.
    \]
\end{proposition}

Note that the $(F,M)$-decomposition is not going to be applied to
permutations $\si\in\Sigma^n$ that are much smaller than $\si_{\ast}$,
since these permutations may have a cycle containing elements from
both $\mathrm{supp}(\sigma_{\ast})$ and
$\mathbb{N}_{n}\setminus\mathrm{supp}(\sigma_{\ast})$. In fact,
\eqref{eqn-arisd2} can be applied only to permutations $\tau \in
S^{\succsim}(\sigma)$ with $\si \in \Delta(\Mm)$, as explained in the
following result, whose proof is straightforward.

\begin{proposition}[$(F,M)$-decomposition for active permutations]\label{prop-arisd2}
  Let $\si\in \Delta(\Mm)$ and $\tau\in S^{\succsim}(\sigma)$.  Then,
  $\tau$ admits $(F,M)$-decomposition $\tau=\tau^F\circ \tau^M$
  given in \eqref{eqn-arisd2} with
  $$
  \si^F \precsim \tau^F \precsim \mathrm{id}_{\kmin}
  \qquad\text{and}\qquad
  \si_{\ast}^M\sim \si^M \precsim \tau^M.
  $$
\end{proposition}

%***************************************************************************
\subsection{Reduction of the normal space} \label{sssec-342}
%*****************************************************************************

In this section we fix a point $\bar x$ and a permutation $\si$ such
that $\bx\in \Mm\cap\Delta (\si)$, and reduce the relevant (active) part of the
tangent and normal space with respect to the canonical split
\begin{equation}
  \RR^{n}=\RR^{\kmin}\otimes \RR^{n-\kmin}  \label{eqn-arisd7}
\end{equation}
induced by the characteristic permutation $\si_{\ast }$ of $\Mm$.

Let us consider any permutation $\tau \in \Sigma ^{n}$ for which the
decomposition (\ref{eqn-arisd2})
\begin{equation*}
\tau =\tau ^{F}\circ \tau ^{M}
\end{equation*}
makes sense (that is, $\tau \in S^{\succsim }(\sigma )$, where $\sigma
\sim \si_{\ast }$ or $\sigma \prec \hspace{-0.1cm}\sim \si_{\ast
}$). Then, we can either consider $ \tau ^{F}$ as an element of
$\Sigma ^{n}$ (giving rise to a stratum $\Delta (\tau ^{F})\subset
\RR^{n}$) or as an element of $\Sigma ^{\kappa _{\ast }}$ (acting on
the space $\mathbb{R}^{\kappa _{\ast }}$). In this latter case, and in
other to avoid ambiguities, we introduce the notation
\begin{equation}
\lbrack \Delta (\tau ^{F})_{\mathbb{R}^{\kappa _{\ast }}}]:=\{z\in
\mathbb{R} ^{\kappa _{\ast }}:P(z)=P(\tau ^{F})\} \label{eqn-arisd6}
\end{equation}
to refer to the corresponding stratum of $\mathbb{R}^{\kappa _{\ast
}}$. The notations $[\Delta (\tau ^{F})_{\mathbb{R}^{\kappa _{\ast
}}}]^{\bot }$, $ [\Delta (\tau ^{F})_{\mathbb{R}^{\kappa _{\ast
}}}]^{\bot \bot }$ refer thus to the corresponding linear subspaces
of $\mathbb{R}^{\kappa _{\ast }}$. We do the same for the stratum
$[\Delta (\tau ^{M})_{\mathbb{R}^{n-\kappa _{\ast }}}]$ (and the
linear subspaces $[\Delta (\tau ^{M})_{\mathbb{R}^{n-\kappa _{\ast
}}}]^{\bot }$, $[\Delta (\tau ^{M})_{\mathbb{R}^{n-\kappa _{\ast
}}}]^{\bot \bot }$), whenever the permutation $\tau ^{M}$ is
considered as an element of $\Sigma ^{n-\kappa _{\ast }}$ acting on
$\mathbb{R}^{n-\kappa _{\ast }}$. A careful glance at the formulas
(\ref{eqn-arisd3}) and (\ref{eqn-arisd4}) reveals the following
relations:
\begin{equation}
\Delta (\tau ^{F})^{\bot }=[\Delta (\tau ^{F})_{\mathbb{R}^{\kappa
_{\ast }}}]^{\bot }\otimes \{0\}_{n-\kappa _{\ast }}\qquad
\text{and}\qquad \Delta (\tau ^{M})^{\bot }=\{0\}_{\kappa _{\ast
}}\otimes \lbrack \Delta (\tau ^{M})_{\mathbb{R}^{n-\kappa _{\ast
}}}]^{\bot }\text{ ;}  \label{eqn-arisd10}
\end{equation}
and respectively,
\begin{equation} \Delta (\tau ^{F})^{\bot \bot
}=[\Delta (\tau ^{F})_{\mathbb{R}^{\kappa _{\ast }}}]^{\bot \bot
}\otimes \mathbb{R}^{n-\kappa _{\ast }}\qquad \text{ and}\qquad
\Delta (\tau ^{M})^{\bot \bot }=\mathbb{R}^{\kappa _{\ast }}\otimes
\lbrack \Delta (\tau ^{M})_{\mathbb{R}^{n-\kappa _{\ast }}}]^{\bot
\bot }\text{ .}  \label{eqn-arisd9}
\end{equation}
It follows easily from (\ref{Ph4}) and (\ref{eqn-arisd9}) that
\begin{equation}
\Delta (\tau )^{\bot \bot }=[\Delta (\tau ^{F})_{\mathbb{R}^{\kappa
_{\ast }}}]^{\bot \bot }\otimes \lbrack \Delta (\tau
^{M})_{\mathbb{R}^{n-\kappa _{\ast }}}]^{\bot \bot }.
\label{eqn-arisd11}
\end{equation}
It also follows easily that
\begin{equation}\label{eqn-arisd12}
\Delta (\tau )^{\bot }=[\Delta (\tau ^{F})_{\mathbb{R}^{\kappa
_{\ast }}}]^{\bot }\otimes \lbrack \Delta (\tau
^{M})_{\mathbb{R}^{n-\kappa _{\ast }}}]^{\bot }.
\end{equation}

In the sequel, we apply the canonical split \eqref{eqn-arisd7} to
the tangent space $T_{\mathcal{M}}(\bar{x})$. In view of \eqref{32}
and \eqref{eqn-arisd11} for $\tau=\si_{\ast}$ and the fact that
$\si_{\ast}^M\sim\si^M$ (see Proposition~\ref{prop-arisd1}),
we obtain that for every $w\in T_{\mathcal{M}}(\bar{x})$
\begin{equation}
w=w^{F}\otimes w^{M}\quad \text{ where }\quad w^{F}\in
\mathbb{R}^{\kmin} \qandq w^{M}\in \lbrack \Delta
(\sigma^M)_{\mathbb{R}^{n-\kappa _{\ast }}}]^{\bot \bot }\,\subset
\mathbb{R}^{n-\kappa _{\ast }}, \label{arisd10}
\end{equation}
where each coordinate of $w^{M}$ is repeated at least twice.

The following theorem reveals a analogous relationship
for the canonical split of the normal space
$N_{\Mm}(\bx)$ of $\mathcal{M}$ at $\bar{x}$. It is the culmination of
most of the developments up to now and thus the most important
auxiliary result in this work. We start by a technical result.

\begin{lemma}\label{lem-hss}
  Let $\bar x \in \mathcal{M} \cap \Delta(\sigma)$ and let the
  $(F,M)$-decomposition of $\sigma$ be $\sigma =\sigma ^{F}\circ
  \sigma ^{M}$. Let the partition of $\N_{\kappa_*}$ defined by
  $\sigma^{F}$ be $P(\sigma^{F}) = \{I_1,\ldots,I_m\}$. Then, for every
  $\epsilon > 0$, there exists $w\in T_{\mathcal{M}}(\bar{x}) \cap
  B(0, \epsilon)$, such that in vector $w^{F} \in \mathbb{R} ^{\kappa
    _{\ast }}$ every subvector $w^{F}_{I_i}$ has distinct coordinates,
  for $i \in \N_m$.
\end{lemma}

\begin{proof}
 % We apply Proposition~\ref{Proposition_same_stratum} to $\bar x$.
 By Corollary~\ref{density-hss}, we can chose $x
  \in \Mm\cap \Delta (\si_{\ast })$ arbitrarily close to $\bar x$. Now
  apply Proposition~\ref{Proposition_same_stratum} to $\bar x$ and $x$
   to conclude that  $x,  \bar \pi_{T}(x) \in \Delta(\sigma')$ for some $\sigma' \succsim \sigma$. Necessarily, we have $\sigma' \sim \sigma_*$, implying that
   $x,  \bar \pi_{T}(x) \in \Delta(\sigma_*)$. This shows that $(\bar
  \pi_{T}(x))^F$ has distinct coordinates. In other words, there is a
  vector $w \in T_{\mathcal{M}}(\bar{x})$ such that $(\bar
  \pi_{T}(x))^F = (\bar x + w)^F = \bar{x}^F + w^F$ has distinct
  coordinates. Since $x$ can be chosen arbitrarily close to $\bar x$,
  we can assume that $w$ is arbitrarily close to $0$. Finally, since
  $\bar{x}^F \in [\Delta(\sigma^F)_{\R^{\kappa_*}}]$ and $w^F =
  (\bar{x}^F + w^F) - \bar{x}^F$ we conclude that $w^{F}_{I_i}$ has
  distinct coordinates, for $i \in \N_m$.
\end{proof}

\begin{theorem}[Reduction of the normal space]\label{jm-double}
  Let $\bar{x}\in \mathcal{M}\cap \Delta (\sigma
  )$ and $v\in N_{\mathcal{M}}(\bar{x})$. Let $v=v^{F}\otimes
  v^{M}$ and $\sigma =\sigma ^{F}\circ \sigma ^{M}$ be the canonical
  split and the $(F,M)$-decomposition defined in \eqref{eqn-arisd1}
  and \eqref{eqn-arisd2} respectively. Then,
  \begin{equation}
    v^{F}\in \lbrack \Delta (\sigma ^{F})_{\mathbb{R}^{\kappa _{\ast
          }}}]^{\bot \bot }\text{.}  \label{prop-30June07-b}
  \end{equation}
\end{theorem}

\begin{proof}
  Let us decompose $v\in N_{\mathcal{M}}(\bar{x})$
  according to Proposition~\ref{props-spectr-man}, that is, $v=v_{\bot
    \bot }+v_{\bot }$ where
  \begin{equation*}
    v_{\bot \bot }\in N_{\mathcal{M}}(\bar{x})\cap \Delta (\sigma
    )^{\bot \bot } \text{\quad and}\quad v_{\bot }\in
    N_{\mathcal{M}}(\bar{x})\cap \Delta (\sigma )^{\bot }.
  \end{equation*}
  Then,
  \begin{equation*}
    v^{F}=v_{\bot \bot }^{F}+v_{\bot }^{F}\text{\quad and}\quad
    v^{M}=v_{\bot \bot }^{M}+v_{\bot }^{M}.
  \end{equation*}
  Since $v_{\bot }\in \Delta (\sigma )^{\bot }$ it follows by
  \eqref{eqn-arisd12} that
  $v_{\bot }^{F}\in \lbrack \Delta (\sigma^{F})_{\mathbb{R }^{\kappa _{\ast }}}]^{\bot }$.
  Note further that
  since $\sigma \in \Delta (\mathcal{M})$, we have $\sigma \prec
  \hspace{-0.1cm}\sim \sigma_{\ast}$, see \eqref{eqn-28Nov2007-1}. Let
  now $ w=w^{F}\otimes w^{M}$ be any element of
  $T_{\mathcal{M}}(\bar{x})$ for which $w^{F}\in \mathbb{R}^{\kappa
    _{\ast }}$ has the property described in Lemma~\ref{lem-hss}. Pick
  any permutation $\tau \in S^{\succsim }(\sigma )$. Then, $\tau $
  admits a canonical decomposition $\tau =\tau ^{F}\circ \tau ^{M}$
  with $\tau ^{M}\succsim \sigma^{M}$ and $ \tau ^{F}\succsim \sigma
  ^{F}$ (Proposition~\ref{prop-arisd2}). It follows that $(\tau
  w)^{F}=\tau ^{F}w^{F}$, $(\tau w)^{M}=\tau ^{M}w^{M}=w^{M}$ (in view
  of (\ref{arisd10})) and $\tau w\in T_{\mathcal{M}}(\bar{x})$ (in
  view of Lemma~\ref{Tinv}$(\ref{Tinv-i})$). Thus, we deduce
  successively:
  \begin{equation*}
    0=\langle v_{\perp },\tau w\rangle=\langle v_{\perp }^{F},(\tau
    w)^{F}\rangle \,+\,\langle v_{\perp }^{M},(\tau w)^{M}\rangle
    =\langle v_{\perp }^{F},\tau ^{F}w^{F}\rangle \,+\,\langle v_{\perp
    }^{M},w^{M}\rangle .
  \end{equation*}
  This yields
  \begin{equation*}
    \langle v_{\perp }^{F},\tau ^{F}w^{F}\rangle \,=-\langle v_{\perp
    }^{M},w^{M}\rangle ,
  \end{equation*}
  which in view of Corollary~\ref{cor-29Jun07-1} in the Appendix
  (applied to $ x:=v_{\perp }^{F}\in \lbrack \Delta (\sigma
  ^{F})_{\mathbb{R}^{\kappa _{\ast }}}]^{\bot }$, $\sigma :=\sigma
    ^{F}$, $y:=w^{F}$, $\sigma' :=\tau^{F}$, and $\alpha :=-\langle v_{\perp
    }^{M},w^{M}\rangle $) yields $v_{\perp }^{F}=\{0\}_{\kappa _{\ast
    }} $. Finally, let us recall that $v_{\bot \bot }\in \Delta
    (\sigma )^{\bot \bot },$ which in view of (\ref{eqn-arisd11})
    yields $v_{\bot \bot }^{F}\in \lbrack \Delta (\sigma
    ^{F})_{\mathbb{R}^{\kappa _{\ast }}}]^{\bot \bot }$. Thus,
      $v^{F}=v_{\bot \bot }^{F}\in \lbrack \Delta (\sigma
      ^{F})_{\mathbb{R} ^{\kappa _{\ast }}}]^{\bot \bot }$. The proof
        is complete.
\end{proof}

%*******************************************************************************
\subsection{Tangential parametrization of a locally symmetric manifold}\label{ssec-tangparam}
%*******************************************************************************

In this subsection we consider a local equation of the
manifold, called {\it tangential parametrization}. We briefly
recall some general properties of this parametrization (for any
manifold $\Mm$) and then, we make use of \thref{jm-double} to
specify it to our context.

The local inversion theorem asserts that for some $\delta>0$
sufficiently small the restriction of $\bar{\pi}_T$ around $\bar x
\in\Mm$
\[
\bar{\pi}_T\colon \mathcal{M}\cap
B(\bar{x},\delta)\rightarrow \bar{x}+T_{\mathcal{M}} (\bar{x})
\]
is a diffeomorphism of $\mathcal{M}\cap B(\bar{x},\delta)$ onto its
image (which is an open neighborhood of $\bar x$ relatively to the
affine space $\bar{x}+T_{\mathcal{M}}(\bar{x})$). Then, there exists a
smooth map
\begin{equation}\label{arisd-phi}
\phi\colon (\bar x +T_{\mathcal{M}}(\bar x))
\cap B(\bar x, \delta) \rightarrow N_{\mathcal{M}}
(\bar x),
\end{equation}
such that
\begin{equation}\label{prop-30June07-2}
\mathcal{M} \cap B(\bar x, \delta) = \{y \in \R^n : y=x + \phi(x), \,
x \in (\bar x+T_{\mathcal{M}}(\bar x)) \cap B(\bar x, \delta)\}.
\end{equation}
In words, the function $\phi$ measures the difference between the
manifold and its tangent space. Obviously, $\phi \equiv 0$ if $\Mm$ is an
affine manifold around $\bar x$.  Note that, technically, the domain of the map
$\phi$ is the open set $\bar \pi_T(\mathcal{M} \cap B(\bar x,
\delta))$, which may be a proper subset of $(\bar x
+T_{\mathcal{M}}(\bar x)) \cap B(\bar x, \delta)$. Even though we
keep this in mind, it will not have any bearing on the developments
in the sequel. Thus, for sake of readability we will avoid
introducing more precise but also more complicated notation, for
example, rectangular neighborhoods around $\bar x$.

We say that the map $\psi\colon (\bar x+T_{\mathcal{M}}(\bar x)) \cap
B(\bar x, \delta) \rightarrow \mathcal{M} \cap B(\bar x, \delta)$
defined by
\begin{equation}\label{prop-30June07-5}
\psi(x) = x + \phi(x)
\end{equation}
is the tangential parametrization of $\Mm$ around $\bar x$. This
function is indeed smooth, one-to-one and onto, with a full rank
Jacobian matrix $J \psi(\bar x)$: it is a local diffeomorphism at
$\bar x$, and more precisely its inverse is $\bar \pi_T$, that is,
locally $\bar \pi_T( \psi(x) ) = x$. The above properties of $\psi$
hold for any manifold.

Let us return to the situation where $\Mm$ is a locally symmetric
manifold. We consider its characteristic permutation
$\si_{\ast}$, and we make the
following assumption on the neighborhood.

\begin{assumption}\label{assmpt-21Nov2008}
  Let $\Mm$ be a locally symmetric $C^2$-submanifold of $\R^n$ of dimension $d$
  and of characteristic permutation $\si_*$.
  We consider $\bar x\in \Mm\cap \De(\si)$ and we take $\delta > 0$ small enough so that:
  \begin{enumerate}
    \item $B(\bar x, \delta)$ intersects only strata $\Delta(\sigma')$
      with $\sigma' \succsim \sigma$ (recall \lemref{key});
    \item $\Mm \cap B(\bar x, \delta)$ is a strongly locally symmetric manifold;
    \item $\mathcal{M}\cap B(\bar{x},\delta)$ is diffeomorphic to its
      projection on $\bar{x}+T_{\mathcal{M}} (\bar{x})$; in other
      words, the tangential parametrization holds.
  \end{enumerate}
  This ensures that
  $$
  \Delta(\sigma)^{\bot\bot} \cap B(\bar x, \delta) = \Delta(\sigma)
  \cap B(\bar x, \delta).
  $$
\end{assumption}

This situation enables us to specify the general properties of the
tangential parametrization.

\begin{lemma}[Tangential parametrization]\label{important-observation}
  Let $\bar x \in \mathcal{M} \cap \Delta(\sigma)$.  Then, the
  function $\phi$ in the tangential parametrization satisfies
  \begin{equation}\label{aris100}
    \phi(x) \in N_{\mathcal{M}}(\bar x) \cap \Delta(\sigma)^{\bot\bot}.
  \end{equation}
  Moreover, for all $x \in (\bar x+T_{\mathcal{M}}(\bar x)) \cap B(\bar
  x, \delta)$ and for all $\sigma' \in S^{\succsim}(\sigma)$ we have
  \begin{equation}\label{5Dec07-2}
    \psi(\sigma'  x) = \sigma' \psi(x)
  \end{equation}
  and
  \begin{equation}\label{5Dec07-3}
    \phi(\sigma' x) = \si' \phi(x) = \phi(x).
  \end{equation}
\end{lemma}

\begin{proof}
  Recalling the direct decomposition of the normal space (see
  Proposition~\ref{props-spectr-man}) we define the mappings
  $\phi_{\bot\bot}(x)$ and $\phi_{\bot}(x)$ as the projections of
  $\phi(x)$ onto $N_{\mathcal{M}}(\bar
  x)\cap\Delta(\sigma)^{\bot\bot}$ and $N_{\mathcal{M}}(\bar x) \cap
  \Delta(\sigma)^{\bot}$ respectively. Thus, (\ref{prop-30June07-5})
  becomes
  \begin{equation}\label{eqn-5Dec07-3}
    \psi(x) = x + \phi_{\bot\bot}(x)+\phi_{\bot}(x).
  \end{equation}
  Splitting each term in both sides of Equation~(\ref{eqn-5Dec07-3})
  in view of the canonical split defined in \eqref{eqn-arisd1}, we
  obtain
  $$
  \left(\begin{array} {c}
    \psi^F(x) \smallskip \\
    \psi^M(x)
  \end{array}\right) =
  \left(\begin{array} {c}
    x^F \smallskip \\
    x^M
  \end{array}\right) +
  \left(\begin{array} {c}
    \phi_{\bot\bot}^F(x)\smallskip \\
    \phi_{\bot\bot}^M(x)
  \end{array}\right) +
  \left(\begin{array} {c}
    \phi_{\bot}^F(x) \smallskip \\
    \phi_{\bot}^M(x)
  \end{array}\right).
  $$
  We look at the second line of this vector equation.
  Since
  $$
  \phi_{\bot\bot}(x) \in N_{\mathcal{M}}(\bar x) \cap \Delta(\sigma)^{\bot\bot}
  \, \mbox{ and } \,
  \phi_{\bot}(x) \in N_{\mathcal{M}}(\bar x) \cap \Delta(\sigma)^{\bot}
  $$
  we deduce from \eqref{eqn-arisd11} and \eqref{eqn-arisd12} that
  $$
  \phi_{\bot\bot}^M(x) \in [\Delta
    (\si^M)_{\mathbb{R}^{n-\kmin}}]^{\bot\bot} \qquad \text{ and }
  \qquad \phi_{\bot}^M(x) \in [\Delta
    (\si^M)_{\mathbb{R}^{n-\kmin}}]^{\bot}.
  $$ Since $x \in \bar x +T_{\mathcal{M}}(\bar x)$ and $\psi(x) \in
  \mathcal{M}$ we deduce from \eqref{eqn-28Nov2007-1} and
  \eqref{32} that
  $x^M,\psi^M(x)\in [\Delta(\sigma^M)_{\mathbb{R}^{n-\kmin}}]^{\bot\bot}$
  (recall that $\si^M\sim\si_{\ast}^M$), yielding $\phi_{\bot}^M(x) \in
  [\Delta (\si^M)_{\mathbb{R}^{n-\kmin}}]^{\bot\bot}$ and thus
  $\phi_{\bot}^M(x)=0$.  In addition, by \thref{jm-double} we
  have $\phi_{\bot}^F(x)=0$. Thus, $\phi_{\bot}(x)=0$, which completes
  the proof of \eqref{aris100}.

  \bigskip

  We now show local invariance. Choose any permutation $\sigma'
  \succsim \sigma$. Since $\phi(x)\in\Delta(\si)^{\bot\bot}$, it
  follows that $\si'\phi(x)=\phi(x)$. Thus,
  \begin{equation}\label{epi}
    \sigma' \psi(x) = \sigma' x + \sigma' \phi(x) = \sigma' x + \phi(x).
  \end{equation}
  Since $\mathcal{M} \cap B(\bar x, \delta)$ is locally symmetric, we
  have $\sigma' \psi(x) \in \mathcal{M} \cap B(\bar x, \delta)$. Thus,
  there exists $x_{\circ} \in (\bar x+T_{\mathcal{M}}(\bar x)) \cap
  B(\bar x, \delta)$ such that
  \begin{equation}\label{epi1}
    \sigma' \psi(x) = \psi(x_{\circ}) = x_{\circ} + \phi(x_{\circ}).
  \end{equation}
  Combining \eqref{epi} with \eqref{epi1} we get
  $$
  x_{\circ} - \sigma' x   = \phi(x) - \phi(x_{\circ}).
  $$
  The left-hand side is an element of $T_{\mathcal{M}}(\bar x)$, by
  \lemref{Tinv}, while the right-hand side is in $N_{\mathcal{M}}(\bar
  x)$. Thus, $x_{\circ} = \sigma' x$ and $\phi(x) = \phi(x_{\circ})$,
  showing the local symmetry of $\phi$ which implies \eqref{5Dec07-2}.
\end{proof}

%********************************************************************************
%********************************************************************************
\section{Spectral manifolds}\label{sec:spectral}
%********************************************************************************
%********************************************************************************

We have now enough material on locally symmetric manifolds to tackle
the smoothness of spectral sets associated to them. Before
continuing the developments, we present the particular
case when $\Mm$ is (a relatively open subset of) a stratum
$\Delta(\si)$. In this case, basic algebraic arguments allow to
conclude directly.

\begin{example}[Lift of stratum $\Delta(\si)$]\label{jerome-lift}
  We develop here the case when $\Mm$ is the connected component of $\Delta(\si)$
  which intersects $\RR^n_{\geq}$. More precisely, we consider
  $\si\in \Sigma^n$, $\bar x\in \Delta(\si)\cap \RR^n_{\geq}$ and $\delta>0$,
  and we assume
  \[
  \Mm = \Delta(\si) \cap B(\bar x,\de)\,.
  \]
  In this case, we show directly that the spectral set
  $$
  \lambda^{-1}(\mathcal{M}) =\bigcup_{x \in
    \mathcal{M}} \On.\Diag (x)
  $$
  is an analytic (fiber) manifold using basic
  arguments exposed in \ssecref{ssec-loc-sym}.
  We stated therein that the orbit $\On_{\Diag(x)}$ is a submanifold
  of $\Sn$ with dimension
  \[
  \sum_{1 \le i < j \le \kappa + m} |I_{i}||I_{j}|
  \]
  where $P(x) = \{I_1,\ldots,I_{\kappa + m}\}$.
  The key is to observe that, in this example,
  for any $x \in \mathcal{M}$ we have
  \[
  \On_{\Diag(x)} \ = \  \On_{\Diag(\bar x)}
  \ \simeq \ {\bf O}^{|I_1|} \times \cdots \times {\bf O}^{|I_{m+\kappa}|}
  \]
  and $P(x) = P(\si)$ (thus also $\si_*=\si$). Then all the orbits
  $\On.\Diag(x)$ are manifolds diffeomorphic to $\On/\On_{\Diag(\bar x)}$ (fibers), whence
  of the same dimension. We deduce that $\lambda^{-1}(\mathcal{M})$ is a submanifold of $\Sn$
  diffeomorphic to the direct product $\Mm\times\big(\On/
  \On_{\Diag(\bx)}\big)$, with dimension
  \begin{align}
  \dim \lambda^{-1}(\mathcal{M})  &=  d + \sum_{1 \le i < j \le \kappa+m}
  |I_{i}||I_{j}|\,. \label{aris-dime}
  \end{align}
  The proof is complete. \qed
\end{example}

The proof of the general situation (that is, $\Mm$ arbitrary locally
symmetric manifold) is a generalization of the above
arguments, albeit a nontrivial one.
%In this section, we use
%the decomposition and the diffeomorphism introduced in the previous
%section to adapt arguments similar to the ones given in
%Example~\ref{jerome-lift} in the general case.
The strategy is more
precisely explained in \secref{ssec-strategy}. Before this, in
Subsection~\ref{sec-split-of-Sn} we
introduce the block-diagonal
decomposition of $\Sn$, and then we show in \secref{ssec-42} that,
 in the special case $\si_*={\rm id}_n$,
locally symmetric manifolds lift through this decomposition.

%%**********************************************************************
\subsection[Split of $\Sn$ induced by an ordered partition]{Split of \boldmath$\Sn$ induced by an ordered partition}\label{sec-split-of-Sn}
%%**********************************************************************

In this section, we introduce a notion of split of the space of
symmetric matrices, associated to an ordered partition. We use later
the canonical split associated to the partition induced by the
characteristic permutation $\si_*$ of the manifold.

\begin{definition}[Ordered partition]
  \label{partition-order}
  Given a partition $P=\{I_1,\ldots,I_m\}$ of $\N_n$ we say that $P$ is
  {\it ordered} if for any $1 \le i < j \le m$ the smallest element in
  $I_i$ is (strictly) smaller than the smallest element in $I_j$. We
  use parenthesis $P = (I_1,\ldots,I_m)$ to indicate that the sets
  $I_1,\ldots,I_m$ in the partition $P$ are ordered. For example, the
  partition $\{\{4\}, \{3, 2\}, \{1, 5\}\}$ of $\N_5$ gives the
  ordered partition $(\{1, 5\}, \{3, 2\}, \{4\})$.
\end{definition}

Now we consider the following linear spaces, defined as direct products
\begin{align}
\label{direct-prods}
{\bf S}_{\sigma}^n :=  {\bf S}^{|I_1|} \times \cdots \times {\bf S}^{|I_m|}
\qquad \mbox{and} \qquad {\bf O}_{\sigma}^n :=  {\bf O}^{|I_1|} \times \cdots \times {\bf O}^{|I_m|},
\end{align}
for the given ordered partition $P(\sigma) = (I_1,\ldots,I_m)$.  We
denote by $X_{\si}=X_1 \times \cdots \times X_m \in {\bf
  S}_{\sigma}^n$ an element of ${\bf S}_{\sigma}^n$, where $X_i \in
{\bf S}^{|I_i|}$.  We can interpret $X_{\si}\in {\bf S}_{\sigma}^n$
as the $n \times n$ block-diagonal matrix with the blocks $X_1,\ldots,X_m$
on the diagonal. This is formalized by the linear embedding
\begin{equation}\label{reve-fred}
  \fonction{i}{\Sn_\si}{\Sn}{X_\si}{X = \Diag(X_1,\ldots,X_m).}
\end{equation}
The product of two elements $A_\si$ and $B_\si$ of
${\bf S}_{\sigma}^n$ is defined component-wise in the natural way.
Clearly, we have $\Diag(X_\si):=\Diag(i(X_\si))$.
%as follows
%\[
%A_\si B_\si := i^{-1}(i(A_{\si})i(B_{\si})).
%\]
For any $X_{\si}=X_1 \times \cdots \times X_m \in  {\bf S}_{\sigma}^n$,
we introduce
$$
\lambda_\sigma (X_{\si}) := \lambda(X_1) \times \cdots \times  \lambda(X_m) \in \RR^n.
$$
Recall that $\la(X) \in \RR^n$ is the ordered vector of eigenvalues
of $X\in {\bf S}^n$. Note the difference between
$\lambda_\sigma (X_{\si})$ and $\la(i(X_\si))$: the coordinates of
the vector $\lambda_\sigma (X_{\si})$ are ordered within each block
while those of $\la(i(X_\si))$ are ordered globally. Nonetheless
they coincide in the following case.

\begin{lemma}\label{jerome-Ialsocangivecrazynames}
  Assume $\la_\si(\bar X_\si)\in \De(\si)\cap\RR_{\geq}^n$.
  If $X_\si$ is close to $\bar X_\si$, then
  $\la(i(X_\si))=\la_{\si}(X_{\si})$.
\end{lemma}

\begin{proof}
  The assumption
  $\la_\si(\bar X_\si)\in \De(\si)\cap\RR_{\geq}^n$
  yields, for $1\leq \ell\leq m-1$,
  \[
  \la_{\min}(\bar X_\ell) > \la_{\max}(\bar X_{\ell+1}).
  \]
  The continuity of the eigenvalues
  implies that for $X_\si$ close to $\bar X_\si$,
  $\la_{\min}(X_\ell) > \la_{\max}(X_{\ell+1})$.
  Since by construction
  $\la_{\si}(X_\si)$ is ordered within each block,
  we get that $\la_{\si}(X_\si)$ is ordered globally and thus
  equal to $\la(i(X_\si))$.
\end{proof}

This permits to differentiate easily functions defined
as a composition with $\la_{\si}$.

\begin{lemma}\label{jerome-Ialsocangivefreakingnames}
  Assume $\la_\si(\bar X_\si)\in \De(\si)\cap\RR_{\geq}^n$.
  If $f\colon \RR^n\ra\RR^n$ is locally symmetric around $\la_\si(\bar X_\si)$,
  that is
  \[
  f(\si'x) = f(x) \qquad \text{for all }\si'\in S^{\succsim}(\si),
  \]
  then $f\circ \la_{\si}$ is $C^{1}$ around $\bar X_\si$, provided
  $f$ is $C^{1}$ around $\la_\si(\bar X_\si)$. Moreover, the Jacobian of $f\circ
  \la_{\si}$ at $\bar X_\si$ applied to $H_{\si}\in {\bf S}_{\si}^n$ is
  \[
  J(f\circ \la_{\si})(\bar X_\si)[H_\si]
  =J(f\circ \la)(i(\bar X_\si))[i(H_\si)].
  \]
\end{lemma}

\begin{proof}
  \lemref{jerome-Ialsocangivecrazynames} gives that around $\bar X_\si$, we have
  $f\circ \la_{\si} = f \circ \la \circ i$. Apply \thref{thm-26Nov2007-1} to all of its components,
  we get that the function $f \circ \la_\si$ is $C^1$.
  The expression of the Jacobian follows from the chain rule.
\end{proof}

Let us come back now to the locally symmetric manifold $\Mm$.  We
fix a point $\bx\in \mathcal{M} \cap \RR_{\geq}^n$, and a
permutation $\sigma\in \Perm$ such that $\bar x \in \Delta(\si)$. We
also consider $\si_{\ast}$ be the characteristic permutation of~$\Mm$ (see Subsection~\ref{ssec-simin}). By (\ref{eqn-28Nov2007-1}),
we have $\sigma \prec \hspace{-0.1cm} \sim \sigma_{\ast}$ or
$\sigma\sim\si_{\ast}$, and thus the $(F,M)$-decomposition can be
applied to $\sigma$, {\it i.e.} $\sigma = \sigma^F \circ \sigma^M$
(recall \secref{sssec-341}). Consider now the ordered partitions of
$\sigma^F$ and $\sigma^M$
\begin{equation}\label{F1}
P(\sigma^F) = (I_1,\ldots,I_{\kappa})  \qqandqq
P(\sigma^M) = (I_{\kappa+1},\ldots,I_{\kappa+m})\,=P(\si_{\ast}^M),
\end{equation}
where $\kappa$ (resp.\ $m$) stands for the cardinality of the partition
$P(\si^F)$ (resp.\  $P(\si^M)$). Recalling the definitions of
$P(\si_{\ast})$, $m_{\ast}$ and $\kappa_{\ast}$ (see respectively
\eqref{part-of-sigmamin}, \eqref{aris-m} and \eqref{aris-k}), we
observe that $\kappa\le\kappa_{\ast}$, $m=m_{\ast}$ by
Proposition~\ref{prop-arisd2}, as well as the equalities
\begin{align}\label{equal-sets}
\big| \cup_{i=1}^\kappa I_i \big| = \big| \cup_{i=1}^{\kappa_*} I_i \big| = \kappa_* \quad \mbox{and} \quad (I_{\kappa+1},\ldots,I_{\kappa+m}) = (I_{\kmin+1}^{\ast},\ldots,I_{\kmin+m_{\ast}}^{\ast}).
\end{align}
The main result of this subsection (forthcoming \propref{real-Fred})
is about the spaces ${\bf S}^{\kappa_*}_{\sigma^F}$ and ${\bf
  S}^{n-\kappa_*}_{\sigma^M}$ defined by \eqref{direct-prods} for
$\si^F$ and $\si^M$ respectively.
Before going any further, let us make more precise a point about notation. Recall from
Example~\ref{example-order} that two vectors $x^F \in \R^{\kappa_*}$ and
$x^M \in \R^{n-\kappa_*}$ give rise~to
\begin{itemize}
\item the usual {\it direct} product $x^F \times x^M$ that corresponds
  to the ordered pair $(x^F, x^M)$ considered as a vector in
  $\R^n$,
\item the canonical product $x^F \otimes x^M$ which {\it intertwines}
  the vectors $x^F$ and $x^M$ into a vector of~$\R^n$. The canonical
  product depends on $\sigma_*$, while the direct product does not.
\end{itemize}

We now recall a general result quoted from Example~3.98 of \cite{BonShap2000}.

\begin{lemma}\label{lem-fred}
  Let $\bar Y \in \mathbf{S}^n$ have eigenvalues
  \[
    \lambda_1(\bar Y) \ge \cdots \ge \lambda_{k-1}(\bar Y) >
    \lambda_{k}(\bar Y) = \dots =
    \lambda_{k+r-1}(\bar Y) > \lambda_{k+r}(\bar Y)
    \ge \cdots \ge \ \lambda_n(\bar Y).
  \]
  Then, there exist an open neighborhood $W\subset \mathbf{S}^n$ of $\bar Y$ and
  an analytic map $\Theta\colon W \rightarrow {\bf S}^r$
  such that
  \begin{itemize}
  \item[(i)] for all $Y\in W$, we have
    $ \{\lambda_{k}(Y),\ldots,\lambda_{k+r-1}(Y)\}=
    \{\lambda_{1}(\Theta(Y)),\ldots,\lambda_{r}(\Theta(Y))\},$
  \item[(ii)] the Jacobian of $\Theta$ has full rank at $\bar Y$.
  \end{itemize}
\end{lemma}

With the help of the previous lemma,
we obtain the following result, used later in
\thref{main-res-simple-form-2}.

\begin{proposition}[Local canonical split of $\Sn$ induced by $\si_*$]\label{real-Fred}
  With the notation of this subsection, there exist an open neighborhood $W\subset
  \Sn$ of $\bar X \in \lambda^{-1}(\bar x)$ and two analytic maps
  $$
  \Theta^F\colon  W \rightarrow {\bf S}^{\kappa_*}_{\sigma^F} \qquad \mbox{and} \qquad
  \Theta^M\colon  W \rightarrow {\bf S}^{n-\kappa_*}_{\sigma^M},
  $$
  such that
  \begin{itemize}
  \item[(i)] $\lambda(X) = \lambda_{\sigma^F}(\Theta^F(X)) \otimes
    \lambda_{\sigma^M}(\Theta^M(X)) \quad \mbox{for all } X\in W;$
  \item[(ii)] the Jacobians of the analytic maps $\Theta^F$ and $\Theta^M$
    have full ranks at $\bar X$.
  \end{itemize}
\end{proposition}

\begin{proof}
  We are going to apply Lemma~\ref{lem-fred}
  for each block (so $(\kappa+m)$ times).
  To have the right order, we start by renumbering the blocks $I_i$:
  since the blocks in the ordered partitions (\ref{F1}) are
  made of consecutive numbers (by Lemma~\ref{consec} ---recall $\bar x\in\Mm\cap\RR^n_{\ge}$),
  there exists a permutation~$\tau\in \Sigma^{\kappa+m}$, such that
  for all $1\le \ell_1 < \ell_2 \le \kappa+m$
  \[
  i \in I_{\tau(\ell_1)}, \,\, j \in I_{\tau(\ell_2)}
  \,\,\Longrightarrow \,\, i < j \ \ (\text{in other words }
  \la_i(\bar{X})>\la_j(\bar{X})).
  \]
  The permutation $\tau$ describes how the canonical product intertwines the blocks of the vectors on the right-hand side of $(i)$.
  So we apply \lemref{lem-fred} for all $\ell = 1,\ldots,\kappa+m$ to
  get open neighborhoods
  $W_\ell \subset \Sn$ of $\bar X$ and analytic maps with Jacobians having full rank
  $$
  \Theta_{\tau(\ell)}\colon W_{\ell} \rightarrow {\bf S}^{|I_{\tau(\ell)}|}.
  $$
  Set $W=\bigcap^{\kappa+m}_{\ell=1}W_\ell$ and
  put the $F$-pieces and the $M$-pieces together, that is, define
  \begin{align*}
    \Theta^F := \Theta_1 \times \cdots \times \Theta_{\kappa} \qquad \mbox{and} \qquad
    \Theta^M := \Theta_{\kappa+1} \times \cdots \times \Theta_{\kappa+m},
  \end{align*}
  restricting the $\Theta_\ell$ to $W$.
  We observe that the above functions satisfy the desired properties.
\end{proof}

%********************************************************************************
\subsection[The lift-up into $\Sn_{\si}$ in the case $\si_*=\mathrm{id}_n$]{The lift-up into \boldmath$\Sn_{\si}$ in the case \boldmath$\si_*=\mathrm{id}_n$}
\label{ssec-42}
%********************************************************************************

In this section, we consider the case when $\kmin=n$ (that is
$\si_*=\id_n$, or again $\Sigma_{\Mm}=\{\mathrm{id}_n\}$).  Let $\bar
x$ and $\si$ such that $\bar x \in \Mm\cap \Delta(\si)$; we have
obviously $\si^F=\si$ (see \propref{prop-arisd1}).
The important property in this case is the
simplification given by \thref{jm-double} which yields
\begin{equation}\label{eqn-4Dec07-1}
  N_{\mathcal{M}}(\bar x)  \subseteq \Delta(\sigma)^{\bot\bot}.
\end{equation}
The goal here is to establish that
the set $\lambda^{-1}_{\sigma}(\mathcal{M})$ is a submanifold
of ${\bf  S}^n_{\sigma}$, and to calculate its dimension.
This is an intermediate step in our way to prove that
$\lambda^{-1}(\mathcal{M})$ is a submanifold
of ${\bf  S}^n$ (in the general case).
This also enables us to grind our strategy: the succession of arguments
will be similar for the general case.

From \eqref{eqn-4Dec07-1}, we can exhibit easily a locally symmetric
equation of $\Mm$. We first recall from \eqref{pi_T} and \eqref{pi_N}
the definitions of $\bar \pi_T(x)$ and $\bar \pi_N(x)$ respectively,
as well as the definition of $\phi$ by \eqref{arisd-phi}. Consider the
ball $\mathcal{B}(\bar x, \delta)$ satisfying
Assumption~\ref{assmpt-21Nov2008}, and define the function
\begin{equation}\label{aris-C1}
\fonction{\bar \phi}{B(\bar x, \delta) \subset \RR^n}
         {N_{\mathcal{M}}(\bar x)\subset \RR^n}
         {x}{\bar x + \phi(\bar \pi_T(x))-\bar \pi_{N}(x).}
\end{equation}

\begin{lemma}[Existence of a locally symmetric local equation in the case $\si_*=\mathrm{id}_n$]\label{lemma-37}
  The function~$\bar \phi$ defined by \eqref{aris-C1} is a local
  equation of $\Mm$ around $\bar x \in \Mm \cap \Delta(\sigma)$
  that is locally symmetric, in other words
  $$
  \bar \phi (\sigma' x) = \sigma' \bar \phi(x) = \bar \phi(x)
  \qquad \mbox{for all } \sigma' \in S^{\succsim}(\sigma).
  $$
\end{lemma}

\begin{proof}
  For $x \in \mathcal{B}(\bar x, \delta)$ we have that
  \[
  \bar \phi(x)=0 \iff
  \bar \pi_{N}(x) = \bar x + \phi(\bar \pi_T(x))\iff
  x = \bar \pi_T(x) + \phi(\bar \pi_T(x))\iff
  x \in \mathcal{M} \cap B(\bar x, \delta),
  \]
  using successively (\ref{pi_TandN}) and (\ref{prop-30June07-2}). The
  Jacobian mapping $J\bar\phi(\bx)$ of $\bar \phi$ at $\bx$ is a
  linear map from~$\R^n$ to~$N_{\mathcal{M}}(\bar x)$, which, when
  applied to any direction $h$, yields
  \begin{align*}
    J \bar \phi(x)[h] &= J \phi(\bar \pi_T(x))[\pi_T(h)]- \pi_{N}(h).
  \end{align*}
  Clearly, for $h \in N_{\mathcal{M}}(\bar x)$ we have $J
  \bar\phi(\bar x)[h]=-h$ showing that the Jacobian in onto and hence
  of full rank. Thus, $\bar\phi$ is a local equation of $\Mm$ around
  $\bx$. Finally, Corollary~\ref{corollary_pi}$(i)$ and
  Lemma~\ref{important-observation} show that for any $\sigma'
  \succsim \sigma$ and any $x \in B(\bar x, \delta)$ we have $(\phi \circ
  \bar \pi_T)(\sigma'x)= (\phi \circ \bar \pi_T)(x)$. Thus, in view of
  Corollary~\ref{corollary_pi}$(ii)$ and Lemma~\ref{important-observation} again,
  for $\sigma' \in S^{\succsim}(\sigma)$, we have
  $$
  (\sigma')^{-1} \bar \phi(\sigma' x) = (\sigma')^{-1} (\bar x + (\phi \circ \bar \pi_T)(x) -
  \sigma' \bar \pi_N(x)) = \bar \phi(x).
  $$
  Since $\bar \phi(x) \in N_{\mathcal{M}}(\bar x)\subset \Delta(\si)^{\bot\bot}$, we
  obtain the second equality $\sigma' \bar \phi(x) = \bar \phi(x)$.
\end{proof}

Let us consider the map
\begin{equation}\label{aris-C11}
  \fonction{\bar \Phi}
           {\lambda^{-1}_{\sigma}(\mathcal{B}(\bar x, \delta))
             \subset {\bf S}^n_{\sigma}}
           {N_{\mathcal{M}}(\bar x)\subset \RR^n}
           {X_{\si}}
           {(\bar \phi \circ \lambda_{\sigma})(X_{\si}) = \bar x + \phi(\bar
             \pi_T(\lambda_{\sigma}(X_{\si})))-\bar \pi_{N}(\lambda_{\sigma}(X_{\si})).}
\end{equation}
Since $\bar \phi$ is a local equation of $\Mm$ around $\bar x$,
we deduce for $X_{\si} \in {\bf S}^n_{\sigma}$
\begin{equation}\label{aris-C2}
  X_{\si} \in \lambda^{-1}_{\sigma}(\mathcal{M}\cap B(\bar x, \delta)) \iff
  \lambda_{\sigma}(X_{\si}) \in \mathcal{M}\cap B(\bar x, \delta) \iff \bar
  \Phi(X_{\si})=0.
\end{equation}
Thus, it suffices to show that $\bar\Phi$ is differentiable and that
its Jacobian $J\bar \Phi$ has full rank at $\bar X_{\si} \in
\lambda^{-1}_{\sigma}(\bar x)$. This is the role of forthcoming
Theorem~\ref{main-thm-4Dec2007-1}. We shall first need the following
lemma.

\begin{lemma}\label{lem-26Nov2007-4}
  The function $\bar\pi_N\circ\lambda_{\sigma}$ is differentiable
  at $\bar X_{\si} \in \lambda^{-1}_{\sigma}(\bar x)$. Moreover,
  for any direction $H_{\si} \in \Sn_{\sigma}$ we have
  \[
  J (\bar \pi_{N}\circ \lambda_{\sigma})(\bar X_{\si})[H_{\si}]
  = \pi_{N}( \diag\, (\bar{U}_{\si}H\trans{\bar{U}_{\si}})),
  \]
  where $\bar U_{\si} \in \On_{\sigma}$ is such that $\bar X_{\si} =
  \trans{\bar U_{\si}}\big(\Diag\,\lambda_{\sigma}(\bar X_{\si})\big)
  \bar U_{\si}$, recalling the embedding \eqref{reve-fred}.
\end{lemma}

\begin{proof}
  The fact that $\bar x \in \Delta(\sigma)^{\perp\perp}$ together with
  \eqref{eqn-4Dec07-1} gives that $\bar x + N_{\mathcal{M}}(\bar x)
  \subseteq \Delta(\sigma)^{\perp\perp}$. Therefore
  $\bar \pi_{N}(x) \in \Delta(\sigma)^{\perp\perp}$, and consequently
  \[
  \si'\bar \pi_{N}(x) = \bar \pi_{N}(x)
  \qquad \text{for all }\si'\in S^{\succsim}(\si).
  \]
  Together with \corref{corollary_pi}, this gives that $\bar \pi_{N}$
  is locally symmetric around $\bar x$.
  So we can apply
  \lemref{jerome-Ialsocangivefreakingnames} to get
  that $\bar \pi_{N}\circ \lambda_{\sigma}$ is differentiable at $\bar X_{\si}$.

  We also get the expression of its Jacobian at $\bar X_{\si}$
  applied to the direction $H_{\si} \in \Sn_{\sigma}$
  by applying \thref{thm-26Nov2007-1} on each component:
  \begin{eqnarray*}
    J(\bar\pi_N\circ\lambda_{\sigma})(\bar X_{\si}) [H_\si]
    &=& J(\bar\pi_N\circ \la)(i(\bar X_\si))[i(H_\si)]\\[1ex]
    &=& J(\bar\pi_N(\la(i(\bar X_\si)))
    [\diag(i(\bar U_{\si})i(H_\si)\trans{i(\bar U_{\si}))}]\\[1ex]
    &=&\pi_N(\diag \,(\bar U_{\si} H_{\si} \trans{\bar U_{\si}})),
  \end{eqnarray*}
  the last equality following by definition of the objects in $\Sn_{\si}$.
  This finishes the proof.
\end{proof}

\begin{theorem}[Local equation of $\lambda^{-1}_{\sigma}(\mathcal{M})$
    in the case $\si_*=\mathrm{id}_n$]\label{main-thm-4Dec2007-1}
  Let $\mathcal{M}$ be a locally symmetric $C^2$ submanifold
  of $\R^n$ around $\bar x \in\Mm\cap \RR^n_{\ge} \cap \Delta(\sigma)$
  of dimension $d$. If $\si_*=\id_n$, then
  $\lambda^{-1}_{\sigma}(\mathcal{M})$ is a $C^2$~submanifold of $\Sn_{\sigma}$ around
  $\bar X_{\si} \in \lambda^{-1}_{\sigma}(\bar x)$,
  whose codimension in $\Sn_{\sigma}$ is $n-d$.
\end{theorem}

\begin{proof}
  By \corref{corollary_pi} and \lemref{important-observation}, the function
  $\phi\circ \bar \pi_T$ is locally symmetric. Therefore
  \lemref{jerome-Ialsocangivefreakingnames} yields that
  $\phi\circ \bar \pi_T \circ \la_{\si}$ is differentiable at $\bar X_{\si}$.
  Combining this with \lemref{lem-26Nov2007-4}, we deduce
  that the function $\bar\Phi$ defined in \eqref{aris-C11} is
  differentiable at $\bar X_{\si}$.

  Let us now show that the Jacobian $J\bar\Phi$ has full rank at $\bar X_{\si}$.
  The gradient of the $i$-th coordinate function
  $(\phi_i \circ \bar \pi_T)$ at $\bar x$ applied to the direction $h$ is
  \[
  \nabla (\phi_i \circ \bar
  \pi_T)(\bar x)[h] = \nabla\phi_i(\bar \pi_T(\bar x))[\pi_T(h)].
  \]
  Thus for $i\in\{1,\dots,n\}$, \lemref{jerome-Ialsocangivefreakingnames}
  and Theorem~\ref{thm-26Nov2007-1} give
  that the gradient of $\phi_i \circ \bar \pi_T \circ \lambda_{\sigma}$
  at $\bar X_{\si}$ in the direction
  $H_{\si}\in {\bf S}^n_{\sigma}$ is
  $$
  \nabla (\phi_i \circ \bar \pi_T \circ \lambda_{\sigma})(\bar X_{\si})[H_{\si}] =
  \nabla \phi_i(\bar \pi_T(\lambda_{\sigma}(\bar X_\si)))[\pi_T(\diag
    \,(\bar{U}_{\si}H_{\si}\trans{\bar{U}_{\si}}))].
  $$
  %where $\bar U \in \On_{\sigma}$ is such that $\bar X = \trans{\bar
  %  U}(\Diag\,\lambda_{\sigma}(\bar X)) \bar U$.
  Combining this with
  Lemma~\ref{lem-26Nov2007-4} we obtain the following expression for
  the derivative of the map~$\bar \Phi$ at $\bar X_{\si}$ in the direction
  $H_{\si} \in \Sn_{\sigma}$:
  \begin{align*}
    J \bar \Phi(\bar X_{\si})[H_{\si}]
    =  J\phi(\bar \pi_T(\lambda_{\sigma}(\bar X_{\si})))
    [\pi_T(\diag \,(\bar{U}_{\si}H_{\si}\trans{\bar{U}_{\si}}))] -
    \pi_{N}( \diag\, (\bar{U}_{\si}H_{\si}\trans{\bar{U}_{\si}})).
  \end{align*}
  Notice that for any $h \in N_{\mathcal{M}}(\bar x)$ defining
  $H_{\si}:=\trans{\bar{U}_{\si}}(\Diag\, h)\bar{U}_{\si} \in \Sn_{\sigma}$
  we have
  $$
  J \bar \Phi(\bar X_{\si})[H_{\si}] = -h,
  $$
  which shows that the linear map $J \bar \Phi(\bar X) \colon \Sn_{\sigma}
  \rightarrow N_{\mathcal{M}}(\bar x)$ is onto and thus has full rank.
  In view of~\eqref{aris-C2}, $\bar\Phi$ is a local equation of
  $\Mm$ around $\bar X_{\si}$.

  Recall that
  $d=\mathrm{dim}\,(\Mm)=\mathrm{dim}\,(T_{\Mm}(\bx)$) and
  $\mathrm{dim}\,(N_{\Mm}(\bx))=n-d$.  Since $\bar\phi$ and $\bar\Phi$
  are local equations of $\Mm$ and $\lambda^{-1}_{\sigma}(\Mm)$
  respectively, the manifolds have the same codimension $n-d$.
\end{proof}

\begin{remark}\label{anotherremark}
  Theorem~\ref{main-thm-4Dec2007-1} remains true if $C^2$ is replaced
  everywhere by $C^{\infty}$ or $C^{\omega}$, see
  Theorem~\ref{thm-26Nov2007-1}. Note however that the statement
  only asserts that $\lambda^{-1}_{\si}(\Mm)$ is a submanifold of
  $\Sn_{\si}$. Nothing is claimed about $\lambda^{-1}(\Mm)$, even in this
  particular case. Nonetheless, this important intermediate result
  will be a basic ingredient in the proof of the main result
  (see proof of \lemref{main-res-simple-form-jeromeF}).
\end{remark}

%**********************************************************************
\subsection{Reduction the ambient space in the general case}\label{ssec-strategy}
%**********************************************************************

We now return to the general case and recall the situation in
Assumption~\ref{assmpt-21Nov2008}. The active space is thus reduced,
as follows:
\[
\mathcal{M} \cap B(\bar x, \delta) \subset \Big( \bar x +
\,T_{\mathcal{M}}(\bar x) \oplus \big( N_{\mathcal{M}}(\bar x) \cap
\Delta(\sigma)^{\bot\bot} \big) \,\Big) \cap B(\bar x,\delta),
\]
where (\ref{prop-30June07-2}) and \eqref{aris100} have been used. To
define a local equation of $\Mm$ in the appropriate space, we
introduced the reduced tangent and normal spaces.
\begin{equation}\label{defn-n1}
  N_{\mathcal{M}}^{\reduc}(\bar x)
  := N_{\mathcal{M}}(\bar x) \cap \Delta(\sigma)^{\bot\bot}
  \qqandqq T_{\mathcal{M}}^{\reduc}(\bar x)
  :=T_{\mathcal{M}}(\bar x) \cap \Delta(\sigma)^{\bot}.
\end{equation}
Note that theses spaces are invariant under permutations $\sigma
^{\prime }\succsim \sigma $ (see Lemma~\ref{Tinv} and
Lemma~\ref{key}). For later use when calculating the dimension of
spectral manifolds, we denote the dimension of
$N_{\mathcal{M}}^{\reduc}(\bar{x})$~by
\begin{align}\label{defn-dim-n1}
  n^{\reduc}:= \dim  N_{\mathcal{M}}^{\reduc}(\bar x).
\end{align}

Let us now define the set on which the local equation of $\la^{-1}(\Mm)$
will be defined.  Let $\bar x = \bar x^F \otimes \bar x^M$ be the
canonical splitting of $\bar x$ in $\R^n$. Naturally
$B(\bar{x}^{F},\delta_1)$ denotes the open ball in
$\mathbb{R}^{\kappa_{\ast }}$ centered at~$\bar{x}^{F}$ with radius
$\delta_1$, and $B(\bar{x}^{M},\delta_2)$ denotes the open ball in
$\mathbb{R}^{n-\kappa _{\ast }}$ centered at $\bx^{M}$ with radius
$\delta_2$. Define the following rectangular neighborhood of $\bar
x$
$$
\mathcal{B}(\bar x, \delta_1, \delta_2) := B(\bar{x}^{F},\delta_1)
\otimes B(\bar{x}^{M},\delta_2).
$$
Choose $\delta_1, \delta_2 > 0$ so that $\mathcal{B}(\bar x,
\delta_1, \delta_2)  \subset B(\bar x, \delta).$ By Assumption~\ref{assmpt-21Nov2008} and Proposition~\ref{prop-arisd2}, the ball~$B(\bar{x}^{F},\delta_1)$ intersects only strata $\Delta(\sigma') \subset \RR^{\kappa_*}$ for $\sigma' \succsim \sigma^F$, and similarly for the ball $B(\bar{x}^{M},\delta_2)$. The key element in
our next development is going to be the set
\begin{equation}\label{Dcal}
  \mathcal{D}:= \big(\bar x + T_{\mathcal{M}}(\bar x) \oplus
  N_{\mathcal{M}}^{\reduc}(\bar x) \big) \cap \mathcal{B}(\bar x, \delta_1, \delta_2),
\end{equation}
which plays the role of a new ambient space (affine subspace of
$\RR^n$ containing all information about~$\Mm$).
We gather properties of $\Dd$ in the next proposition.

\begin{proposition}[Properties of $\Dd$]\label{jerome-all-about}
   In the situation above, there holds
   \begin{equation}\label{aris-new}
     \bar{x}+ T_{\mathcal{M}}(\bar{x})\oplus
     N_{\mathcal{M}}^{\reduc}(\bar{x}) \,= \,
     T_{\mathcal{M}}^{\reduc}(\bar{x})\oplus \Delta (\sigma )^{\bot \bot }.
   \end{equation}
   Hence, we can reformulate
   \[
   \mathcal{D}= \big( T_{\mathcal{M}}^{\reduc}(\bar{x})\oplus \Delta
   (\sigma )^{\bot \bot }\big)\,\cap\,\mathcal{B}(\bar x, \delta_1, \delta_2).
   \]
   This set is relatively open in the affine space
   \[
   \mathcal{R}^{d+n^{\reduc}} :=
   \bar x + T_{\mathcal{M}}(\bar x) \oplus N_{\mathcal{M}}^{\reduc}(\bar x).
   \]
   Moreover, the set $\Dd$ is invariant under all permutations
   $\sigma' \succsim \sigma$, and hence a locally symmetric set.
\end{proposition}

\begin{proof}
  The above formula follows directly by combining \eqref{defn-n1},
  \eqref{tang_decomp} and Corollary~\ref{Cor-23a}. Indeed, we obtain
  successively
  \begin{align}
    \bar{x}+& T_{\mathcal{M}}(\bar{x})\oplus
    N_{\mathcal{M}}^{\reduc}(\bar{x}) \notag
    \\
    & =\bar{x}+ T_{\mathcal{M}}(\bar{x})\oplus \big( N_{\mathcal{M}}(\bar{x})\cap
    \Delta (\sigma )^{\bot \bot }\big) \notag
    \\
    & =\bar{x}+(T_{\mathcal{M}}(\bar{x})\cap \Delta (\sigma )^{\bot
    })\oplus (T_{\mathcal{M}}(\bar{x})\cap \Delta (\sigma )^{\bot
      \bot }) \oplus (N_{\mathcal{M}}(\bar{x})\cap \Delta (\sigma )^{\bot
      \bot }) \notag
    \\
    & =\bar{x}+ \big(T_{\mathcal{M}}(\bar{x})\cap \Delta (\sigma
    )^{\bot } \big)\oplus \Delta (\sigma )^{\bot \bot } \notag
    \\
    &=\bar{x}+ T_{\Mm}^{\reduc}(\bar{x})\oplus \Delta (\sigma )^{\bot
      \bot } \notag,
  \end{align}
  which yields \eqref{aris-new} since $\bx\in\Delta (\sigma )^{\bot
    \bot}$ and $0 \in T_{\Mm}^{\reduc}(\bar{x})$. The reformulation of
  $\Dd$ is then obvious.  Note that by Lemma~\ref{key},
  Lemma~\ref{Tinv}, and Proposition~\ref{prop-arisd2}, the set
  $\mathcal{D}$ is invariant under permutations~$\sigma' \succsim \sigma$, and hence is
  locally invariant.
\end{proof}

Let us introduce the projections onto the reduced spaces
\[
\bar \pi^{\reduc}_{N}(x) = \mbox{Proj\,}_{\bar x + N_{\mathcal{M}}^{\reduc}(\bar x)}(x)
\qqandqq
\pi^{\reduc}_{N}(x) = \mbox{Proj\,}_{N_{\mathcal{M}}^{\reduc}(\bar x)}(x).
\]
Note that there holds $\bar \pi^{\reduc}_{N}(x) = \pi^{\reduc}_{N}(x)+
\bar \pi^{\reduc}_{N}(0)$ and $\bar \pi_{T}(x) = \pi_{T}(x)+ \bar
\pi_{T}(0)$ as well as
\begin{align}\label{5Dec2007-4}
  \bar x+x &= \bar \pi_{T}(x)+\bar \pi^{\reduc}_{N}(x)\qquad
  \mbox{ for all } \, x \in \bar x +  T_{\Mm}(\bar x) \oplus
  N_{\Mm}^{\reduc}(\bar x).
\end{align}
Similarly to \eqref{aris-C1}, we define the map
\begin{equation}\label{aris-D1}
  \fonction{\bar \phi}{\mathcal{D} \subset  \mathcal{R}^{d+n^{\reduc}}}
           {N_{\mathcal{M}}^{\reduc}(\bar x)\ \subset \ \mathcal{R}^{d+n^{\reduc}}}
           {x}{\bar x + \phi(\bar \pi_T(x)) - \bar \pi^{\reduc}_{N}(x),}
\end{equation}
and we show that this function is a locally symmetric local equation of $\Mm$.
This is the content of the following result, analogous to Lemma~\ref{lemma-37}.

\begin{theorem}[Existence of a locally symmetric local equation]\label{thm-40}
  The map $\bar \phi$ is well-defined and locally symmetric, and provides a
  local equation of $\Mm$ around $\bar x$.
\end{theorem}

\begin{proof}
  The set $\Dd$ is chosen so that $\phi$ is well-defined.
  Thanks to \lemref{important-observation} and the fact
  that $\bar x -\bar \pi^{\reduc}_{N}(x) \in N^{\reduc}_{\mathcal{M}}(\bar x)$,
  the range of $\bar \phi(x)$ is in $N_{\mathcal{M}}^{\reduc}(\bar
  x)$. The remainder
  of the proof follows closely the proof of
  Lemma~\ref{lemma-37}. For all $x \in \mathcal{D}$, in view of
  \eqref{5Dec2007-4}, \eqref{prop-30June07-2} and
  Lemma~\ref{important-observation} we obtain
  \[
  \bar \phi(x)=0 \iff \bar \pi^{\reduc}_{N}(x) = \bar x + \phi(\bar
  \pi_T(x)) \iff x = \bar \pi_T(x) + \phi(\bar\pi_T(x))\iff x \in
  \mathcal{M}\cap B(\bx,\delta).
  \]
  The Jacobian of $\bar \phi$ at $x$ is a linear map from
  $T_{\mathcal{M}}(\bar x) \oplus N_{\mathcal{M}}^{\reduc}(\bar x)$ to
  $N^{\reduc}_{\mathcal{M}}(\bar x)$, which applied to any direction
  $h$ yields
  \begin{align*}
    J \bar \phi(x)[h] &= J \phi(\bar \pi_T(x))[\pi_T(h)]- \pi^{\reduc}_{N}(h).
  \end{align*}
  Clearly, for $h \in N^{\reduc}_{\mathcal{M}}(\bar x)$ we have
  $J\bar\phi(\bar x)[h]=-h$ showing that the Jacobian $J\bar \phi$ at
  $\bx$ is onto and has a full rank. Thus, $\bar \phi$ is a local
  equation of $\Mm$ around $\bx$. Finally
  Corollary~\ref{corollary_pi}, Lemma~\ref{important-observation}, and
  Lemma~\ref{Lemma_project_strata} show that for any $\sigma' \succsim
  \sigma$ and any $x \in \mathcal{D}$ we have $(\phi \circ \bar
  \pi_T)(\sigma'x)= (\phi \circ \bar \pi_T)(x)$. This yields the local
  symmetry of $\bar \phi$.
\end{proof}

We introduce the spectral function $\bar \Phi$ associated with $\bar \phi$
\begin{equation}\label{aris-D11}
  \fonction{\bar \Phi}{\lambda^{-1}(\mathcal{D}) \subset \Sn}
           {N^{\reduc}_{\mathcal{M}}(\bar x) \ \subset \ \mathcal{R}^{d+n^{\reduc}}}
           {X}
           {(\bar \phi \circ \lambda)(X)
             = \bar x + \phi(\bar \pi_T(\lambda(X)))-\bar \pi^{\reduc}_{N}(\lambda(X)).}
\end{equation}
By construction, we get that the zeros of $\bar \Phi$ characterize $\Mm$, since
\begin{equation}\label{local}
  X \in \lambda^{-1}(\mathcal{M}\cap B(\bx,\delta)) \iff \lambda(X)
  \in \mathcal{M}\cap B(\bx,\delta) \iff \bar \Phi(X)=0.
\end{equation}
At this stage, let us compare \eqref{aris-D11} with \eqref{aris-C11}
and the particular treatment in Subsection~\ref{ssec-42}. In
Subsection~\ref{ssec-42} we had $N_{\mathcal{M}}(\bar x)  \subseteq
\Delta(\sigma)^{\bot\bot}$ yielding $N_{\mathcal{M}}^{\reduc}(\bar
x) = N_{\mathcal{M}}(\bar x)$ and thus $\mathcal{D} =
\mathcal{B}(\bar x, \delta_1, \delta_2)$, an open subset of $\RR^n$.
Unfortunately, in the general case, there is an extra
difficulty, which stems from the fact that $\Dd$ is not open in
$\RR^n$, but only relatively open with respect to the affine
subspace~$\mathcal{R}^{d+n^{\reduc}}$, and consequently the function
$\bar\Phi$ is defined in a subset of $\Sn$ of lower dimension
(namely, $\lambda^{-1}(\mathcal{D})$). For this reason, we shall
successively establish the following properties.
\begin{enumerate}
\item {\it Transfer of local approximation.} We show that the set
  $\lambda^{-1}(\mathcal{D})$ is an analytic manifold locally around $\bar X \in
  \lambda^{-1}(\bar x)$ and we calculate its dimension;
\item {\it Transfer of local equation.} We show that the function
  $\bar \Phi$ defined on $\lambda^{-1}(\mathcal{D})$ is
  differentiable and its differential at $\bar X$ (a linear map
  on the tangent space of $\lambda^{-1}(\mathcal{D})$) has a full rank.
\end{enumerate}

%**********************************************************************
\subsection{Transfer of the local approximation} \label{ssec-431}
%**********************************************************************

The goal of this section is to show that locally around $\bar X \in
\lambda^{-1}(\bar x)$ the set $\lambda^{-1}(\mathcal{D})$ is an
analytic submanifold of $\Sn$. We do this in two steps: the first
step consists of showing that both the $M$-part and the $F$-part of
$\Dd$ give rise to two analytic submanifolds in the spaces ${\bf
S}^{n-\kappa_{\ast}}_{\sigma^M}$ and ${\bf
S}^{\kappa_{\ast}}_{\sigma^F}$ correspondingly, while the second
step shows that intertwining the two parts preserves this property
in the space $\Sn$. Throughout this section, we consider that
Assumption~\ref{assmpt-21Nov2008} is in force (and recall \eqref{F1}
and \eqref{equal-sets}).

\begin{lemma}[Decomposition of $\Dd$]\label{jerome-facto}
  Applying the (F,M)-decomposition to the affine manifold $\Dd$,
  we get
  \[
  \mathcal{D}=\big\{x^{F}\otimes x^{M} : \quad x^{F}\in \Dd^F,\ x^{M}\in \Dd^M\big\},
  \]
  where $\Dd^F$ and $\Dd^M$ are affine manifolds defined by:
  \begin{align*}
  \Dd^M &:=
  \lbrack \Delta (\sigma^{M})_{\mathbb{R}^{n-\kappa _{\ast }}}]
    \cap B(\bar{x}^{M},\delta_2), \mbox{ and } \\[0.1cm]
  \Dd^F&:= \big(
     [T_{\mathcal{M}}^{\reduc}(\bar{x})]^{F}
     \oplus
     \lbrack \Delta (\sigma ^{F})_{\mathbb{R}^{\kappa_{\ast }}}] \big)
       \, \cap \, B(\bx^{F},\delta_1),
   \end{align*}
   where $\lbrack T_{\mathcal{M}}^{\reduc}(\bar{x})]^{F}$
   is the $F$-part of the reduced space $T_{\mathcal{M}}^{\reduc}(\bar{x})$.
   The sets $\Dd^M$ and $\Dd^F$ are locally symmetric.
  Moreover, the dimension of $\Dd^M$ is $n-\kappa_*$, while the dimension of
   $\Dd^F$ is
   $$
   \dim \mathcal{D}^F = d+ n^{\reduc} - m.
   $$
\end{lemma}

\begin{proof}
  We deduce from the definition of $T_{\mathcal{M}}^{\reduc}(\bar{x})$
  in \eqref{defn-n1} and by \eqref{arisd10} that for every
  $x=x^{F}\otimes x^{M} \in T_{\mathcal{M}}^{\reduc}(\bar{x})$ we have
  $x^{M}=0$. According to (\ref{eqn-arisd11})
  \[
  \Delta (\sigma )^{\bot \bot } =
  [\Delta  (\sigma^{F})_{\mathbb{R}^{\kappa _{\ast }}}]^{\bot \bot }\otimes
  \lbrack \Delta (\sigma^{M})_{\mathbb{R}^{n-\kappa _{\ast }}}]^{\bot
  \bot },
  \]
  which combined with \propref{jerome-all-about} yields
  \begin{align*}
    &\mathcal{D}=\big\{x^{F}\otimes x^{M} : x^{F}\in \big(\lbrack
    T_{\mathcal{M}}^{\reduc}(\bar{x})]^{F}\oplus \lbrack \Delta (\sigma
      ^{F})_{\mathbb{R}^{\kappa _{\ast }}}]^{\bot \bot }\big)\,\cap
        B(\bar{x}^{F},\delta_1 ), \\
        & \hspace{10cm} x^{M}\in \lbrack \Delta (\sigma
        ^{M})_{\mathbb{R}^{n-\kappa _{\ast }}}]^{\bot \bot }\cap B(\bar{x}^{M},\delta_2
          )\big\}.
  \end{align*}
  Now, in view of Assumption~\ref{assmpt-21Nov2008},
  the closure of the affine space (that is the sign `{\scriptsize $\bot\bot$}')
  is not needed in the above representation; in other terms:
  \begin{align*}
    \big(\lbrack
    T_{\mathcal{M}}^{\reduc}(\bar{x})]^{F}\oplus \lbrack \Delta (\sigma
      ^{F})_{\mathbb{R}^{\kappa _{\ast }}}]^{\bot \bot }\big)\,\cap
        B(\bar{x}^{F},\delta_1 ) &\ =\   \big(\lbrack
        T_{\mathcal{M}}^{\reduc}(\bar{x})]^{F}\oplus \lbrack \Delta (\sigma
          ^{F})_{\mathbb{R}^{\kappa _{\ast }}}]\big)\,\cap
            B(\bar{x}^{F},\delta_1 )\\
            \lbrack \Delta (\sigma^{M})_{\mathbb{R}^{n-\kappa _{\ast }}}]^{\bot \bot }\cap B(\bar{x}^{M},\delta_2) &\ =\   \lbrack \Delta (\sigma
              ^{M})_{\mathbb{R}^{n-\kappa _{\ast }}}]\cap B(\bar{x}^{M},\delta_2).
  \end{align*}
    Hence, we get the desired expressions for $\mathcal{D}^F$ and $\mathcal{D}^M$.
    By Proposition~\ref{jerome-all-about}, the set $\mathcal{D}$ is invariant under all
    permutations in $S^{\succsim}(\sigma)$. Thus, by Proposition~\ref{prop-arisd2}, being the $F$- and $M$-parts of
    $\mathcal{D}$, the sets
    $\mathcal{D}^F$ and $\mathcal{D}^M$ invariant with respect to the permutations
   in $S^{\succsim}(\si^F)$ and $S^{\succsim}(\si^M)$, respectively. We now compute
    the dimension of $\Dd^F$. Observe that
    Proposition~\ref{jerome-all-about} yields
  \begin{align*}
    \bar x + T_{\mathcal{M}}(\bar{x})\,\oplus\, N_{\Mm}^{\reduc}(\bar{x})
    &=T^{\reduc}_{\mathcal{M}}(\bar{x})\,\oplus \,\Delta (\sigma)^{\bot \bot } \\
    &=\big( [T_{\mathcal{M}}^{\reduc}(\bar{x})]^{F}\oplus
    \lbrack \Delta (\sigma ^{F})_{\mathbb{R}^{\kappa _{\ast }}}]^{\bot
        \bot }\big) \otimes \big( \{0\}_{n-\kappa _{\ast}}\oplus \lbrack
      \Delta (\sigma ^{M})_{\mathbb{R}^{n-\kappa _{\ast }}}]^{\bot \bot}\big).
  \end{align*}
  Thus, using \eqref{aris-new}, \eqref{defn-dim-n1} and the fact that
  $m=\dim \left( [\Delta (\sigma^{M})_{\mathbb{R}^{n-\kappa_{\ast }}}]^{\bot \bot }\right)$,
  we get
  \[
  d\,+\,n^{\reduc}\,=\dim \mathcal{D}^F\,+\,m\,,
  \]
  which ends the proof.
\end{proof}

In the following two lemmas, we show that the two parts of $\Dd$
lift up to two manifolds $\lambda^{-1}_{\sigma^M}\big(\Dd^M\big)$
and $\lambda^{-1}_{\sigma^F}\big(\Dd^F\big)$. Let us start with the
easier case concerning the $M$-part.

\begin{lemma}[The analytic manifold $\mathcal{S}^{M}$]
  \label{main-res-simple-form-jeromeM}
  Let $\bar x \in \mathcal{M} \cap \Delta(\sigma)$ and
  let $\sigma = \sigma^F \circ \sigma^M$ be the $(F,M)$-decomposition
  of $\sigma$. Then, the
  set
  \[
    \mathcal{S}^{M} :=\lambda^{-1}_{\sigma^M}\big(\Dd^M\big)
    \quad \subset {\bf S}^{n-\kappa_{\ast}}_{\sigma^M}
  \]
  is an analytic submanifold of ${\bf S}^{n-\kappa_{\ast}}_{\sigma^M}$ around
  $\bar X_{\si^M}^M \in \lambda^{-1}_{\sigma^M}(\bar x^M)$,
  whose codimension is
  \[
  \sum_{i=1}^{m} \frac{|I_{\kappa+i}|(|I_{\kappa+i}|+1)}{2} - m.
  \]
\end{lemma}

\begin{proof}
  According to the partition $P(\si^M) =\{I_{\kappa+1},\ldots,I_{\kappa+m}\}$,
  a vector in $[\Delta (\sigma^{M})_{\mathbb{R}^{n-\kappa_{\ast }}}]$ has equal
  coordinates within each block $I_{\kappa+i}$.  Each
  block lifts to a multiple of the identity matrix (in the appropriate
  space).  Since the lifting $\lambda^{-1}_{\sigma^M}$ is
  block-wise, $\mathcal{S}^{M}$ is then a direct product of multiples
  of identity matrices, and thus an analytic submanifold of ${\bf
    S}^{n-\kappa_{\ast}}_{\sigma^M}$ with dimension $m$.
 \end{proof}

Let us now deal with the $F$-part.

\begin{lemma}[The analytic manifold $\mathcal{S}^{F}$]\label{main-res-simple-form-jeromeF}
  Let $\bar x \in \mathcal{M} \cap \Delta(\sigma)$,
  and $\sigma = \sigma^F \circ \sigma^M$ be the $(F,M)$-decomposition
  of $\sigma$. Then the set
  \[
  \mathcal{S}^{F} :=\lambda ^{-1}_{\sigma^F}(\Dd^F)
  \quad \subset {\bf S}^{\kappa_{\ast}}_{\sigma^F}
  \]
  is an analytic submanifold around $\bar X_{\si^F}^F \in \lambda^{-1}_{\sigma^F}(\bar x^F)$
  of codimension $\kappa_{\ast}-(d+ n^{\reduc} - m)$.
\end{lemma}

\begin{proof}
  Recall that by \lemref{jerome-facto}, $\Dd^F$ is a locally
  symmetric, affine submanifold of $\RR^{\kappa_{\ast}}$.
  Our first aim here is to show that
  \begin{equation}\label{claim}
  N_{\Dd^F}(\bar x^F) \subset [\Delta(\sigma^F)_{\R^{\kappa_*}}]^{\bot\bot}.
  \end{equation}
  (Compare \eqref{claim} with \eqref{eqn-4Dec07-1}.)
  To this end, fix $\epsilon > 0$ and let $\omega \in
  T_{\mathcal{M}}(\bar x) \cap B(0, \epsilon)$ be a vector with
  the properties stated in Lemma~\ref{lem-hss}.  By
  (\ref{tang_decomp}), there is a unique representation $\omega =
  \omega_{\bot} + \omega_{\bot\bot}$ for some $\omega_{\bot} \in
  T_{\mathcal{M}}^{\reduc}(\bar{x})$ and $\omega_{\bot \bot} \in
  T_{\mathcal{M}}(\bar{x}) \cap \Delta(\sigma)^{\bot\bot}$.
  Taking the $F$-trace of $w$, we have $\omega^F = \omega_{\bot}^F +
  \omega_{\bot\bot}^F$ with
  $$
  \omega_{\bot}^F \in [T_{\mathcal{M}}^{\reduc}(\bar{x})]^F
  $$
  and $\omega_{\bot\bot}^F \in
  [\Delta (\sigma^{F})_{\mathbb{R}^{\kappa_{\ast }}}]^{\bot \bot }$.
  Let $P(\sigma^F)=\{I_1,\ldots,I_{\kappa}\}$ be the partition determined
  by $\sigma^F$. Note that $\omega_{\bot}^F = \omega^F -
  \omega_{\bot\bot}^F$. Since subvector $\omega_{I_i}^F$
  has distinct coordinates, while
  $(\omega_{\bot\bot}^F)_{I_i}$ has equal coordinates (definition
  of $\lbrack \Delta(\sigma ^{F})_{\mathbb{R}^{\kappa_{\ast }}}]^{\perp\perp}$),
  we conclude that the subvector $(\omega_{\bot}^F)_{I_i}$ has
  distinct coordinates, for all $i \in \N_m$.

  Let us now consider $\Dd^F$. Fix any $x^F \in
  \lbrack \Delta (\sigma ^{F})_{\mathbb{R}^{\kappa_{\ast }}}] \cap B(\bar x^F,\delta_1)$.
  Taking $\omega$ close enough
  to $0$ ensures that $\omega_{\bot}^F$ is close enough to $0$
  so that all of the coordinates of the vector
  $\omega_{\bot}^F + x^F$ are distinct, and moreover $\omega_{\bot}^F + x^F \in
  \Dd^F$.
  All that shows
  \[
  \Dd^F \cap [\Delta({\rm id}_{\kappa_*})_{\R^{\kappa_*}}] \not= \emptyset.
  \]
  Thus, applying Corollary~\ref{Ithoughtthiswasover} (for
  $n=\kappa^*$), we see that the characteristic permutation of the
  affine manifold is $\Dd^F$ is ${\rm id}_{\kappa_*}$ entailing a
  trivial $(F,M)$-decomposition of $\RR^{\kappa_*}$.  The inclusion
  \eqref{claim} now follows from \thref{jm-double} applied to $\RR^{\kappa_*}$.

  To conclude, we apply
  Theorem~\ref{main-thm-4Dec2007-1} and Remark~\ref{anotherremark}
  to $\Dd^F$ to get that the set
  $\mathcal{S}^{F}$
  is an analytic submanifold of
  ${\bf S}^{\kappa_{\ast}}_{\sigma^F}$ of
  codimension $\kappa_{\ast}-(d+ n^{\reduc} - m)$ there.
\end{proof}

\begin{theorem}[$\lambda^{-1}(\mathcal{D})$ is a manifold in $\Sn$]
  \label{main-res-simple-form-2}
  Under Assumption~\ref{assmpt-21Nov2008}, consider
  the set $\Dd$ defined by~\eqref{Dcal}.
  Then the  set $\lambda^{-1}(\mathcal{D})$
  is an analytic submanifold of $\Sn$ around
  $\bar X\in \la^{-1}(\bar x)$, with dimension
  \begin{equation}\label{kesaco}
  \dim \lambda^{-1}(\mathcal{D})=\frac{n(n+1)}{2} + d + n^{\reduc} - \kappa_*
  - \sum_{i=1}^{m} \frac{|I_{\kappa+i}|(|I_{\kappa+i}|+1)}{2}.
  \end{equation}
\end{theorem}

\begin{proof}
  Consider the $(F,M)$-decomposition of $\RR^n$ induced by $\si_*$, and
  apply \propref{real-Fred} to get a neighborhood $W$ of $\bar X$
  in $\Sn$ and analytic maps $\Theta^F$ and $\Theta^M$
  such that
  \begin{align}
    \label{bozakosmata}
    \lambda(X) =\lambda_{\sigma^F}(\Theta^F(X))
    \otimes \lambda_{\sigma^F}(\Theta^M(X))
    \qquad \mbox{for all }X \in W.
  \end{align}
  Set  $\bar X_{\si^F}^F := \Theta^F(\bar X) \in {\bf S}^{\kappa_*}_{\sigma^F}$
  and $\bar X_{\si^M}^M := \Theta^M(\bar X)\in {\bf S}^{n-\kappa_*}_{\sigma^M}$.
  Since $\bar x = \lambda(\bar X)
  = \lambda_{\sigma^F}(\bar X_{\si^F}^F) \otimes \lambda_{\sigma^M}(\bar X_{\si^M}^M)$,
  by the fact that the canonical product is well-defined,
  we deduce $\bar x^F = \lambda_{\sigma^F}(\bar X_{\si^F}^F)$
  and  $\bar x^M = \lambda_{\sigma^M}(\bar X_{\si^M}^M)$,
  concluding that $\bar X_{\si^F}^F \in \mathcal{S}^F$
  and $\bar X_{\si^M}^M \in \mathcal{S}^M$
  (recall Lemma~\ref{main-res-simple-form-jeromeF} and
  Lemma~\ref{main-res-simple-form-jeromeM}).
  Consider the respective codimensions
  \begin{align}
  \label{codim-s1}
    s_1:=\mbox{co-dim\,} \mathcal{S}^F&=\kappa_{\ast}-(d+n^{\reduc}-m), \mbox{ and }\\
\label{codim-s2}
    s_2:=\mbox{co-dim\,} \mathcal{S}^M&= \sum_{i=1}^{m} \frac{|I_{\kappa+i}|(|I_{\kappa+i}|+1)}{2} - m.
  \end{align}
  Since the maps $\Theta^F$ and $\Theta^M$ have Jacobians of full rank
  at $\bar X$, they are open around it.  By shrinking $W$ if
  necessary, we may assume there exist analytic maps
  $$
  \Psi^F\colon \Theta^F(W) \rightarrow \RR^{s_1}\,\qqandqq \Psi^M\colon
  \Theta^M(W) \rightarrow \RR^{s_2},
  $$
  with Jacobians having full rank at
  $\bar X_{\si^F}^F$ and $\bar X_{\si^M}^M$ respectively, such that
  $$
  \Psi^F(X_{\si^F}^F)=0 \,\Leftrightarrow\,X_{\si^F}^F \in
  \mathcal{S}^F \cap \Theta^F(W)
  \qquad\text{and}\qquad
  \Psi^M(X_{\si^M}^M)=0\,\Leftrightarrow\, X_{\si^M}^M \in \mathcal{S}^M \cap \Theta^M(W).
  $$
  Together, the two conditions above are equivalent to
  \[
  X_{\si^F}^F\times X_{\si^M}^M\in \Theta^F(W) \times \Theta^M(W)
  \qqandqq
  \lambda_{\sigma^F}(X_{\si^F}^F)\otimes \lambda_{\sigma^M}(X_{\si^M}^M)
  \in \mathcal{D}.
  \]
  We now define a local equation for $\lambda^{-1}(\mathcal{D})$
  around $\bar X$ as follows:
  \begin{equation*}
    \fonction{\Psi}{W\subset \Sn}{\mathbb{R}^{s_{1}}\times \mathbb{R}^{s_{2}}}
    {X}{(\Psi^F \circ \Theta^F)(X)\times (\Psi^M\circ\Theta^M)(X).}
  \end{equation*}
  Indeed, using (\ref{bozakosmata}), for all $X \in W$ we have
  \[
  \Psi(X)=0 \iff
  \lambda(X)=\lambda_{\sigma^F}(\Theta^F(X))
  \otimes\lambda_{\sigma^M}(\Theta^M(X)) \in \mathcal{D}
  \iff
  X\in \lambda^{-1}(\Dd).
  \]
  The fact that the Jacobian of $\Psi$ has full rank at $\bar X$
  follows from the chain rule and the fact that all the Jacobians
  $J\Theta^F(\bar X)$, $J\Theta^M(\bar X)$, $J\Psi^F(\bar X_{\si^F}^F)$, and
  $J\Psi^M(\bar X_{\si^M}^M)$ are of full rank. Thus, $\Psi$ is an analytic
  local equation of $\lambda^{-1}(\mathcal{D})$ around $\bar X$,
  which yields that $\lambda^{-1}(\mathcal{D})$ is a submanifold~$\Sn$~around~$\bar X$.
  We compute its dimension as follows
  \begin{align*}
    \dim \lambda^{-1}(\mathcal{D})
    &= \dim \Sn - \big(\mbox{co-dim\,} \mathcal{S}^F
    + \mbox{co-dim\,} \mathcal{S}^M\big)\\
    & = \frac{n(n+1)}{2} + d + n^{\reduc} - \kappa_*
    -\sum_{i=1}^{m} \frac{|I_{\kappa+i}|(|I_{\kappa+i}|+1)}{2},
  \end{align*}
  where \eqref{codim-s1} and \eqref{codim-s2} were used.
\end{proof}

Theorem~\ref{main-res-simple-form-2} is an important intermediate
result for the forthcoming Section~\ref{ssec-45}, which contains the
final step of the proof. Nonetheless, in the following particular
case, Theorem~\ref{main-res-simple-form-2} allows us to conclude
directly.

\pagebreak

\begin{example}
\label{aris-27mars09}
Fix a permutation $\sigma_* \in \Sigma^n$ with the property described in Theorem~\ref{consec-sets}. In view of Remark~\ref{rem-8Apr2009}, it is instructive to consider the particular case when
$$
\Mm = \bigcup_{{\scriptsize \begin{array}{c}\sigma \sim \si_{\ast}\\[-0.5ex] \si\prec  \hspace{-0.1cm} \sim \sigma_{\ast}\end{array}}} \hspace*{-1ex}\Delta(\sigma).
$$ Clearly, $\Mm$ is a locally symmetric manifold with characteristic
permutation $\sigma_*$ and relatively open in~$\Delta(\sigma_*)^{\bot\bot}$.  Moreover, for any $\bar x \in \Mm \cap
\Delta(\sigma)$, where $\sigma \sim \sigma_*$ or
$\si\prec \hspace{-0.1cm} \sim \sigma_{\ast}$, we have
$N_{\Mm}^{\reduc}(\bar x)=\{0\}$, that is $n^{\reduc} =0$.  This means
that the affine manifolds $\Mm$ and $\mathcal{D}$ coincide locally
around $\bar x$, see~\eqref{Dcal}. In this case
Theorem~\ref{main-res-simple-form-2} shows directly that
$\lambda^{-1}(\Mm)$ is a manifold in $\Sn$ with dimension given by
(\ref{kesaco}). At first glance, it appears that the dimension depends
on the particular choice of~$\bar x$. But since $\sigma \sim \sigma_*$
or $\si\prec \hspace{-0.1cm} \sim \sigma_{\ast}$ we recall that we
have $d = \kappa_*+ m_*$, $m = m_*$, and $|I_{\kappa+i}| =
|I^*_{\kappa_*+i}|$ for all $i=1,\ldots,m$. Thus, the dimension depends
only on $\sigma_*$. In fact, one can verify that (\ref{kesaco})
becomes
\begin{align*}
\dim \lambda^{-1}(\Mm) &= d + {\kappa_* \choose 2} + \kappa_*(n-\kappa_*) + \sum_{1 \le i < j \le m_*} |I^*_{\kappa_*+i}||I^*_{\kappa_*+j}| \\
&= d + \sum_{1 \le i < j \le \kappa_*+m_*} |I^*_{i}||I^*_{j}|.
\end{align*}
Thus, according to (\ref{aris-dime}), we have
$$
\dim \lambda^{-1}(\Mm) = \dim \lambda^{-1}(\Delta(\sigma_*)),
$$
and that is a particular case of the forthcoming general formula (\ref{dim-lamM}).\qed
\end{example}

In the situation of Example~\ref{aris-27mars09} the manifold $\Mm$
has a trivial reduced normal space. The following remark
sheds more light on this aspect.

\begin{remark}[Case of trivial reduced normal space ($N_{\Mm}^{\reduc}(\bar x)=\{0\})$]
Let $\Mm$ be a locally symmetric manifold, with characteristic permutation $\sigma_{\ast}$ and let $\bar{x}\in \Mm \cap \Delta(\sigma)$. Then, by (\ref{prop-30June07-2}) and (\ref{aris100}), it can be
easily seen that
\begin{align*}
  N_{M}^{\reduc}(\bar{x})=\{0\}  \, &\Longleftrightarrow \, \Mm\cap B(\bar{x},\delta)=(\bar x+T_{\mathcal{M}}(\bar x)) \cap B(\bar x, \delta),\, \text{for some } \delta>0.
\end{align*}
Applying Corollary~\ref{density-hss} to the left-hand side of the last equality,
 we see on the right-hand side that $(\bar x+T_{\mathcal{M}}(\bar x)) \cap B(\bar x, \delta) \cap  \Delta(\sigma_*)$ is dense in $(\bar x+T_{\mathcal{M}}(\bar x)) \cap B(\bar x, \delta)$. Inclusions (\ref{eqn-28Nov2007-1}) and (\ref{32}) show that $(\bar x+T_{\mathcal{M}}(\bar x)) \subset  \Delta(\sigma_*)^{\bot\bot}$, thus we obtain:
\begin{align*}
  N_{M}^{\reduc}(\bar{x})=\{0\}  \, &\Longleftrightarrow \, \Mm\cap B(\bar{x},\delta)= \Delta(\sigma_*)^{\bot\bot} \cap B(\bar x, \delta),\, \text{for some } \delta>0.
\end{align*}
There are two possibilities with respect to the position of $\bar x$:
\begin{itemize}
  \item If $\bar{x}\in\Delta(\sigma_{\ast})$, then we can shrink
  $\delta>0$ to get $\Mm\cap B(\bar x, \delta)=\Delta(\sigma_{\ast})\cap B(\bar x,\delta)$.
  This is the situation, for instance, in Example~\ref{jerome-lift}.
  \item If $\bar{x}\notin\Delta(\sigma_{\ast})$, then $\bar{x}\notin\Delta(\sigma)$ for some
$\sigma \prec \hspace{-0.1cm}\sim  \sigma_{\ast}$.
This is the situation, for instance, in  Example~\ref{aris-27mars09}.
\end{itemize}
\end{remark}

%**********************************************************************
\subsection{Transfer of local equations, proof of the main result}\label{ssec-45}
%**********************************************************************

This section contains the last step of our argument: we show
that \eqref{aris-D11} is indeed a local equation of $\Mm$ around
$\bar X \in \la^{-1}(\bar x)$.

\begin{lemma}[The Jacobian of $D\bar\Phi(\bar X)$]\label{lemma-45}
  The map $\bar \Phi$ defined in \eqref{aris-D11} is of class $C^2$ at
  $\bar X$. Denoting~by
  $$
  D\bar\Phi(\bar X)\colon T_{\lambda^{-1}(\mathcal{D})}(\bar X)\longrightarrow
  N^{\reduc}_{\mathcal{M}}(\bar x)
  $$
  the differential of $\bar \Phi$ at $\bar X$, we
  have for any direction $H \in T_{\lambda^{-1}(\mathcal{D})}(\bar X)$:
  \begin{align}
    \label{gradient-of-Phi}
    D\bar \Phi(\bar X)\,[H]
    =  D\phi\,(\bar \pi_T(\lambda(\bar X)))\,[\pi_T(\diag
      (\bar{U}\,H\,\trans{\bar{U}}))]\,
    - \,\pi^{\reduc}_{N}( \diag (\bar{U}\,H\,\trans{\bar{U}})),
  \end{align}
  where $\bar U \in \On$ is such that $\bar X = \trans{\bar
  U}(\Diag\,\lambda(\bar X)) \bar U$.
\end{lemma}

\begin{proof}
  We deduce from Corollary~\ref{corollary_pi} and
  Lemma~\ref{important-observation} that for any $\sigma' \succsim
  \sigma$ and $x \in \mathcal{D}$ we have
  \begin{align}
    \label{sym-phi}
    (\phi \circ \bar \pi_T)(\sigma'x)= (\phi \circ \bar \pi_T)(x).
  \end{align}
  In addition, the gradient of the $i$-th coordinate function $(\phi_i
  \circ \bar \pi_T)(x)$ at $\bar x$, applied to any direction $h\in
  T_{\mathcal{D}}(\bx)=T_{\Mm}^{\reduc}(\bx)\oplus\Delta(\si)^{\bot\bot}$, see (\ref{aris-new}), yields
  $$
  \nabla\,(\phi_i \circ \bar \pi_T)(\bar x)[h] = \nabla \phi_i(\bar
  \pi_T(\bar x))[\pi_T(h)].
  $$
  Thus, by Theorem~\ref{thm-26Nov2007-1},
  we obtain the following expression for the gradient at $\bar X$ of
  the function $X \mapsto (\phi_i \circ \bar \pi_T)(\lambda(X))$
  applied to the direction $H\in T_{\lambda^{-1}(\mathcal{D})}(\bar X)$
  $$
  \nabla\, (\phi_i \circ \bar \pi_T \circ \lambda)(\bar X)[H] =\,
  \nabla\phi_i\,(\bar \pi_T(\lambda(\bar X)))\,[\pi_T(\diag
    \,(\bar{U}H\trans{\bar{U}}))], \quad \mbox{for } i\in\N_n,
  $$
  where $\bar U \in \On$ is such that $\bar X = \trans{\bar
    U}(\Diag\,\lambda(\bar X)) \bar U$.
  Since
  $N^{\reduc}_{\mathcal{M}}(\bar x) \subseteq
  \Delta(\sigma)^{\perp\perp}$, we observe that the proof
  of \lemref{lem-26Nov2007-4}  can be readily adapted to find the
  Jacobian of  $\bar \pi^{\reduc}_{N}\circ\lambda$ at~$\bar X$.
  We thus obtain \eqref{gradient-of-Phi}.
\end{proof}

We now show that the differential of $\bar \Phi$ at $\bar X$ is of
full rank. We accomplish this without actually computing the tangent
space of the manifold $\lambda^{-1}(\mathcal{D})$ at $\bar X$.
Instead we show that the tangent space is sufficiently rich to
guarantee surjectivity.

\begin{lemma}[Surjectivity of $D \bar \Phi(\bar X)$]\label{lemma-46}
  The linear mapping (the differential of $\bar\Phi$ at $\bar X$)
  $$
  D\bar
  \Phi(\bar X) \colon T_{\lambda^{-1}(\mathcal{D})}(\bar X)\longrightarrow
  N^{\reduc}_{\mathcal{M}}(\bar x)
  $$
  is onto, and thus has full rank.
\end{lemma}

\begin{proof}
  Let $\bar U\in \On$ be such that
  $\bar X = \trans{\bar U} (\Diag \, \lambda(\bar X))\,\bar U$.
  The tangent space of $\On$ at $\bar U$ is
  $$
  \{\bar U A : A \mbox{ is an $n\times n$ skew-symmetric matrix}\}.
  $$
  Thus, for any $n \times n$ skew symmetric matrix $A$ there exists an
  analytic curve $t \mapsto U(t) \in \On$ such that
  \[
  U(0)=\bar U
  \qquad\text{and}\qquad \dot{U}(0):=\frac{d}{dt}U(0) = \bar UA.
  \]
  Fix now any vector $h \in N_{\mathcal{M}}^{\reduc}(\bar x)$. Consider
  the curve $t \mapsto \trans{U(t)} (\Diag \,(\bar x + t h))U(t)$. For
  all values of $t$ close to zero, this curve lies in
  $\lambda^{-1}(\mathcal{D})$ because $\bar x+ th$ lies in $\Dd$.
  Introduce the vector $x_t$ made of
  the entries of $\bar x + th$ reordered in decreasing way.
  Since the space
  $N_{\Mm}^{\reduc}(\bar x)$ is invariant under all permutations
  $\sigma' \succsim \sigma$ we see that $x_t$ lies in
  $\bx\,+\,N_{\mathcal{M}}^{\reduc}(\bar x)$, for $t$ close to
  zero. The derivative of this curve at $t=0$ ({\it i.e.} a tangent
  vector in $T_{\lambda^{-1}(\mathcal{D})}(\bar X)$) is
  \begin{align*}
    H&:= \trans{\dot U(0)} (\Diag \,\bar x )U(0) + \trans{U(0)} (\Diag \, h) U(0)
    + \trans{U(0)} (\Diag \,\bar x )\dot U(0) \\
    &= -A \trans{\bar U} (\Diag \,\bar x ) \bar U + \trans{\bar U} (\Diag \, h)\bar U
    + \trans{\bar U}(\Diag \,\bar x )\bar UA,
  \end{align*}
  where we use that $\trans{A}=-A$. Substituting the above expression of $H$ into
  \eqref{gradient-of-Phi}, and using the fact that $\bar U \trans{\bar
    U}= \trans{\bar U} \bar U= I$ and that $\bar U A \trans{{\bar U}}
  (\Diag \,\bar x )$ and $(\Diag \,\bar x )\bar UA\trans{U}$ have the
  same diagonal we obtain
  $$
  D \bar \Phi(\bar X)[H] = -h.
  $$
  This shows that $D\bar\Phi(\bar X)$ is
  surjective onto $N^{\reduc}_{\mathcal{M}}(\bar x)$, which completes the proof.
\end{proof}

\begin{theorem}[Main result: $\lambda^{-1}(\Mm)$ is a $C^2$ manifold in $\Sn$]\label{main-thm-4Dec2007-3}
  Suppose $\mathcal{M}$ is a locally symmetric $C^2$ submanifold of
  $\R^n$ of dimension $d$. Then $\lambda^{-1}(\mathcal{M})$ is a
  $C^2$ submanifold of $\Sn$ of dimension
  \begin{equation}\label{dim-lamM}
    \dim \lambda^{-1}(\Mm)\,=\,  d\, +\, \sum_{1 \le i < j \le
      \kmin+m_{\ast}} |\,I_{i}^{\ast}\,|\,|\,I_{j}^{\ast}\,|,
  \end{equation}
  where $\sigma_{\ast}$ is the characteristic permutation of $\Mm$ and
  $P(\si_{\ast})=\{I_{1}^{\ast},\ldots,I_{\kmin+m_{\ast}}^{\ast}\}$.
\end{theorem}

\begin{proof}
  Fix any $\bar x \in \mathcal{M} \cap \R^n_{\ge}$ and $\bar X \in
  \lambda^{-1}(\bar x)$ and consider the spectral function $\bar
  \Phi$ introduced in~\eqref{aris-D11}. Equation \eqref{local} shows that $\bar \Phi$ is a
  local equation of $\Mm$.
  Lemmas~\ref{lemma-45} and \ref{lemma-46} prove
  that  $\bar \Phi$ is a $C^2$ local equation of $\la^{-1}(\Mm)$ around $\bar X$.
  Thus $\la^{-1}(\Mm)$ is a $C^2$ submanifold of $\Sn$ around~$\bar X$.
  Moreover, the dimension of $\lambda^{-1}(\Mm)$ is
  $$
  \dim \lambda^{-1}(\Mm)\,
  =\,\dim \lambda^{-1}(\mathcal{D})\,-\,\dim (N_{\Mm}^{\reduc}(\bx)).
  $$
  Using \eqref{defn-n1} and \thref{main-res-simple-form-2}, we get
  $$
  \dim \lambda^{-1}(\mathcal{M}) = d+ \frac{n(n+1)}{2}  - \kappa_*
  -\sum_{i=1}^{m} \frac{|I_{\kappa+i}|(|I_{\kappa+i}|+1)}{2}.
  $$
  Recall that $\si^M=\si^M_*$ (\propref{prop-arisd1}),
  so that $|I_{\kappa+i}| = |I^*_{\kappa_*+i}|$ for all
  $i=1,\ldots,m$, that $m = m_*$, and that $\sum_{i=1}^{m_*}
  |I^*_{\kappa_*+i}| = n - \kappa_*$.  Substituting this in the above equality,
  we obtain
  \begin{align*}
    \dim \lambda^{-1}(\mathcal{M}) &= d+ \frac{n^2}{2}  - \frac{\kappa_*}{2}
    -\sum_{i=1}^{m_*} \frac{|I^*_{\kappa_*+i}|^2}{2} \\
    &=  d+ \frac{n^2}{2}  - \frac{\kappa_*}{2}  - \frac{1}{2} \Big( \sum_{i=1}^{m_*} |I^*_{\kappa_*+i}| \Big)^2
    + \sum_{1 \le i < j \le m_{\ast}} |I^*_{\kappa_*+i}||I^*_{\kappa_*+j}| \\
    %&= d+ \frac{n^2}{2}  - \frac{\kappa_*}{2}  - \frac{1}{2} ( n- \kappa_*)^2
    %+ \sum_{1 \le i < j \le m_{\ast}} |I^*_{\kappa_*+i}||I^*_{\kappa_*+j}| \\
    &= d+\frac{\kmin(\kmin-1)}{2} + \kmin(n-\kmin)  + \sum_{1 \le i < j \le m_{\ast}} |I^*_{\kappa_*+i}||I^*_{\kappa_*+j}| \\
    &= d+\sum_{1 \le i < j \le \kmin+m_{\ast}} |I_{i}^{\ast}|\,|I_{j}^{\ast}|,
  \end{align*}
  the last equality coming from the fact that, by definition
  (\ref{part-of-sigmamin}), all the sets in
  $\{I_1^{\ast},\ldots,I_\kmin^{\ast}\}$ have size one.
\end{proof}

Notice that the dimension \eqref{dim-lamM}
of $\la^{-1}(\Mm)$ depends only on the dimension of
the underlying manifold $\Mm$ and its characteristic permutation $\si_*$.
This is not the case with the dimension \eqref{kesaco} of $\la^{-1}(\Dd)$
which also depends on the active permutation $\si$
(by $n^{\reduc}$, $\kappa$ and $m$).

\begin{remark}[Variants of the main result] Theorem~\ref{main-thm-4Dec2007-3}
  has been announced and proved for the $C^2$ case. Let us now see
  what can be said in other cases:
  \begin{itemize}
  \item[(i)] {\bf [\,$C^\infty$ and $C^{\omega}$\,]} The statement of
    Theorem~\ref{main-thm-4Dec2007-3} holds true in these two cases.
    In particular, we have:
    $\lambda^{-1}(\Mm)$ is a $C^{\infty}$ (respectively, analytic)
    submanifold of $\Sn$, whenever $\Mm$ is a $C^{\infty}$
    (respectively, analytic) locally symmetric submanifold of $\RR^n$.
    The proof is identical.
  \item[(ii)] {\bf [\,$C^k$ case, $k\notin\{1,2,\infty, \omega\}$\,]}
    It is not known whether or not the transfer principle of
    \thref{thm-26Nov2007-1} remains true for the general $C^k$ case,
    for $k\notin\{1,2,\infty\}$. If such a statement is true, then
    Theorem~\ref{main-thm-4Dec2007-3} will also hold for the $C^k$
    case ($k\geq 2$) with the same proof (as in (ii)).
  \item[(iii)] {\bf [\,$C^1$ case\,]} The $C^1$ case seems somehow compromised
    by the use of Lemma~\ref{lem-isoms} (Determination of isometries).
    Indeed, the aforementioned lemma uses the intrinsic Riemannian
    structure of $\Mm$ (which demands an at least $C^2$ differentiable
    structure for $\Mm$). Thus, our method does not apply for this case.
\end{itemize}
\end{remark}

%%****************************************************************************
%\subsection{Examples}
%%****************************************************************************

%In this section we assemble some examples of spectral manifolds that
%can be constructed by the lift-up procedure (as an applications of
%Theorem~\ref{main-thm-4Dec2007-3}).

%\begin{example}[Lift-up of strata and their closures]
%Let $\si\in\Sigma^n$. Then $\lambda^{-1}(\Delta(\si))$
%(Example~\ref{jerome-lift}) and
%$\lambda^{-1}(\Delta(\si)^{\bot\bot})$ (Example~\ref{aris-27mars09})
%are spectral manifolds. \qed
%\end{example}

\begin{example}[Matrices of constant rank in $\Sn$]
  Let $r\in \{0,1,\dots, n\}$ and let us consider the subspace $\Sn_r$
  of $\Sn$ consisting of all symmetric matrices of constant rank
  $r$. We show here that this set is a spectral manifold of dimension
  $r(2n-r+1)/2$ around a matrix $\bar X\in\Sn_r$.

  Let $\bar x\in \la(\bar X) \in \RR^n_{\geq}$ and set
  $I=\{i\in\NN_n:\ \bar x_i=0\}$.
  Let $\de = \min\{|\bar x_i|: i\in \NN_n\setminus I\}$
  and denote by~$\Nn$ the set of vectors
  of $\RR^n$ with exactly $r$ non-zero entries. Observe that the set
  $\Mm= \Nn \cap B(\bar x,\de/2)$ is a linear submanifold of $\RR^n$
  of dimension $r$ around $\bx$, with the
  $(n-r)$-local equations $x_i=0$ for $i\in I$ there.
  It is also locally symmetric with characteristic
  permutation $\sigma_{\ast}=(i_1, \dots, i_r)$ for $i_k\in I$ ($k=1,\ldots,r$).
  Thus, by Theorem~\ref{main-thm-4Dec2007-3}, $\la^{-1}(\Mm)$ is a
  submanifold of $\Sn$ around $\bar X$ with dimension
  \[
  \dim \la^{-1}(\Mm) = r + \frac{r(r-1)}{2} + r(n-r) =
  \frac{r(2n-r+1)}{2}.
  \]
  We retrieve in particular easily the dimensions of the particular
  cases $r=1$ (rank-one matrices) and $r=n$ (invertible matrices).
  \qed
\end{example}

\begin{remark}[The case $\kmin \in \{0,1\}$]
  If $\mathcal{M}$ is a connected, submanifold of $\RR^n$ of dimension
  $d$, such that $\kmin \in \{0,1\}$, then
  $\Mm\subset\Delta(\si_{\ast})$. The same arguments as in
  Example~\ref{jerome-lift} allow to conclude that $\la^{-1}(\Mm)$ is
  a spectral manifold of dimension given by \eqref{aris-dime}.\qed
\end{remark}

%%**********************************************************************
%%**********************************************************************
\section{Appendix: A few side lemmas}
%%**********************************************************************
%%**********************************************************************

This appendix section contains a few results that were not central to the development, but are necessary for the proof of the main theorem.

Let $y_1,\ldots,y_n$ be any reals and let $y=(y_1,\ldots,y_n)$.
Consider the $(n!+1)\times (n+1)$ matrix~$Y$ with first
row $(1,\ldots,1,0) \in \R^{n+1}$
and consecutive rows equal to $(\sigma y, 1)$ for each $\sigma \in \Sigma^n$. For example, when $n=2$ the matrix $Y$ is $3\times 3$ and equal to
$$
\left(
\begin{array}{ccc}
1 & 1  & 0 \\
y_1 & y_2 & 1 \\
y_2 & y_1 & 1 \\
\end{array}
\right).
$$

\begin{lemma}[Matrix of full rank]\label{lem-29Jul07-2}
If for $n \ge 2$ the numbers $y_1,\ldots,y_n$ are not all equal, then
the matrix $Y$ defined above has full rank.
\end{lemma}

\begin{proof}
Suppose that $(x, \alpha) \in \R^n \times \R$ is in the null space
of $Y$. Then, $\trans{y} P x + \alpha = 0$ for all permutation
matrices $P$ and $x_1+\cdots + x_n = 0$. Hence, $\trans{y}(P-Q)x=0$
for all permutation matrices~$P$ and $Q$. Without loss of
generality, $y_1 \not= y_2$. For any distinct indices $r$ and $s$,
choose $P$ and $Q$ so that $(P-Q)x=(x_r-x_s,x_s-x_r,0,\ldots,0)$. This
shows that $x_s=x_r$. Since $r$ and $s$ are arbitrary, we deduce
$x=0$ and hence $\alpha=0$, as required.
\end{proof}

The following result is used
in the proof of Theorem~\ref{jm-double}.

\begin{corollary}\label{cor-29Jun07-1}
Let $x \in \Delta(\sigma)^{\perp}$ for some $\sigma \in \Sigma^n$
and let $P(\sigma)=\{I_1,\ldots,I_m\}$. Let $y \in \R^n$ be such that
each subvector $y_{I_i}$, $i\in\N_m$, has distinct
coordinates. Then, the existence of a constant $\alpha\in\RR$ such
that
\begin{equation}
  \label{eqn-29Jun07-6}
  \langle x, \sigma' y \rangle = \alpha \,\, \mbox{ for all } \sigma' \succsim \sigma,
\end{equation}
is equivalent to the fact that $x=0$ (and thus $\alpha=0$).
\end{corollary}

\begin{proof} The sufficiency part is obvious, so we need only prove the necessity.
We prove the claim by induction on $m$. If $m=1$ then $x
\in\Delta(\sigma)^{\perp}$ is equivalent to $x_1+\cdots+x_n=0$. This
together with (\ref{eqn-29Jun07-6}) implies that the
extended vector $\bar x := (x, -\alpha)$ is a solution to the linear
system $Y \bar x = 0$, where $Y$ is defined above.  By Lemma~\ref{lem-29Jul07-2}, $Y$ has full
column rank, which implies that $x=0$ and $\alpha=0$. Suppose now
that the result is true for $m-1$, we prove it for $m$. For each
$\sigma' \succsim \sigma$ we have the natural disjoint decomposition
$\sigma' = \sigma'_1 \circ \cdots \circ \sigma'_m$, where each permutation
$\sigma'_j \in \Sigma^{|I_j|}$ is the restriction of $\sigma'$ to the set $I_j$,
 $j \in \N_m$. Thus,
$$
\langle x, \sigma' y\rangle = \langle x_{I_1}, \sigma'_1 y_{I_1}\rangle +
\cdots + \langle x_{I_m}, \sigma'_m y_{I_m}\rangle.
$$
Fix a permutation
$\sigma_1' \in \Sigma^{|I_1|}$. Since
$$
\langle x_{I_2}, \sigma'_2 y_{I_2}\rangle +
\cdots + \langle x_{I_m}, \sigma'_m y_{I_m}\rangle
= \alpha - \langle x_{I_1}, \sigma'_1 y_{I_1}\rangle
$$
for any $\sigma'_j \in \Sigma^{|I_j|}$, $j=2,\ldots,m$, we conclude by
the induction hypothesis that $x_{I_2}=\cdots = x_{I_m}=0$ and that
$\alpha - \langle x_{I_1}, \sigma'_1 y_{I_1}\rangle=0$. But the
permutation $\sigma'_1$ was arbitrary, so we obtain
$$
\langle x_{I_1}, \sigma'_1 y_{I_1}\rangle = \alpha \,\, \mbox{ for all }
\sigma'_1 \in \Sigma^{|I_1|}.
$$
This, by the considerations in the base case of the induction, shows
that $x_{I_1}=0$ and $\alpha=0$.
\end{proof}

%***********************************************************************
                         %END TEXT
%***********************************************************************

%********************************************************************
\subsection*{Acknowledgment}
%********************************************************************
\noindent The authors wish to thank Vestislav Apostolov (UQAM,
Montreal, Canada), Vincent Beck (ENS Cachan, France), Matthieu
Gendulphe (University of Fribourg, Switzerland), and Joaquim Ro\'e
(UAB, Barcelona, Spain) for interesting and useful discussions on
early stages of this work. We especially thank  Adrian Lewis
(Cornell University, Ithaca, USA) for useful discussions, and in
particular for pointing out Proposition~\ref{Trans-alg}, as well as a shorter proof of Lemma~\ref{lem-29Jul07-2}.

%\bibliographystyle{plain}
%\bibliography{master}

\begin{center}
--------------------------------------------------
\end{center}

\bigskip

\noindent Aris DANIILIDIS\smallskip

\noindent Departament de Matem\`{a}tiques, C1/308\newline
Universitat Aut\`{o}noma de Barcelona\newline E-08193 Bellaterra
(Cerdanyola del Vall\`{e}s), Spain.\medskip

\noindent E-mail: \thinspace\texttt{arisd@mat.uab.es}\\
\texttt{http://mat.uab.es/\symbol{126}arisd}\smallskip

\noindent Research supported by the MEC Grant
No.~MTM2008-06695-C03-03 (Spain).

\vspace{0.8cm}

\noindent J\'{e}r\^{o}me MALICK\smallskip

\noindent CNRS, Laboratoire J. Kunztmann\\
Grenoble, France\medskip

\noindent E-mail: \texttt{jerome.malick@inria.fr}\newline
\texttt{http://bipop.inrialpes.fr/people/malick/}

\vspace{0.8cm}

\noindent Hristo SENDOV \smallskip

\noindent Department of Statistical \& Actuarial Sciences\\
The University of Western Ontario, London, Ontario, Canada
\medskip

\noindent E-mail: \texttt{hssendov@stats.uwo.ca}\newline
\texttt{http://www.stats.uwo.ca/faculty/hssendov/Main.html} \smallskip

\noindent Research supported by the NSERC of Canada.

\end{document}

%% file: strat.pdf_t
\begin{picture}(0,0)%
\includegraphics{strat.pdf}%
\end{picture}%
\setlength{\unitlength}{3947sp}%
\begingroup\makeatletter\ifx\SetFigFont\undefined%
\gdef\SetFigFont#1#2#3#4#5{%
  \reset@font\fontsize{#1}{#2pt}%
  \fontfamily{#3}\fontseries{#4}\fontshape{#5}%
  \selectfont}%
\fi\endgroup%
\begin{picture}(4605,3844)(2986,-5159)
\put(7126,-2386){\makebox(0,0)[lb]{\smash{{\SetFigFont{11}{13.2}{\familydefault}{\mddefault}{\updefault}{\color[rgb]{0,0,0}$\Delta((123))$}%
}}}}
\put(3001,-2161){\makebox(0,0)[lb]{\smash{{\SetFigFont{11}{13.2}{\familydefault}{\mddefault}{\updefault}{\color[rgb]{0,0,0}$\Delta((13))$}%
}}}}
\put(4726,-1486){\makebox(0,0)[lb]{\smash{{\SetFigFont{11}{13.2}{\familydefault}{\mddefault}{\updefault}{\color[rgb]{0,0,0}$x_3$}%
}}}}
\put(3526,-5086){\makebox(0,0)[lb]{\smash{{\SetFigFont{11}{13.2}{\familydefault}{\mddefault}{\updefault}{\color[rgb]{0,0,0}$x_1$}%
}}}}
\put(7576,-3886){\makebox(0,0)[lb]{\smash{{\SetFigFont{11}{13.2}{\familydefault}{\mddefault}{\updefault}{\color[rgb]{0,0,0}$x_2$}%
}}}}
\put(6601,-4936){\makebox(0,0)[lb]{\smash{{\SetFigFont{11}{13.2}{\familydefault}{\mddefault}{\updefault}{\color[rgb]{0,0,0}$\Delta((12))$}%
}}}}
\put(3076,-4111){\makebox(0,0)[lb]{\smash{{\SetFigFont{11}{13.2}{\familydefault}{\mddefault}{\updefault}{\color[rgb]{0,0,0}$\Delta$(id)}%
}}}}
\end{picture}%